%% file: main.tex
\documentclass[11pt]{article}
\usepackage[english]{babel}
\usepackage[utf8]{inputenc}
\usepackage{amsmath, amssymb}
\usepackage{amsthm}
\usepackage[margin=1in]{geometry}
\newtheorem*{remark}{Remark}
\usepackage{graphicx}
\usepackage{tikz}
\usepackage{caption}
\usepackage{subcaption}
\usepackage{mathtools}
\usepackage{microtype}
\usepackage{temp}

\definecolor{ga_blue}{HTML}{4C72B0}
\definecolor{ga_orange}{HTML}{DD8452}
\definecolor{ga_green}{HTML}{55A868}
\definecolor{ga_purple}{HTML}{8172B2}
\definecolor{ga_gray}{HTML}{F7F7F7}

\definecolor{authorred}{HTML}{D62728}

\input{notation.tex}

\newtheorem{theorem}{Theorem}[section]
\newtheorem{thm}{Theorem}[section]
\newtheorem{prop}[thm]{Proposition}
\newtheorem{lemma}[thm]{Lemma}

\usepackage{todonotes} 

\title{Adaptive Consistency for Mimetic Finite Differences}

\author{Yoo Jin Cha$^{1}$, Omar Duran$^{2}$, Nicola Castelletto$^{3}$, Victor A. P. Magri$^{3}$, Hamdi A. Tchelepi$^{2}$ \\
        \small $^{1}$Institute for Computational and Mathematical Engineering, Stanford University, USA \\
        \small $^{2}$Department of Energy Science and Engineering, Stanford University, USA \\
        \small $^{3}$Lawrence Livermore National Laboratory, USA
}
\date{}

\begin{document}

\sloppy

\include{introduction}
\include{theory}
\include{application}
\include{conclusion}

\bibliographystyle{apacite}
\bibliography{bib}

\end{document}

%% file: notation.tex
\newcommand{\mat}[1]{\mathbf{#1}}
\newcommand{\vecsym}[1]{\boldsymbol{#1}}

\newcommand{\domain}{\Omega}
\newcommand{\bndry}{\partial}
\newcommand{\spacedim}{d}
\newcommand{\cell}{\mathcal{C}}
\newcommand{\fc}{\mathcal{F}}
\newcommand{\HdivSpace}[1]{H(\text{div}; #1)}
\newcommand{\LtwoSpace}[1]{L^2(#1)}

\newcommand{\press}{p}
\newcommand{\flux}{\mathbf{m}}
\newcommand{\testflux}{\mathbf{v}}
\newcommand{\testpress}{q}
\newcommand{\perm}{\mathbf{K}}
\newcommand{\mob}{\Lambda}
\newcommand{\normal}{\mathbf{n}}
\newcommand{\fcarea}{A_{\fc}}
\newcommand{\diffK}{\vecsym{\mathcal{K}}} 

\newcommand{\pdrop}{\Delta \mathbf{P}}

\newcommand{\fcnormRes}{\mathbf{R}}
\newcommand{\localRes}{\mathbf{r}}
\newcommand{\localnormRes}{\mathbf{R}}

\newcommand{\Mmat}{M}
\newcommand{\Bmat}{B}
\newcommand{\dvec}{\mathbf{d}}

\newcommand{\Cmat}{\mathbf{C}}
\newcommand{\Nmat}{\mathbf{N}}

\newcommand{\tol}{\tau}

\newcommand{\Fset}{\mathbb{F}}
\newcommand{\Tset}{\mathbb{T}}
\newcommand{\cellind}{\eta}

\newcommand{\normtwo}[1]{\Vert #1 \Vert_2}
\newcommand{\normtwosq}[1]{\Vert #1 \Vert_2^2}
\newcommand{\norminf}[1]{\Vert #1 \Vert_{\infty}}

\newcommand{\triple}[2][]{\left\vert\!\left\vert\!\left\vert #2 \right\vert\!\right\vert\!\right\vert_{#1}}
\newcommand{\tripleval}[2][]{\sqrt{#2^T #1 #2}}

\newcommand{\pProj}{\press^{\text{proj}}}
\newcommand{\mProj}{\flux^{\text{proj}}}
\newcommand{\Mtpfa}{\Mmat^{\text{TPFA}}}
\newcommand{\Mmfd}{\Mmat^{\text{MFD}}}
\newcommand{\Madapt}{\Mmat^{\tol}}

\newcommand{\mSol}{\flux_h}
\newcommand{\mMfd}{\mSol^{\text{MFD}}}
\newcommand{\mAdapt}{\mSol^{\tol}}

\newcommand{\lin}{\mathrm{lin}} 
\newcommand{\exact}{\mathrm{exact}} 

\newcommand{\Cindep}{\mathfrak{C}}

%% file: introduction.tex
\maketitle

\begin{abstract}
The mimetic finite difference (MFD) method provides a robust discretization for flow simulation on general polyhedral meshes, but its computational cost can become significant due to dense local operators and reduced global sparsity. While the two-point flux approximation (TPFA) offers a substantially cheaper alternative, its accuracy is generally restricted to $K$-orthogonal grids. To balance these competing considerations, we present an adaptive MFD framework based on a residual-based consistency indicator derived from the discrete constitutive equations. The indicator measures local inconsistency and enables adaptive TPFA/MFD stencil selection through a user-prescribed tolerance $\tau$. Because the adaptation is performed within a mimetic framework, arbitrary TPFA/MFD partitions remain stable and structure preserving. Theoretical analysis establishes uniform coercivity and proves explicit tolerance-controlled convergence of the relative flux error. Numerical experiments on challenging polyhedral reservoir benchmarks demonstrate accuracy comparable to full MFD discretizations while substantially reducing matrix density and computational cost.

\end{abstract}

\noindent\keywords{Reservoir Simulation; Mimetic Finite Differences; $K$-Orthogonality; Adaptive Discretization; Error Bounds; Unstructured Grid}

\section*{Highlights} 
\begin{itemize}
    \item Replaces heuristic geometric metrics with a systematic consistency indicator derived directly from the discrete constitutive equations using projected linear fields.
    \item Establishes an explicit, user-controlled upper bound on the relative flux error, proving rigorous convergence scaled directly by the parameter $\tau$.
    \item Improves global operator sparsity and reduces matrix density compare to a full-MFD discretizations at a computational cost compared with a single residual evaluation.
    \item Validates performance across diverse reservoir benchmarks, demonstrating that geometric complexity exerts a primary influence on stencil selection compared to permeability heterogeneity.
\end{itemize}

\section{Introduction}
Accurate prediction of fluid displacement is a central objective in reservoir simulation, as it directly impacts applications such as CO$_2$ sequestration, enhanced oil recovery, and groundwater remediation. In CO$_2$ sequestration, for example, reliable tracking of how the injected fluid displaces the resident phase is essential, because it controls plume migration, trapping mechanisms, and long-term storage security~\cite{benson2005ccs, Juanes2006}. Numerically, this behavior is governed by the fluxes computed across cell interfaces, which determine inter-cell mass transfer and ultimately the evolution of saturation fronts.\\

\begin{figure}[H]
\centering
\begin{subfigure}[b]{0.28\textwidth}
    \centering
    \includegraphics[width=\linewidth]{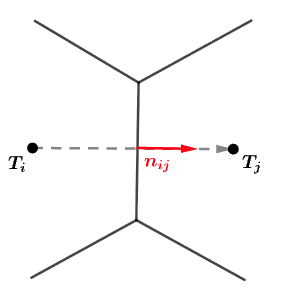}
    \caption{$K$-orthogonal grid}
\end{subfigure}
\hspace{0.1\textwidth}
\begin{subfigure}[b]{0.35\textwidth}
    \centering
    \includegraphics[width=\linewidth]{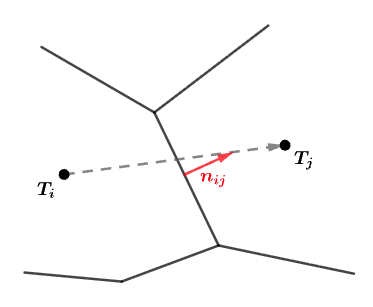}
    \caption{Non-$K$-orthogonal grid}
\end{subfigure}
\caption{
\small 
Illustration of $K$-orthogonality.
(a) For cells $T_i$ and $T_j$, the face normal $\mathbf{n}_{ij}$ is aligned with the center-to-center direction, ensuring consistency of TPFA.
(b) Misalignment leads to inconsistency.
}
\label{fig:korth}
\end{figure}

In practice, these fluxes are approximated using spatial discretization schemes, most commonly two-point or multi-point flux approximations. The two-point flux approximation (TPFA) computes fluxes from pressure differences between neighboring cells. It is simple and widely used, but it is consistent only on $K$-orthogonal grids, where the grid geometry is aligned with the permeability tensor (see Figure~\ref{fig:korth})~\cite{Aziz,Lie2019}. On more general grids, this alignment is lost, and TPFA can introduce significant errors. Multi-point flux approximation (MPFA) methods restore consistency on non-$K$-orthogonal grids by incorporating additional neighboring information~\cite{aavatsmark,Wenjuan}. This improved robustness comes at the cost of larger stencils and more tightly coupled local systems, which can complicate implementation and increase computational cost. Additionally, MPFA is typically implemented as a single-field discretization, where the system is expressed entirely in terms of cell pressures, this is done typically by constructing stencils that enforce local flux continuity without necessarily enforcing the global adjoint relationship between the pressure-gradient and divergence operators.\\

Mimetic finite difference (MFD) methods provide a robust alternative, traditionally formulated as a mixed two-field discretization solving simultaneously for velocity (vector field) and pressure (scalar field). Like MPFA, MFD preserves key mathematical properties of the continuous problem, including local mass conservation and consistency, but it does so while maintaining a symmetric, uniformly coercive global system and preserving the global adjoint relationship between the pressure-gradient and divergence operators~\cite{BeirodaVeiga2009}. Furthermore, MFD was natively devised to support general polyhedral meshes, making it well suited for the highly heterogeneous and geometrically complex grids encountered in reservoir simulation~\cite{daVeiga2014}. Crucially, the underlying algebraic framework of MFD naturally accommodates TPFA-type discretizations through appropriate choices of the local inner-product operator. This flexibility allows different inner products to be deployed in different regions of the domain while retaining a common mimetic structure. However, when the full MFD inner-product operator is employed throughout the entire domain, the resulting system becomes significantly more expensive to assemble and solve due to denser local operators and reduced global sparsity.\\

These trade-offs motivate the use of heterogeneous discretizations across different regions of the domain. In zones that closely satisfy $K$-orthogonality, TPFA remains sufficiently accurate, whereas the full MFD operator is required only in highly distorted regions. Adaptive approaches leveraging this partitioning strategy have been proposed~\cite{Dong2022}; however, they often rely on manual region selection and lack a systematic, quantitative criterion for identifying where the full MFD stencil is strictly necessary. Alternatively, purely geometric mesh-quality metrics, such as distortion or skewness coefficients, have been used to guide stencil selection~\cite{Lowrie2011}. While these metrics characterize grid geometry, their relationship to the actual discretization error is indirect and problem-dependent. Consequently, they provide only qualitative guidance and can fail to pinpoint where physical flow features demand a higher-order stencil. \\

Against this background, we adopt a different perspective: rather than relying on static geometric indicators or manual mapping, we employ a dynamic, local consistency diagnostic driven by the physics of the problem. Our strategy builds upon the classical patch test for consistency~\cite{Militello1991} and modern discretization frameworks~\cite{beiraodaveiga2017divergence, alhinai2017generalized}. Instead of utilizing linear-field verification solely as a static verification check, we embed it directly into the discretization workflow to drive adaptive stencil selection. While this shares a broad motivation with the quadrature-reduction techniques in Multi-Point Flux Mixed Finite Element (MFMFE) methods~\cite{ingram2010multipoint}, where stable mixed operators condense to lower-cost finite-volume stencils, our approach differs fundamentally. Whereas MFMFE reductions are structurally bound to specific cell geometries a priori, our framework utilizes an indicator mechanism governed by a user-defined tolerance $\tau$ to partition the domain. By embedding this indicator loop natively within a face-centered mimetic complex~\cite{bochev2006principles}, the local transition between TPFA and MFD operators leaves the underlying algebraic and topological structure unaltered. Consequently, global uniform coercivity, strict mass conservation, and continuous face-normal fluxes are naturally guaranteed across arbitrary configurations of stencil partitions.

Building on the consistency-based adaptation framework outlined above, our objective is to develop a practical adaptive MFD framework that provides explicit tolerance-controlled bounds on the relative flux error while reducing the computational overhead relative to a full MFD discretization on complex polyhedral meshes. To achieve this, we construct a residual-based consistency indicator derived directly from the discrete constitutive equations. This indicator measures local inconsistency and is compared against a prescribed tolerance to automatically guide TPFA/MFD selection. Importantly, the cost of the adaptation remains comparable to a single global residual evaluation, requiring neither additional linear solves nor expensive local reconstructions. We emphasize that the goal of this work is not to derive sharp a posteriori error estimators, which typically require more sophisticated and computationally expensive dual or goal-oriented techniques. Rather, our focus is on developing a practical adaptation mechanism that provides explicit control of the flux error while maintaining the efficiency and robustness of the underlying mimetic framework. \\

The remainder of this paper is structured as follows. Section 2 outlines the continuous governing equations and establishes the local mixed discrete structural framework. Section 3 details the physical motivation for the proposed adaptation strategy, develops the residual-based consistency indicator through projected linear fields, and introduces the Local Adaptation (LA) and Global Adaptation (GA) indicator formulations. The section further provides a comprehensive convergence analysis proving tolerance-controlled error bounds for both approaches. Section 4 demonstrates the framework's behavior across multiple complex reservoir benchmarks, highlighting its matrix sparsity and error accumulation properties. Section 5 summarizes our findings and suggests avenues for future development. 

\captionsetup{font=small,justification=raggedright,singlelinecheck=false}
\setlength{\emergencystretch}{3em}
\hyphenpenalty=500
\exhyphenpenalty=500

%% file: theory.tex
\section{Problem Statement}

We consider single-phase incompressible flow in a porous medium over a computational domain $\domain \subset \mathbb{R}^\spacedim$ with boundary $\bndry \domain$. Let $\flux \in \HdivSpace{\domain}$ and $\press \in \LtwoSpace{\domain}$ denote, respectively, the Darcy flux and pressure. We write the governing equations in strong form as
\begin{align*}
\diffK^{-1}\flux + \nabla \press &= 0 \quad \text{in } \domain, \\
\nabla\cdot \flux &= 0 \quad \text{in } \domain,
\end{align*}
with suitable boundary conditions on $\bndry \domain$ (e.g., prescribed pressure and/or normal flux). Note that diffusive tensor $\diffK \coloneqq \mob \perm$ incorporates the mass mobility  $\mob$ and the absolute permeability $\perm$.

\paragraph{Global weak form.}
Multiplying by test functions $\testflux \in \HdivSpace{\domain}$ and $\testpress \in \LtwoSpace{\domain}$ and integrating by parts yields the standard mixed variational formulation: find $(\flux,\press)\in \HdivSpace{\domain}\times\LtwoSpace{\domain}$ such that
\begin{equation}
\begin{aligned}
 a(\flux,\testflux) - b(\testflux,\press) + \ell(\testflux) &= 0, \quad &&\forall\, \testflux \in \HdivSpace{\domain},\\
 b(\flux,\testpress) &= 0, \quad &&\forall\, \testpress \in \LtwoSpace{\domain},
\end{aligned}
\label{eq:global_mixed_form}
\end{equation}
where the global bilinear and linear forms are defined by
\begin{align*}
 a(\flux,\testflux) &\coloneqq \int_{\domain} \diffK^{-1}\flux\cdot\testflux\,dV, \\
 b(\testflux,\press) &\coloneqq \int_{\domain} \press\,(\nabla\cdot\testflux)\,dV, \\
 \ell(\testflux) &\coloneqq \int_{\bndry \domain} \press\,(\testflux\cdot\normal)\,dS.
\end{align*}

\paragraph{Localization to cells.}
Because the integrals in \eqref{eq:global_mixed_form} are additive over a mesh partition $\Tset$ of $\domain$, the global forms decompose into sums of cell contributions, e.g., $a=\sum_{\cell\in\Tset} a_{\cell}$ and $b=\sum_{\cell\in\Tset} b_{\cell}$. We therefore work on an individual cell $\cell\subset\domain$ and use the corresponding restricted (cell-wise) forms $a_{\cell}$, $b_{\cell}$, and $\ell_{\cell}$ inherited from the global definitions.

On each individual cell $\cell \subset \domain$, we map the continuous functions to a scalar cell pressure unknown $\press_\cell \in \mathbb{R}$ and a localized vector of face-based flux unknowns $\flux_\cell \in \mathbb{R}^{|(\cell)_\fc|}$, where $|(\cell)_\fc|$ represents the number of faces of cell $\cell$. Approximating the weak operators in discrete space~\cite{BeirodaVeiga2009}, the forms are replaced by matrix counterparts $a_\cell(\flux, \testflux) \approx \flux_\cell^T \Mmat_\cell \testflux_\cell$, $b_\cell(\testflux, \press) \approx \testflux_\cell^T \Bmat_\cell \press_\cell$, and $\ell_\cell(\testflux) \approx \testflux_\cell^T \dvec_\cell$, yielding the local saddle-point system matrix row:
\begin{align*}
\begin{bmatrix}
\Mmat_\cell & -\Bmat_\cell^T \\
\Bmat_\cell & 0
\end{bmatrix}
\begin{bmatrix}
\flux_\cell \\
 \press_\cell
\end{bmatrix}
=
\begin{bmatrix}
-\dvec_\cell \\
0
\end{bmatrix}.
\end{align*}
The specific form of the symmetric positive definite inner product matrix $\Mmat_\cell$ is chosen depending on the active cell discretization, varying between the classic two-point flux approximation $\Mtpfa_\cell$ and the mimetic finite difference operator $\Mmfd_\cell$.

\section{Residual-Based Constitutive Indicators}

To identify where the low-cost $\Mtpfa_\cell$ operator remains valid, this section develops a residual-based constitutive indicator derived from the discrete constitutive equations. We first introduce the indicator and its corresponding adaptive TPFA/MFD discretization, and then establish the theoretical properties of the resulting scheme. In particular, we prove that the adaptive discretization remains stable under arbitrary stencil partition and that the relative flux error is explicitly controlled by the prescribed tolerance $\tau$.

\subsection{Consistency Conditions}
We begin by defining the local matrix invariants required to isolate genuine structural errors from a nominal flow profile. Mimetic discretizations require discrete operators to exactly reproduce linear pressure fields $\press(\mathbf{x}) = \mathbf{a} \cdot \mathbf{x} + b$ inducing uniform flux profiles $\flux = -\diffK \mathbf{a}$. Let $\mathbf{c}_{ij}$ denote the center-to-face vector stretching from the cell centroid to the face centroid, and $\mathbf{n}_{ij}$ the outward unit face normal field for cell $\cell$. Exact normal interface fluxes and localized cell pressure drops project into the discrete space via the structural mapping matrices $\Nmat_{\cell}$ and $\Cmat_{\cell}$:
\begin{align*}
\flux_{\cell} = \Nmat_{\cell} \diffK_{\cell} \mathbf{a}, \quad \Bmat_{\cell} \press_{\cell} - \dvec_{\cell} = \Cmat_{\cell} \mathbf{a}.
\end{align*}
Constitutive algebraic closure under a consistent scheme requires that the face-based inner product matrix satisfies the primary consistency identity:
\begin{align}
\Mmat_{\cell} \Nmat_{\cell} \diffK_{\cell} = \Cmat_{\cell}.
\label{consistency}
\end{align}
Evaluating the identity \eqref{consistency} for the two-point flux approximation yields the localized algebraic consistency defect metric:
\begin{align*}
\delta_{\cell} \coloneqq \normtwo{\Mtpfa_{\cell} \Nmat_{\cell} \diffK_{\cell}  - \Cmat_{\cell}}.
\end{align*}
The matrix norm $\delta_{\cell}$ remains a local algebraic quantity with physical units and ignores the residual fluxes which measure the exactness between the discrete pressure and fluxes; this prevents a direct uniform bound on continuous solution errors. This limitation motivates a dimensionless, residual-based indicator mapped directly through the discrete operators.

\subsection{Flux Projection}
To construct such a residual-based indicator, we introduce a local projection that satisfies interface admissibility across strong material discontinuities. Let an interior interface $\fc = \bndry \cell_L \cap \bndry \cell_R$ separate adjacent blocks $\cell_L$ and $\cell_R$. Assume that each cell is internally homogeneous and isotropic, with a sharp permeability discontinuity across the interface:
\begin{align*}
\diffK(\mathbf{x}) = \begin{cases} \diffK_L = k_L \mat{I}, & \mathbf{x} \in \cell_L \\ \diffK_R = k_R \mat{I}, & \mathbf{x} \in \cell_R \end{cases}
\end{align*}
where $k_L \neq k_R$. Let $d_L = \normtwo{\mathbf{x}_\fc - \mathbf{x}_L}$ and $d_R = \normtwo{\mathbf{x}_\fc - \mathbf{x}_R}$ denote the normal half-cell distances from the cell centroids to the face centroid. The strong form of the problem statement requires pressure continuity and normal mass conservation at the trace interface $\fc$:
\begin{align}
\press_L(\mathbf{x}_\fc) &= \press_R(\mathbf{x}_\fc)  &&\text{(Pressure Continuity)}, \label{eq:p_cont} \\
\flux_L(\mathbf{x}_\fc) \cdot \normal_\fc &= \flux_R(\mathbf{x}_\fc) \cdot \normal_\fc  &&\text{(Normal Mass Conservation)}, \label{eq:flux_cont}
\end{align}
where $\normal_\fc$ is the unit normal pointing outward from $\cell_L$ to $\cell_R$.

We define a reference flux projection by evaluating a TPFA configuration across the interface zone. This baseline deliberately assumes perfect geometric alignment between the cell centroids and the face-normal axis, providing a clean metric to isolate stencil inconsistencies caused by grid non-$\diffK$-orthogonality.

\begin{prop}[Admissibility of the Local Orthogonal Flux Projection]\label{prop:admissibility_orthogonal}
Let the local subdomain continuous pressure fields be governed by local gradients $\mathbf{g}_L$ and $\mathbf{g}_R$:
\begin{align*}
\press_L(\mathbf{x}) = \mathbf{g}_L \cdot (\mathbf{x} - \mathbf{x}_\fc) + \pi_\fc, \quad \press_R(\mathbf{x}) = \mathbf{g}_R \cdot (\mathbf{x} - \mathbf{x}_\fc) + \pi_\fc
\end{align*}
where $\pi_\fc$ is the unique interface pressure. Let $\kappa_L \coloneqq (\diffK_L \normal_\fc) \cdot \normal_\fc$ and $\kappa_R \coloneqq (\diffK_R \normal_\fc) \cdot \normal_\fc$ denote the face-normal directional diffusive components of the adjacent cells. Then, the scalar effective face-normal diffusive projection defined by:
\begin{align*}
\kappa_\fc = \frac{d_L + d_R}{\frac{d_L}{\kappa_L} + \frac{d_R}{\kappa_R}}
\end{align*}
evaluated against a globally uniform nominal gradient field $\mathbf{g} = [a, b, c]^T$ yields a discrete normal interface flux:
\begin{align*}
m_\fc^{\text{proj}} = -\fcarea \kappa_\fc (\mathbf{g} \cdot \normal_\fc)
\end{align*}
that satisfies conditions \eqref{eq:p_cont} and \eqref{eq:flux_cont} identically.
\end{prop}

\begin{proof}
By construction, $\press_L(\mathbf{x}_\fc) = \press_R(\mathbf{x}_\fc) = \pi_\fc$, satisfying condition \eqref{eq:p_cont} identically. Evaluating the local pressure fields at the respective cell centroids yields:
\begin{align*}
p_L \coloneqq \press_L(\mathbf{x}_L) = -d_L (\mathbf{g}_L \cdot \normal_\fc) + \pi_\fc \implies \pi_\fc - p_L = d_L (\mathbf{g}_L \cdot \normal_\fc), \\
p_R \coloneqq \press_R(\mathbf{x}_R) = d_R (\mathbf{g}_R \cdot \normal_\fc) + \pi_\fc \implies p_R - \pi_\fc = d_R (\mathbf{g}_R \cdot \normal_\fc).
\end{align*}
Applying the continuous Darcy's law along the face-normal directional component, the normal velocities at each side of the interface manifest as:
\begin{align*}
(\flux_L \cdot \normal_\fc) = -(\diffK_L \mathbf{g}_L) \cdot \normal_\fc = -\kappa_L (\mathbf{g}_L \cdot \normal_\fc) = -\frac{\kappa_L}{d_L}(\pi_\fc - p_L), \\
(\flux_R \cdot \normal_\fc) = -(\diffK_R \mathbf{g}_R) \cdot \normal_\fc = -\kappa_R (\mathbf{g}_R \cdot \normal_\fc) = -\frac{\kappa_R}{d_R}(p_R - \pi_\fc).
\end{align*}
Imposing the normal mass conservation matching condition \eqref{eq:flux_cont} establishes the interface flux balance:
\begin{align*}
-\frac{\kappa_L}{d_L}(\pi_\fc - p_L) = -\frac{\kappa_R}{d_R}(p_R - \pi_\fc)
\end{align*}
which isolates the unique continuous internal interface trace pressure:
\begin{align*}
\pi_\fc = \left( \frac{\kappa_L}{d_L} + \frac{\kappa_R}{d_R} \right)^{-1} \left( \frac{\kappa_L}{d_L} p_L + \frac{\kappa_R}{d_R} p_R \right).
\end{align*}
Substituting $\pi_\fc$ back into the continuous normal velocity relations defines the macroscopic interface normal flux expression:
\begin{align}
(\flux \cdot \normal_\fc) = -\left( \frac{d_L}{\kappa_L} + \frac{d_R}{\kappa_R} \right)^{-1} (p_R - p_L). \label{eq:interflux_concise}
\end{align}
Since the projection is evaluating using a globally linear field, the local gradients coincide across the interface, i.e., $\mathbf{g}_L = \mathbf{g}_R = \mathbf{g}$. Under this globally linear field, the pressure difference between the two cell centroids satisfies $p_R - p_L = (d_L + d_R)(\mathbf{g} \cdot \normal_\fc)$. Substitution into \eqref{eq:interflux_concise} scales the interface velocity field:
\begin{align*}
(\flux \cdot \normal_\fc) = -\left( \frac{d_L + d_R}{\frac{d_L}{\kappa_L} + \frac{d_R}{\kappa_R}} \right) (\mathbf{g} \cdot \normal_\fc) \equiv -\kappa_\fc (\mathbf{g} \cdot \normal_\fc).
\end{align*}
Integrating this directional velocity profile over the fixed face area $\fcarea$ yields $m_\fc^{\text{proj}} = \int_\fc (\flux \cdot \normal_\fc)\,ds = -\fcarea \kappa_\fc (\mathbf{g} \cdot \normal_\fc)$, completing the proof.
\end{proof}

With this conservative harmonic projection field established, we present the algorithmic implementation details that transform these continuous admissibility definitions into discrete indicator mappings.

\subsection{Residual-Based Indicators and Operator Assembly}
The construction of the adaptation indicators entails four simple steps:\\

\noindent\textbf{Step 1: Projection of Admissible Flow Fields} \\
Prescribe a globally uniform nominal gradient field $\mathbf{g}$, inducing a linear pressure field and the corresponding continuous velocity field. The cell-center pressures and corresponding face mass fluxes are given by:
\begin{align*}
\pProj_{\cell} = \press_{\lin}(\mathbf{x}_{\cell}), \quad m_\fc^{\text{proj}} = -\fcarea \kappa_\fc (\mathbf{g} \cdot \normal_\fc).
\end{align*}

\noindent\textbf{Step 2: Local Subdomain Evaluation and Normalization} \\
Evaluate the projected pairs against the local unassembled TPFA operator equations:
\begin{align*}
\localRes_{\cell} \coloneqq \Mtpfa_{\cell} \mProj_{\cell} - \Bmat_{\cell}\pProj_{\cell} + \dvec_{\cell}.
\end{align*}
Isolating the localized pressure drop profile vector $\pdrop_{\cell} \coloneqq - \Bmat_{\cell} \pProj_{\cell} + \dvec_{\cell}$, the dimensionless element-wise normalized residual array maps to:
\begin{align*}
\localnormRes_{\cell} \coloneqq \frac{\localRes_{\cell}}{\normtwo{\pdrop_{\cell}}}.
\end{align*}

\noindent\textbf{Step 3: Construction of Local and Global Adaptation Indicators} \\
At this stage, two adaptation paradigms are defined according to whether the normalized residual information is evaluated locally of after global assembly across neighboring interfaces. \\

\noindent \textbf{Local Adaptation (LA).}
The local adaptation paradigm evaluates stencil admissibility directly from the localized normalized residual vector obtained in Step 2. The corresponding cell-wise indicator is defined as 
\[
\eta_\mathcal{C}^{\text{LA}} 
\coloneqq \| \mathbf{R}_\mathcal{C} \|_\infty
= 
\max 
|(\mathbf{R}_\mathcal{C})_\fc|.
\]
\noindent \textbf{Global Adaptation (GA).} Alternatively, neighboring residual contributions may first be assembled across shared interfaces. The resulting face-normalized residual field is defined as
\begin{align*}
\fcnormRes_\Fset = \sum_{\cell \in \Tset} (\localnormRes_{\cell})_\fc.
\end{align*}
The corresponding cell-wise indicator is then given by
\begin{align*}
\cellind_{\cell}^{\text{GA}} \coloneqq \max |(\fcnormRes_\Fset)_\cell|.
\end{align*}
For notational simplicity, we denote either adaptation metric by $\eta_\mathcal{C}$ whenever the distinction between Local and Global Adaptation is not required.\\

\noindent\textbf{Step 4: Apply Thresholding Criteria} \\
Given a user-prescribed tolerance $\tol > 0$, we introduce a binary stencil activation flag $\chi_{\cell} \in \{0, 1\}$ to classify each cell domain:
\begin{align}
\chi_{\cell} \coloneqq \begin{cases} 0, & \text{if } \cellind_{\cell} \le \tol \quad \text{(TPFA-compatible)}, \\ 1, & \text{if } \cellind_{\cell} > \tol \quad \text{(MFD-compatible)}. \end{cases} \label{eq:cell_marking}
\end{align}

\noindent\textbf{Adapted Operator Assembly} \\
Once the classification flag $\chi_{\cell}$ is assigned to each cell according to equation~\eqref{eq:cell_marking}, the local cell-wise inner product blocks $\Mmat_{\cell}$ are determined via:
\begin{align*}
\Mmat_{\cell} = (1 - \chi_{\cell})\Mtpfa_{\cell} + \chi_{\cell}\Mmfd_{\cell}.
\end{align*}
These local blocks are collectively assembled to obtain the global adapted inner product matrix block $\Madapt$. The impact of the stencil activation process on the resulting matrix structure is illustrated in Figure~\ref{fig:adaptive_operator_structure}, which compares the sparsity patterns of full TPFA, adaptive MFD, and full MFD discretizations.
\begin{figure}[H]
\centering

\begin{subfigure}[t]{0.3\textwidth}
    \centering
    \includegraphics[width=\textwidth]{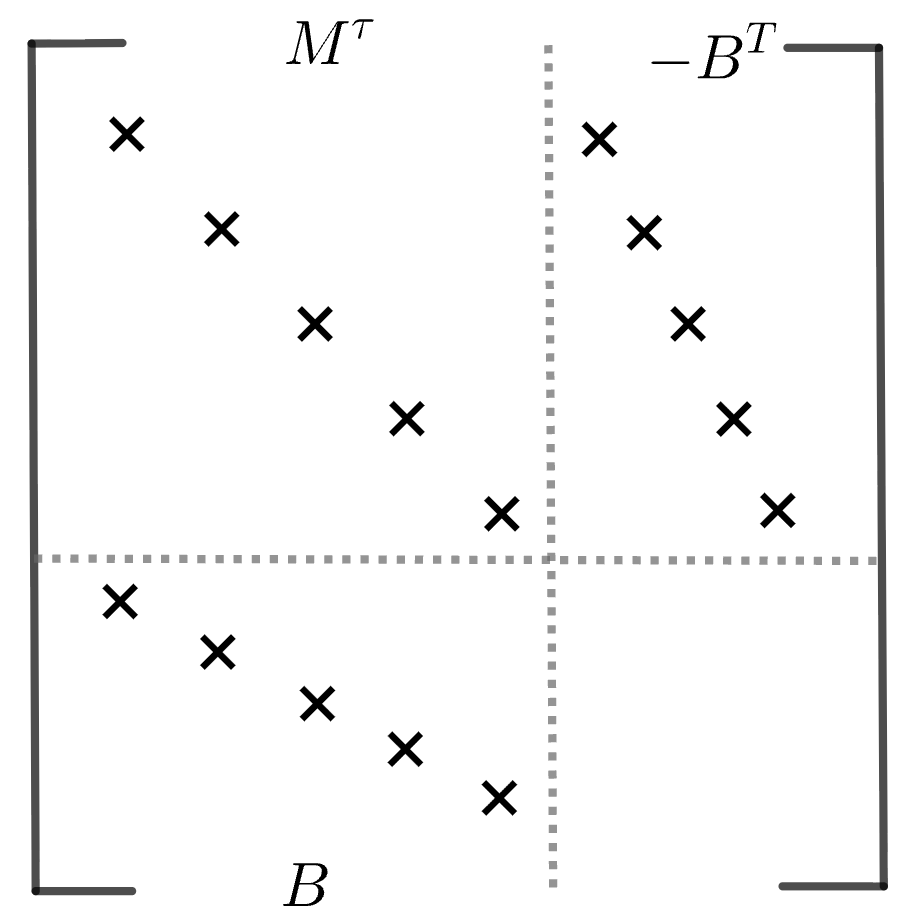}
    \caption{Full TPFA $(\tol \rightarrow \infty)$}
\end{subfigure}
\hfill
\begin{subfigure}[t]{0.3\textwidth}
    \centering
    \includegraphics[width=\textwidth]{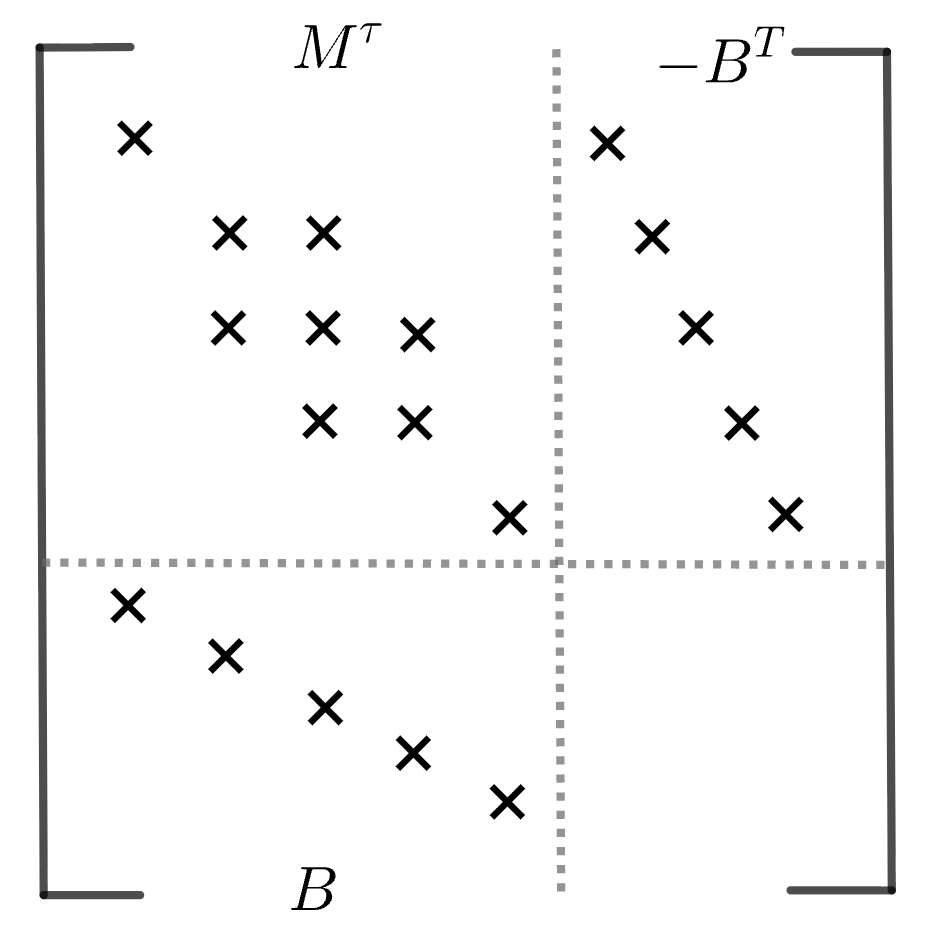}
    \caption{Adaptive MFD}
\end{subfigure}
\hfill
\begin{subfigure}[t]{0.3\textwidth}
    \centering
    \includegraphics[width=\textwidth]{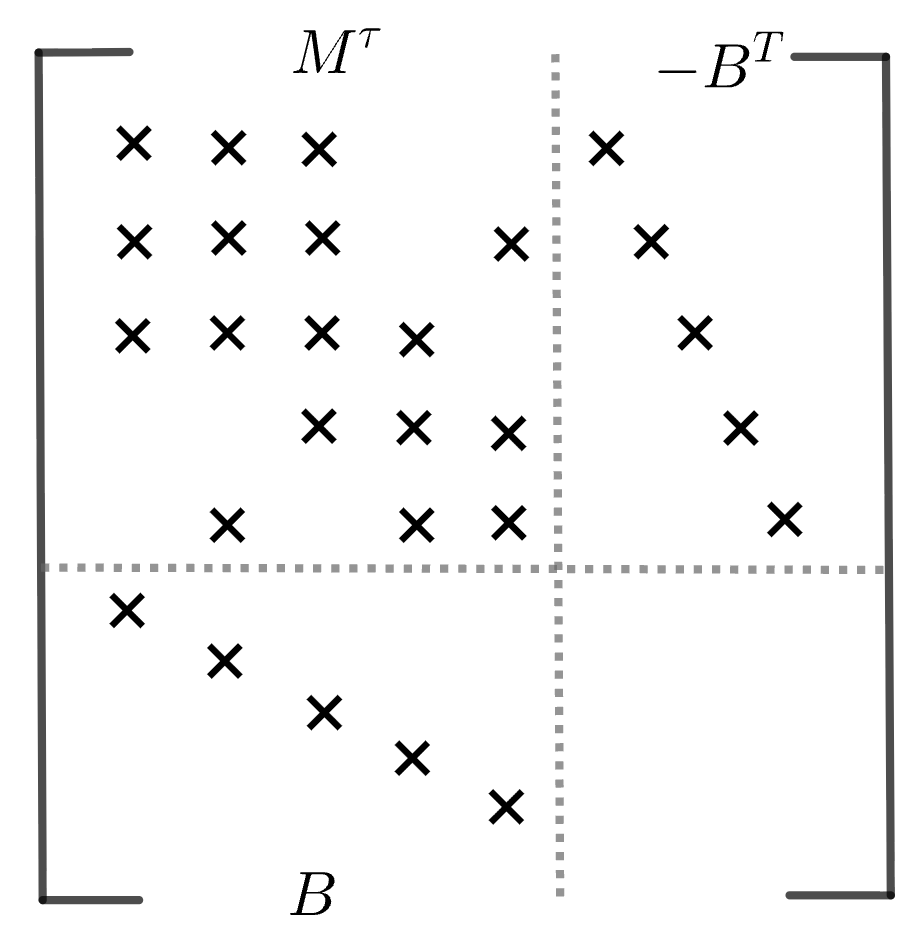}
    \caption{Full MFD $(\tol = 0)$}
\end{subfigure}

\caption{
Sparsity patterns of the global saddle-point operator for full TPFA, adaptive MFD, and full MFD dicretization. As additional cells are classified as MFD-compatible, the density of the global adapted inner product block $M^\tau$ increases while the divergence coupling structure remains unchanged.
}
\label{fig:adaptive_operator_structure}
\end{figure}
Before investigating the accuracy of the adaptive discretization scheme, we must verify that arbitrarily altering local stencils preserves the fundamental mathematical properties of the global matrix.

\begin{lemma}[Uniform Coercivity]\label{lem:coercivity}
Let the cell partition $\Tset$ be restricted to star-shaped, geometrically admissible polyhedral elements where positive two-point transmissibility weights are structurally guaranteed. Let the local cell-wise flux inner product matrices satisfy uniform positive-definiteness bounds independently such that for all localized flux test vectors $\testflux_{\cell} \in \mathbb{R}^{|(\cell)_\fc|}$:

\begin{align*}
\alpha_{\min} \testflux_{\cell}^T \testflux_{\cell} \le \testflux_{\cell}^T \Mmfd_{\cell} \testflux_{\cell} \le \alpha_{\max} \testflux_{\cell}^T \testflux_{\cell}, \quad
\beta_{\min} \testflux_{\cell}^T \testflux_{\cell} \le \testflux_{\cell}^T \Mtpfa_{\cell} \testflux_{\cell} \le \beta_{\max} \testflux_{\cell}^T \testflux_{\cell}
\end{align*}
where $\alpha_{\min}, \beta_{\min} > 0$. Then, for any arbitrary cell partition layout determined by the tolerance parameter $\tol$, the globally assembled hybrid inner product matrix block $\Madapt$ preserves spectral stability:
\begin{align*}
\mu_{\min} \testflux^T \testflux \le \testflux^T \Madapt \testflux \le \mu_{\max} \testflux^T \testflux, \quad \forall \testflux \in \mathbb{R}^{|\Fset|}
\end{align*}
where $\mu_{\min} = \min(\alpha_{\min}, \beta_{\min}) > 0$ and $\mu_{\max} = 2 \max(\alpha_{\max}, \beta_{\max})$ are stable structural constants independent of $\tol$.
\end{lemma}

\begin{proof}
Decompose the global bilinear energy form into its aggregated cell sum components based on the active adaptive partition:
\begin{align*}
\testflux^T \Madapt \testflux = \sum_{\cell \in \Tset^\tol_{\text{TPFA}}} \testflux_{\cell}^T \Mtpfa_{\cell} \testflux_{\cell} + \sum_{\cell \in \Tset^\tol_{\text{MFD}}} \testflux_{\cell}^T \Mmfd_{\cell} \testflux_{\cell}.
\end{align*}
Applying the local element-wise lower coercivity limits to each active subdomain partition yields the initial lower bound:
\begin{align*}
\testflux^T \Madapt \testflux \ge \sum_{\cell \in \Tset^\tol_{\text{TPFA}}} \beta_{\min} \testflux_{\cell}^T \testflux_{\cell} + \sum_{\cell \in \Tset^\tol_{\text{MFD}} } \alpha_{\min} \testflux_{\cell}^T \testflux_{\cell}.
\end{align*}
Factoring out the minimum spectral bound allows us to collapse the localized summation over the entire global grid partition $\Tset$:
\begin{align*}
\testflux^T \Madapt \testflux &\ge \min(\alpha_{\min}, \beta_{\min}) \left( \sum_{\cell \in \Tset^\tol_{\text{TPFA}}} \testflux_{\cell}^T \testflux_{\cell} + \sum_{\cell \in \Tset^\tol_{\text{MFD}}} \testflux_{\cell}^T \testflux_{\cell} \right) \\
&= \min(\alpha_{\min}, \beta_{\min}) \sum_{\cell \in \Tset} \testflux_{\cell}^T \testflux_{\cell}.
\end{align*}
Because every interior face is shared by exactly two adjacent cells, the localized unassembled norm accumulation bounds the global Euclidean face norm from below via $\sum_{\cell \in \Tset} \testflux_{\cell}^T \testflux_{\cell} \ge \testflux^T \testflux$. Substituting this topological property satisfies the lower stability limit:
\begin{align*}
\testflux^T \Madapt \testflux \ge \min(\alpha_{\min}, \beta_{\min}) \testflux^T \testflux \implies \mu_{\min} = \min(\alpha_{\min}, \beta_{\min}) > 0.
\end{align*}
The upper spectral bound follows analogously. Invoking the local upper spectral limits on each unassembled cell component yields:
\begin{align*}
\testflux^T \Madapt \testflux \le \sum_{\cell \in \Tset^\tol_{\text{TPFA}}} \beta_{\max} \testflux_{\cell}^T \testflux_{\cell} + \sum_{\cell \in \Tset^\tol_{\text{MFD}} } \alpha_{\max} \testflux_{\cell}^T \testflux_{\cell}.
\end{align*}
Bounding the aggregated cell sums via the maximum local eigenvalues reduces the expression to:
\begin{align*}
\testflux^T \Madapt \testflux \le \max(\alpha_{\max}, \beta_{\max}) \sum_{\cell \in \Tset} \testflux_{\cell}^T \testflux_{\cell}.
\end{align*}
Invoking the face multiplicity upper bound due to shared interior interfaces yields the inequality $\sum_{\cell \in \Tset} \testflux_{\cell}^T \testflux_{\cell} \le 2 \testflux^T \testflux$. Substituting this upper scaling relationship yields the clean terminal upper bound:
\begin{align*}
\testflux^T \Madapt \testflux \le 2 \max(\alpha_{\max}, \beta_{\max}) \testflux^T \testflux \implies \mu_{\max} = 2 \max(\alpha_{\max}, \beta_{\max}).
\end{align*}
This completes the uniform spectral stability proof.
\end{proof}

\begin{remark}
Lemma~\ref{lem:coercivity} ensures that the globally assembled adaptive system satisfies a uniform discrete inf-sup condition under any adaptation layout. Having proved coercivity and noting that the discrete differential terms remain unchanged regardless of the active adaptation layout, the stability of the global system follows from standard Brezzi theory~\cite{BrezziFortin1991}.
\\
\end{remark}
Having established the stability of the adaptive discretization, we now investigate how the proposed indicators control the resulting relative flux error.

\subsection{Convergence of Local and Global Indicators}

To guarantee that refinement tracking parameters remain uniformly stable across mesh scaling boundaries, we formalize our error analysis within the discrete flux energy norm. For any discrete face-based flux vector $\testflux \in \mathbb{R}^{|\Fset|}$, the global energy norm $\triple[\mat{M}^\tol]{\cdot}$ is defined via the globally assembled, symmetric positive-definite adapted inner product matrix $\Madapt$ as:
\begin{align*}
    \triple[\Madapt]{\testflux} \coloneqq \tripleval[\Madapt]{\testflux} = \left( \sum_{\cell \in \Tset} \testflux_{\cell}^T \Mmat_{\cell} \testflux_{\cell} \right)^{1/2}.
\end{align*}
As an analytical consequence of the uniform coercivity limits established in Lemma~\ref{lem:coercivity}, this flux energy norm is strictly equivalent to the standard Euclidean vector $\ell_2$-norm up to scale-invariant structural constants:
\begin{align*}
    \mu_{\min}^{1/2} \normtwo{\testflux} \le \triple[\Madapt]{\testflux} \le \mu_{\max}^{1/2} \normtwo{\testflux}, \quad \forall \testflux \in \mathbb{R}^{|\Fset|}.
\end{align*}
By formulating the convergence analysis into this flux-centered space, the error indicators track the precise constitutive mismatch of the Darcy velocity field.

\begin{prop}[Conditional Linear Convergence under Local Adaptation]\label{prop:bound_la_conditional}
Let $\flux$ denote the exact linear flux solution vector. Suppose the active adaptive partition contains a locally uniform, face-symmetric subdomain patch where adjacent cell geometries and volumes match identically ($v_1 = v_2 = v$). Under the local cell-wise indicator condition $\cellind_{\cell} = \norminf{\mathbf{R}_\mathcal{C}} \le \tol$, $\forall \cell \in \Tset$, and the uniform stability of the global saddle-point operator:
\begin{align*}
    \mathcal{A}^\tau \coloneqq \begin{bmatrix} \Madapt & -B^T \\ B & 0 \end{bmatrix},
\end{align*}
satisfying $\normtwo{(\mathcal{A}^\tau)^{-1}} \le C_{\text{stab}}$, the relative discretization error in flux satisfies:
\begin{align*}
    \frac{\triple[\Madapt]{\flux - \mAdapt}}{\triple[\Madapt]{\flux}} \le C^\tau_1 \tau,
\end{align*}
where the mesh-dependent scaling constant $C^\tau_1$ and the maximum local subdomain pressure drop $P_{\max}$ are given respectively by:
\begin{align*}
    C^\tau_1 \coloneqq \frac{2 C_{\text{stab}} P_{\max} \mu_{\max}^{1/2} \sqrt{|\Fset^\tau|}}{\triple[\Madapt]{\flux}}, \quad \text{and} \quad P_{\max} \coloneqq \max_{\cell \in \Tset} \normtwo{\pdrop_{\cell}}.
\end{align*}
\end{prop}

\begin{proof}
Let $\localRes^\tol$ be the global perturbation flux residual vector supported strictly on the active faces $\Fset^\tau$, defined via the constitutive operator mismatch on the linear field: $\localRes^\tol = (\Mmfd - \Madapt)\flux$. For an interior interface entry $\fc$ matching adjacent cells $\cell_1$ and $\cell_2$ within the symmetric patch, the row component expands to $|(\localRes^\tol)_\fc| = |(\localRes_{\cell_1}^\tol)_\fc + (\localRes_{\cell_2}^\tol)_\fc|$. Invoking the triangle inequality yields:
\begin{align*}
    |(\localRes^\tol)_\fc| \le |(\localRes_{\cell_1}^\tol)_\fc| + |(\localRes_{\cell_2}^\tol)_\fc|.
\end{align*}
Because the subdomain patch configuration is uniform and face-symmetric, the local directional velocity components decouple, and the corresponding host cell pressure drops scale symmetrically across the interface, forcing $\normtwo{\pdrop_{\cell_1}} = \normtwo{\pdrop_{\cell_2}}$. Bounding each independent subdomain component via the localized tolerance boundary condition $|(\localRes_{\cell}^\tol)_\fc| \le \tau \normtwo{\pdrop_{\cell}}$ and substituting the global maximum drop invariant $P_{\max}$ simplifies the face inequality to :
\begin{align*}
    |(\localRes^\tol)_\fc| \le \tau \left( \normtwo{\pdrop_{\cell_1}} + \normtwo{\pdrop_{\cell_2}} \right) \le 2 P_{\max} \tau.
\end{align*}
Performing quadratic error accumulation over the full active face manifold $\Fset^\tau$ isolates the Euclidean norm boundary of the operator mismatch:
\begin{align*}
    \normtwo{\localRes^\tol}^2 = \sum_{\fc \in \Fset^\tau} |(\localRes^\tol)_\fc|^2 \le \sum_{\fc \in \Fset^\tau} 4 P_{\max}^2 \tau^2 = 4 P_{\max}^2 |\Fset^\tau| \tau^2 \implies \normtwo{\localRes^\tol} \le 2 P_{\max} \sqrt{|\Fset^\tau|} \tau.
\end{align*}
We express the global error configurations as a coupled saddle-point system with a vanishing right-hand side divergence row:
\begin{align*}
    \mathcal{A}^\tau \begin{bmatrix} \flux - \mAdapt \\ p^{\text{MFD}} - p_h^\tau \end{bmatrix} = \begin{bmatrix} -\localRes^\tol \\ 0 \end{bmatrix}.
\end{align*}
Invoking the bounded inverse stability constant of the saddle-point operator directly controls the primal flux error component via the system norm:
\begin{align*}
    \normtwo{\flux - \mAdapt} \le \left\| \begin{bmatrix} \flux - \mAdapt \\ p^{\text{MFD}} - p_h^\tau \end{bmatrix} \right\|_2 \le \normtwo{(\mathcal{A}^\tau)^{-1}} \normtwo{\localRes^\tol} \le C_{\text{stab}} \normtwo{\localRes^\tol}.
\end{align*}
Incorporating the upper coercivity spectral weight $\mu_{\max}^{1/2}$ from the norm equivalence relation maps the boundary back into the adaptive energy space:
\begin{align*}
    \triple[\Madapt]{\flux - \mAdapt} \le \mu_{\max}^{1/2} \normtwo{\flux - \mAdapt} \le 2 C_{\text{stab}} P_{\max} \mu_{\max}^{1/2} \sqrt{|\Fset^\tau|} \tau.
\end{align*}
Performing normalized division by the reference solution energy norm $\triple[\Madapt]{\flux}$ yields the scale constant $C^\tau_1$ and concludes the proof.
\end{proof}

Although Proposition~\ref{prop:bound_la_conditional} establishes relative error control under symmetric grid configurations, the Local Adaptation (\text{LA}) paradigm breaks down on asymmetric grids due to a decoupled flux information across cell interfaces. Consider an interior interface $\fc = \partial\cell_1 \cap \partial\cell_2$ shared by a perfectly $\diffK$-orthogonal cell $\cell_1$ ($\delta_{\cell_1} = 0$) and a geometrically distorted neighbor $\cell_2$ possessing a significant non-vanishing consistency defect ($\delta_{\cell_2} \gg 0$). Because the \text{LA} framework evaluates marking metrics on unassembled cell subdomains, the localized evaluation for $\cell_1$ yields $\cellind_{\cell_1} \le \tol$, forcing the stencil activation flag to $\chi_{\cell_1} = 0$ and mistakenly retaining the low-cost two-point stencil $\Mtpfa_{\cell_1}$ on this matrix row. However, the geometric misalignment of $\cell_2$ generates an active local half-cell residual $(\localRes_{\cell_2}^\tol)_\fc \neq 0$. Absent an interface trace assembly step, the true global residual entry on the face evaluates to $|(\localRes^\tol)_\fc| = |(\localRes_{\cell_2}^\tol)_\fc| = \mathcal{O}(\|\Delta P_{\cell_2}\|_2)$, leaving the error perturbation completely unconstrained by the tolerance threshold check on $\cell_1$. This unbounded cross-interface error propagation induces a non-uniform loss of accuracy in both the computed pressure and flux fields, highlighting the operational necessity of the face-centered indicator assembly implemented under the Global Adaptation (\text{GA}) paradigm, which we formalize next.

\begin{prop}[Global Adaptation Interfacial Stability]\label{prop:residual_comparison}
Let $\fc = \bndry \cell_1 \cap \bndry \cell_2$ be an interior interface. Let $\fcnormRes_\Fset$ be the global face-normalized indicator assembled from the localized subdomain components. If there exists a uniform positive mesh constant $\gamma > 0$, independent of $\tau$, satisfying the normalized non-cancellation alignment condition:
\begin{align}
    \left| \frac{(\localRes_{\cell_{1}}^{\tau})_\fc}{\|\Delta P_{\cell_{1}}\|_2} + \frac{(\localRes_{\cell_{2}}^{\tau})_\fc}{\|\Delta P_{\cell_{2}}\|_2} \right| \ge \gamma \left( \frac{|(\localRes_{\cell_{1}}^{\tau})_\fc|}{\|\Delta P_{\cell_{1}}\|_2} + \frac{|(\localRes_{\cell_{2}}^{\tau})_\fc|}{\|\Delta P_{\cell_{2}}\|_2} \right), \label{eq:coercivity_condition}
\end{align}
then the individual unassembled local cell residual components are strictly bounded by the global interface indicator via:
\begin{align*}
    |(\localRes_{\cell_{1}}^{\tau})_\fc| \le \frac{\|\Delta P_{\cell_{1}}\|_2}{\gamma} |(\fcnormRes_\Fset)_\fc| \quad \text{and} \quad |(\localRes_{\cell_{2}}^{\tau})_\fc| \le \frac{\|\Delta P_{\cell_{2}}\|_2}{\gamma} |(\fcnormRes_\Fset)_\fc|.
\end{align*}
\end{prop}

\begin{proof}
Recognizing the left-hand side of inequality \eqref{eq:coercivity_condition} as the absolute definition of the assembled face indicator $|(\fcnormRes_\Fset)_\fc|$, dropping the non-negative second cell term from the right-hand side yields the direct inequality:
\begin{align*}
    |(\fcnormRes_\Fset)_\fc| \ge \gamma \left( \frac{|(\localRes_{\cell_{1}}^{\tau})_\fc|}{\|\Delta P_{\cell_{1}}\|_2} + \frac{|(\localRes_{\cell_{2}}^{\tau})_\fc|}{\|\Delta P_{\cell_{2}}\|_2} \right) \ge \gamma \frac{|(\localRes_{\cell_{1}}^{\tau})_\fc|}{\|\Delta P_{\cell_{1}}\|_2}.
\end{align*}
Multiplying by $\|\Delta P_{\cell_{1}}\|_2$ and dividing by the positive alignment constant $\gamma > 0$ isolates the bound for cell $\cell_1$. The symmetric bound for cell $\cell_2$ follows analogously, validating the statement.
\end{proof}

\begin{theorem}[Relative Error Bound under Global Adaptation]\label{thm:bound}
Let $\flux$ be the exact linear flux solution vector. Let $\mMfd$ and $\mAdapt$ denote the discrete full $\text{MFD}$ and adaptive flux solution vectors, with stable inverse operator bounds within the energy space. Assuming exact reproduction under full $\text{MFD}$ stencils ($\mMfd = \flux$), there exists a constant $\Cindep > 0$ completely independent of both $\tau$ and the mesh refinement size $h$ satisfying:
\begin{align*}
    \frac{\triple[\Madapt]{\flux - \mAdapt}}{\triple[\Madapt]{\flux}} \le \Cindep \, \tau.
\end{align*}
\end{theorem}

\begin{proof}
Let $e_{\flux} \coloneqq \mMfd - \mAdapt = \flux - \mAdapt$ denote the discrete flux error vector. Subtracting the active adapted operator rows from the reference full $\text{MFD}$ configuration isolates the exact discrete error layout $M^{\tau} e_{\flux} - B^T (p^{\text{MFD}} - p_h^\tau) = -\localRes^{\tau}$, where $B$ is the global discrete divergence operator, and $\localRes^{\tau} \coloneqq (M^{\text{MFD}} - M^{\tau})\mMfd$ is the unassembled operator mismatch residual vector supported strictly across the active face subset $\Fset^\tau$. Testing this system via an inner product with the flux error vector $e_{\flux}$ eliminates the pressure coupling block identically ($e_{\flux}^T B^T = 0$), since the flux error field is discretely divergence-free ($B e_{\flux} = 0$). Decomposing the remaining operators cell-wise over the active two-point flux approximation grid partition ($\Tset_{\text{TPFA}}^\tau$) reduces the global relation to the energy sum:
\begin{align*}
    \triple[\Madapt]{e_{\flux}}^2 = \sum_{\cell \in \Tset_{\text{TPFA}}^\tau} (e_{\flux, \cell})^T \localRes_{\cell}^\tol,
\end{align*}
where $\localRes_{\cell}^\tol \coloneqq (\Mmat_{\cell}^{\text{MFD}} - \Mmat_{\cell}^{\text{TPFA}})\flux_{\cell}$. To explicitly expose the local dual-space scaling transitions, we substitute the operator-split factorization $(e_{\flux, \cell})^T \localRes_{\cell}^\tol = \left((\Mmat_{\cell}^{\text{TPFA}})^{1/2} e_{\flux, \cell}\right)^T \left((\Mmat_{\cell}^{\text{TPFA}})^{-1/2} \localRes_{\cell}^\tol\right)$. Applying the vector Cauchy-Schwarz inequality to each element component, followed by the discrete Cauchy-Schwarz inequality across the global cell summation index, yields the fully explicit chain of boundaries:
\begin{align*}
    \triple[\Madapt]{e_{\flux}}^2 &= \sum_{\cell \in \Tset_{\text{TPFA}}^\tau} \left( (\Mmat_{\cell}^{\text{TPFA}})^{1/2} e_{\flux, \cell} \right)^T \left( (\Mmat_{\cell}^{\text{TPFA}})^{-1/2} \localRes_{\cell}^\tol \right) \\
    &\le \sum_{\cell \in \Tset_{\text{TPFA}}^\tau} \left\| (\Mmat_{\cell}^{\text{TPFA}})^{1/2} e_{\flux, \cell} \right\|_2 \left\| (\Mmat_{\cell}^{\text{TPFA}})^{-1/2} \localRes_{\cell}^\tol \right\|_2 \\
    &= \sum_{\cell \in \Tset_{\text{TPFA}}^\tau} \left( (e_{\flux, \cell})^T \Mmat_{\cell}^{\text{TPFA}} e_{\flux, \cell} \right)^{1/2} \left( (\localRes_{\cell}^\tol)^T (\Mmat_{\cell}^{\text{TPFA}})^{-1} \localRes_{\cell}^\tol \right)^{1/2} \\
    &\le \left( \sum_{\cell \in \Tset_{\text{TPFA}}^\tau} (e_{\flux, \cell})^T \Mmat_{\cell}^{\text{TPFA}} e_{\flux, \cell} \right)^{1/2} \left( \sum_{\cell \in \Tset_{\text{TPFA}}^\tau} (\localRes_{\cell}^\tol)^T (\Mmat_{\cell}^{\text{TPFA}})^{-1} \localRes_{\cell}^\tol \right)^{1/2}.
\end{align*}
Recognizing that the partial sum satisfies $\sum_{\cell \in \Tset_{\text{TPFA}}^\tau} (e_{\flux, \cell})^T \Mmat_{\cell}^{\text{TPFA}} e_{\flux, \cell} \le \triple[\Madapt]{e_{\flux}}^2$ due to the positive definiteness of the unadapted cell blocks, direct division isolates the localized dual-space stability relationship:
\begin{align}
    \triple[\Madapt]{\flux - \mAdapt} \le \left( \sum_{\cell \in \Tset_{\text{TPFA}}^\tau} (\localRes_{\cell}^\tol)^T (\Mmat_{\cell}^{\text{TPFA}})^{-1} \localRes_{\cell}^\tol \right)^{1/2} \label{eq:ga_dual_bound}.
\end{align}
To evaluate the asymptotic behavior under uniform refinement ($h \to 0$) in general $d$-dimensional space, we observe that face-integrated mass fluxes scale as $\mathcal{O}(h^{d-1})$, requiring the local operator to scale as $\Mmat_{\cell} = \mathcal{O}(h^{2-d})$ to preserve continuous $L^2(\Omega)$ energy dimensions. This dictates an inverse dual operator scaling of $(\Mmat_{\cell}^{\text{TPFA}})^{-1} = \mathcal{O}(h^{d-2})$. Concurrently, the global adaptation criteria matched with Proposition~\ref{prop:residual_comparison} guarantees that localized residuals scale strictly with local cell pressure drops, yielding a squared norm scaling of $\normtwosq{\localRes_{\cell}^\tol} = \mathcal{O}(h^2 \tau^2)$. Combining these dual properties bounds the local inverse energy contribution of a single element by $\mathcal{O}(h^d \tau^2)$, scaling directly with cell volume. Because uniform refinement increases the global cell count as $|\Tset| = \mathcal{O}(h^{-d})$, accumulating these components across the full active partition cancels out the grid size parameters identically:
\begin{align*}
    \sum_{\cell \in \Tset_{\text{TPFA}}^\tau} (\localRes_{\cell}^\tol)^T (\Mmat_{\cell}^{\text{TPFA}})^{-1} \localRes_{\cell}^\tol \le \mathcal{O}(h^{-d}) \cdot \mathcal{O}(h^d \tau^2) \equiv \Cindep^2 \tau^2 \label{eq:asymptotic_cancellation}.
\end{align*}
Substituting this scale-invariant upper bound back into \eqref{eq:ga_dual_bound} and performing normalized division by $\triple[\Madapt]{\flux}$ concludes the linear tracking proof.
\end{proof}
Having developed the adaptive framework and established its theoretical properties, we now examine how the proposed indicator behaves in practice across a range of challenging reservoir simulation benchmarks.

%% file: application.tex
\section{Numerical Experiments}

The numerical experiments are designed to evaluate both the accuracy and the computational efficiency of the proposed adaptive MFD framework. We begin with a convergence study to examine the behavior of the residual-based indicators under simultaneous mesh refinement and tolerance variation and to verify the convergence properties predicted by the theoretical analysis. The remaining experiments focus on the performance of the indicator-based adaptive MFD scheme on a collection of increasingly complex mesh configurations. 

For each benchmark, we investigate the effect of the tolerance $\tau$ on the resulting TPFA-MFD classification and the relative flux error. We additionally quantify the reduction in stencil complexity through the sparsity of the adaptive inner-product matrix $\Madapt$. Throughout this section, the sparsity reduction is defined as
\[
\frac{\mathrm{nnz}(M^{\text{MFD}})-\mathrm{nnz}(\Madapt)}
{\mathrm{nnz}(M^{\text{MFD}})}
\times 100~(\%),
\]
where $\mathrm{nnz}(\cdot)$ denotes the number of nonzero matrix entries. Given the robustness of the Global Adaptation (GA) strategy, all these subsequent benchmarks results are presented using the GA framework.

Additionally, we measure the relative pressure error in the discrete $L^2$-norm and the relative flux error in the energy norm induced by $\Madapt$ such that
\begin{equation*}
    e_p=\frac{\|p_h-p_{\text{ref}}\|_2}{\|p_{\text{ref}}\|_2}, \qquad
    e_m=\frac{\triple[\Madapt]{\flux - \flux_{\text{ref}}}}{\triple[\Madapt]{\flux_{\text{ref}}} }.
\end{equation*}

where the context of the reference solutions $(p_\text{ref},\flux_\text{ref})$ corresponds to the full MFD solution, with the exception of the next subsection.

\subsection{Convergence Behavior of the Adaptive MFD Framework}
We begin with a controlled convergence study in which both the Local Adaptation (LA) and Global Adaptation (GA) strategies are applied to the same benchmark problem. The resulting error trends provide a direct comparison of the two indicator constructions across a range of mesh resolutions and tolerance values.

To provide a controlled verification setting, we construct an extruded mesh containing both TPFA-consistent and TPFA-inconsistent regions. The mesh is obtained by extruding a structured Cartesian grid on the unit square $\Omega = [0,1] \times [0,1]$. A localized cross-shaped deformation is introduced near the center of the domain by applying a vertical and horizontal shear ribbons over $x,y \in [0.4, 0.6]$ with vertex perturbations of the form $\delta y \propto \sin(\pi y)$ and $\delta x \propto \sin(\pi x)$. The resulting mesh contains a localized region of non-$\diffK$-orthogonal cells while remaining Cartesian elsewhere. This structure makes TPFA remains exact away from the distorted region, whereas the central ribbon zone provides a controlled setting in which the adaptive discretization must identify where the full MFD operator is required. 

Uniform mesh refinement is then performed by doubling the number of cells per coordinate direction. We consider refinement levels $N_h \in \{ 8,16,32,64,128,256,512\}$ where $N_h$ denotes the number of cells along each coordinate direction with characteristic mesh size $h = 1/N_h$. Since the ribbon deformation occupies a fixed physical region, mesh refinement increases the resolution of the distorted zone while preserving its overall geometric extent.

For each refinement level, we first evaluate the residual-based indicators using the linear pressure field
\begin{equation*}
    p_{\lin}(x,y,z)=x+y+1,
\end{equation*}
together with the corresponding projected flux field obtained from the harmonic face projection described in Proposition~\ref{prop:admissibility_orthogonal}. Since TPFA exactly reproduces linear solutions on $\diffK$-orthogonal meshes, any nonzero indicator value is attribute solely to the local loss of 
$\diffK$-orthogonality induced by the ribbon deformation. Using the projected pressure-flux pair, we construct both the Local Adaptation and Global Adaptation indicators described in Section~3. Cells whose indicator exceeds a prescribed tolerance $\tau$ are classified as MFD-compatible, while the remaining cells retain the TPFA stencil. This procedure yields a tolerance-dependent cell partition, which is computed for $\tau \in \{10^{0},10^{-1},10^{-2},10^{-3},10^{-4}\}$.

For each tolerance level, the corresponding adaptive operator $\Madapt$ is assembled and used to solve a manufactured Darcy problem with exact solution 
\begin{equation*}
p_{\exact}(x,y)
=
\sin(2\pi x)\sin(2\pi y)+x,
\end{equation*}
subject to Dirichlet boundary conditions prescribed from the exact field. The associated source term is given by $f = 8\pi^2 \sin(2\pi x)\sin(2\pi y)$ and an isotropic tensor $\diffK = \mathbf{I}$ is prescribed throughout the domain. This manufactured solution introduces nontrivial pressure variations throughout the domain while concentrating the strongest gradients within the distorted ribbon region providing a stringent test of the adaptive discretization.

We measure the relative pressure error in the discrete $L^2$-norm and the relative flux error in the energy norm, considering the association $(p_\text{ref},\flux_\text{ref}) \rightarrow (p_{\exact},\flux_{\exact})$, for the pressure and exact flux. The errors are reported as functions of the mesh size $h$ for each tolerance level. Provided that the adaptive classifier correctly identifies the TPFA-inconsistent cells, the adaptive discretization is expected to recover the asymptotic convergence behavior of the underlying MFD scheme. Consequently, for sufficiently small values of $\tau$, we expect
$e_p=\mathcal{O}(h^2)$ and 
$e_m=\mathcal{O}(h)$ as $h \to 0$.

\begin{figure}[H]
    \centering
    \begin{minipage}{0.49\textwidth}
        \centering
        \includegraphics[width=\linewidth]{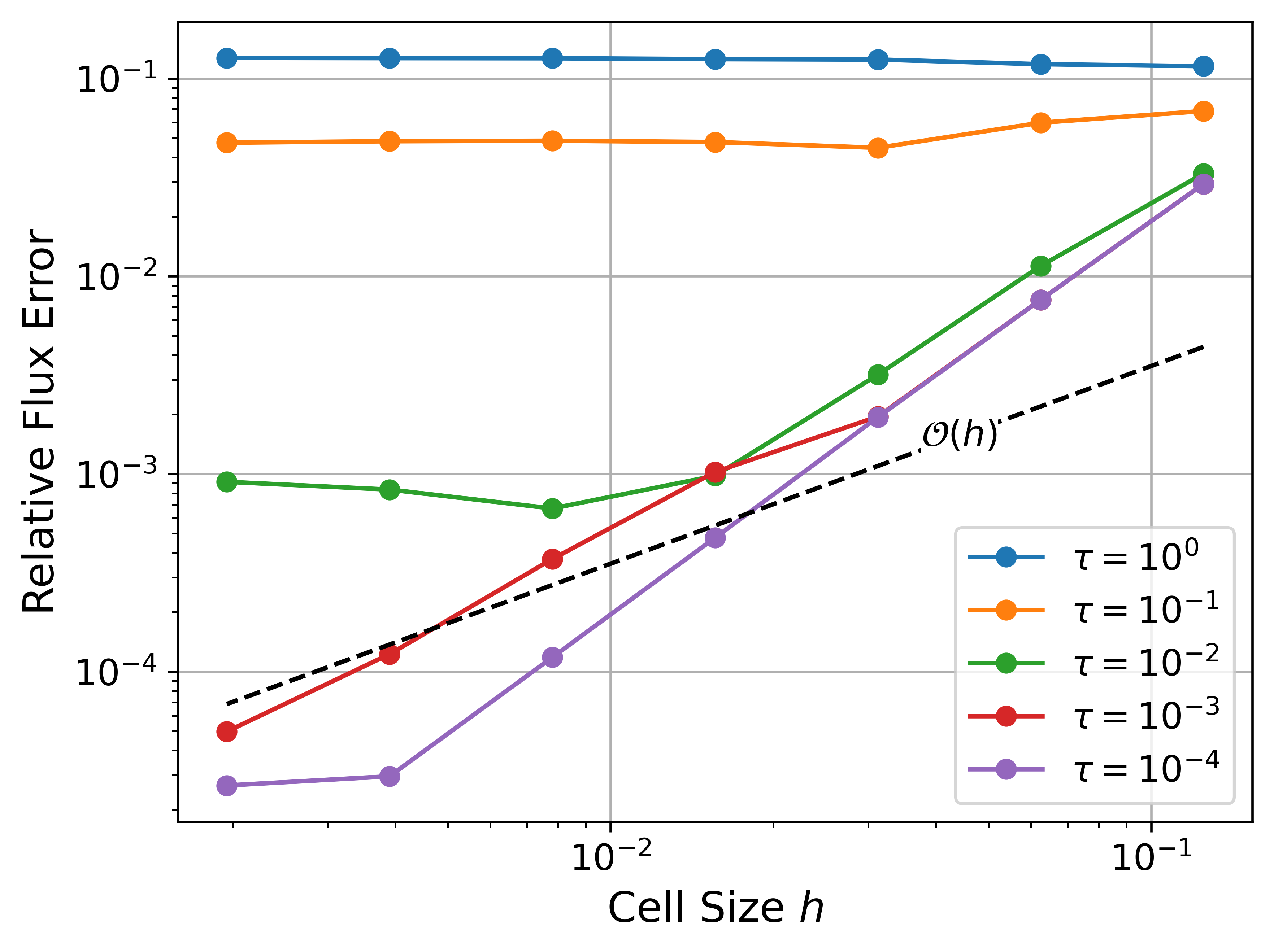}
        \vspace{0.25em}
        \small (a) Flux convergence (LA)
    \end{minipage}
    \hfill
    \begin{minipage}{0.49\textwidth}
        \centering
        \includegraphics[width=\linewidth]{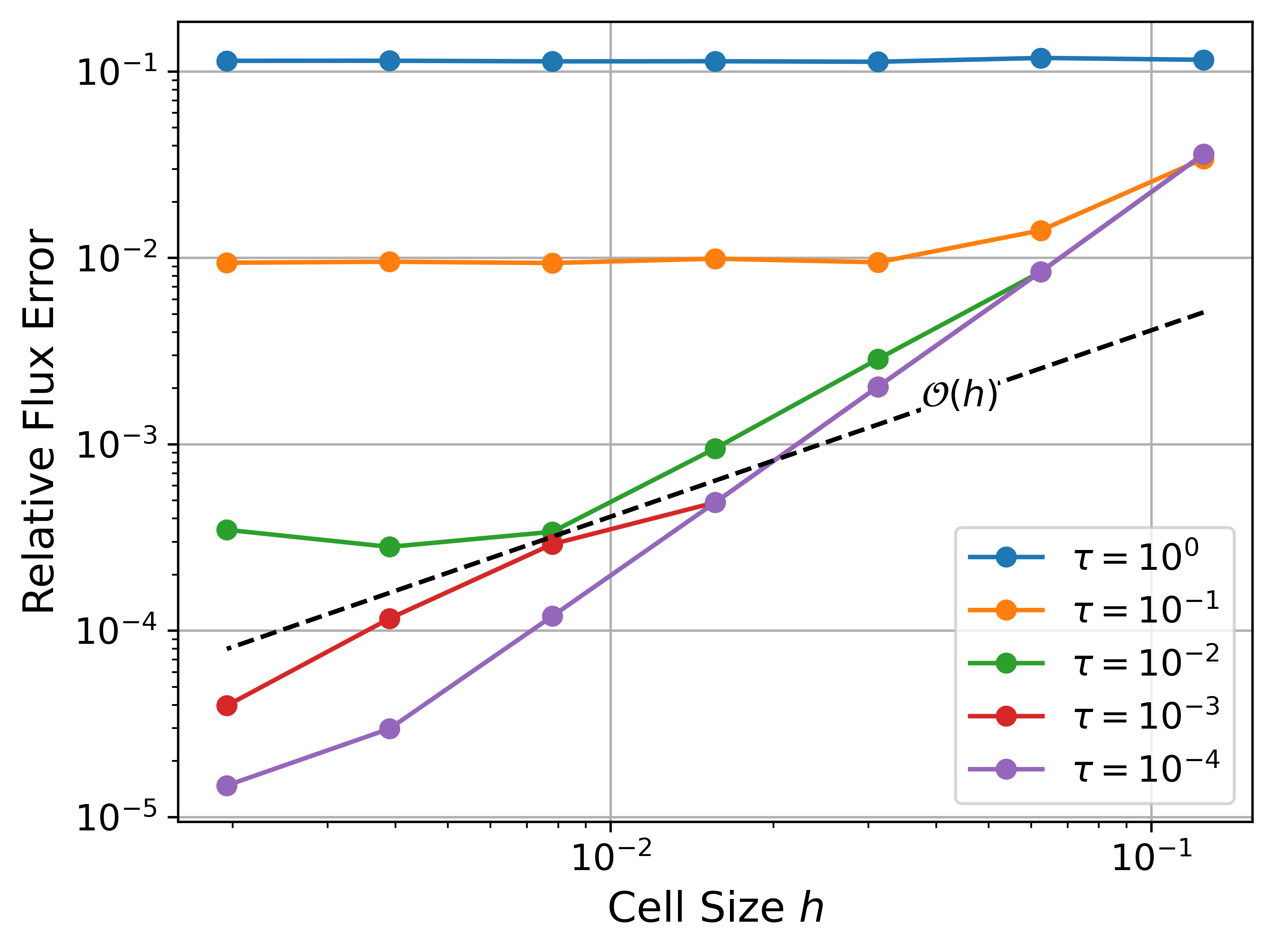}
        \vspace{0.25em}
        \small (b) Flux convergence (GA)
    \end{minipage}

    \vspace{0.75em}

    \begin{minipage}{0.49\textwidth}
        \centering
        \includegraphics[width=\linewidth]{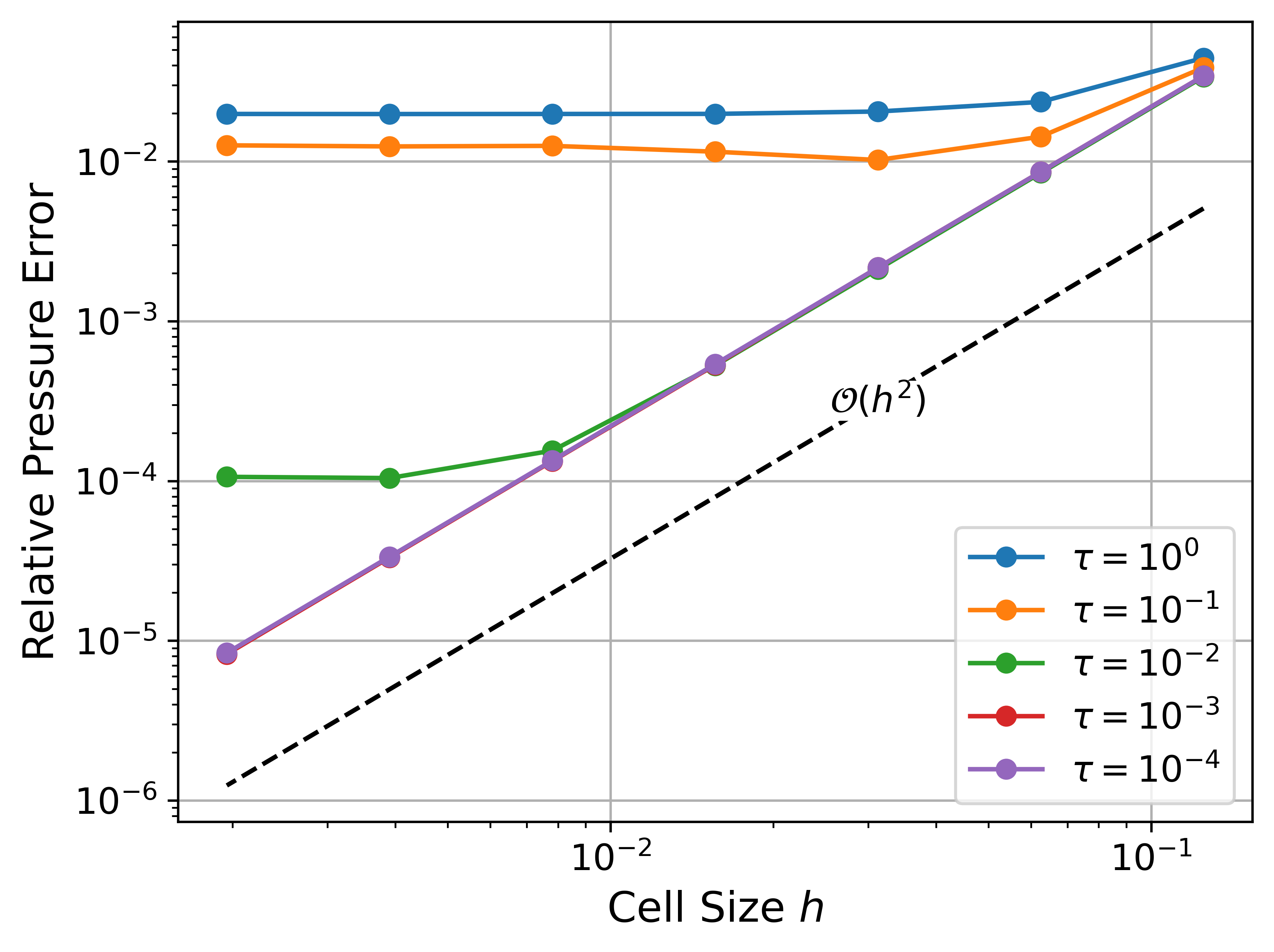}
        \vspace{0.25em}
        \small (c) Pressure convergence (LA)
    \end{minipage}
    \hfill
    \begin{minipage}{0.49\textwidth}
        \centering
        \includegraphics[width=\linewidth]{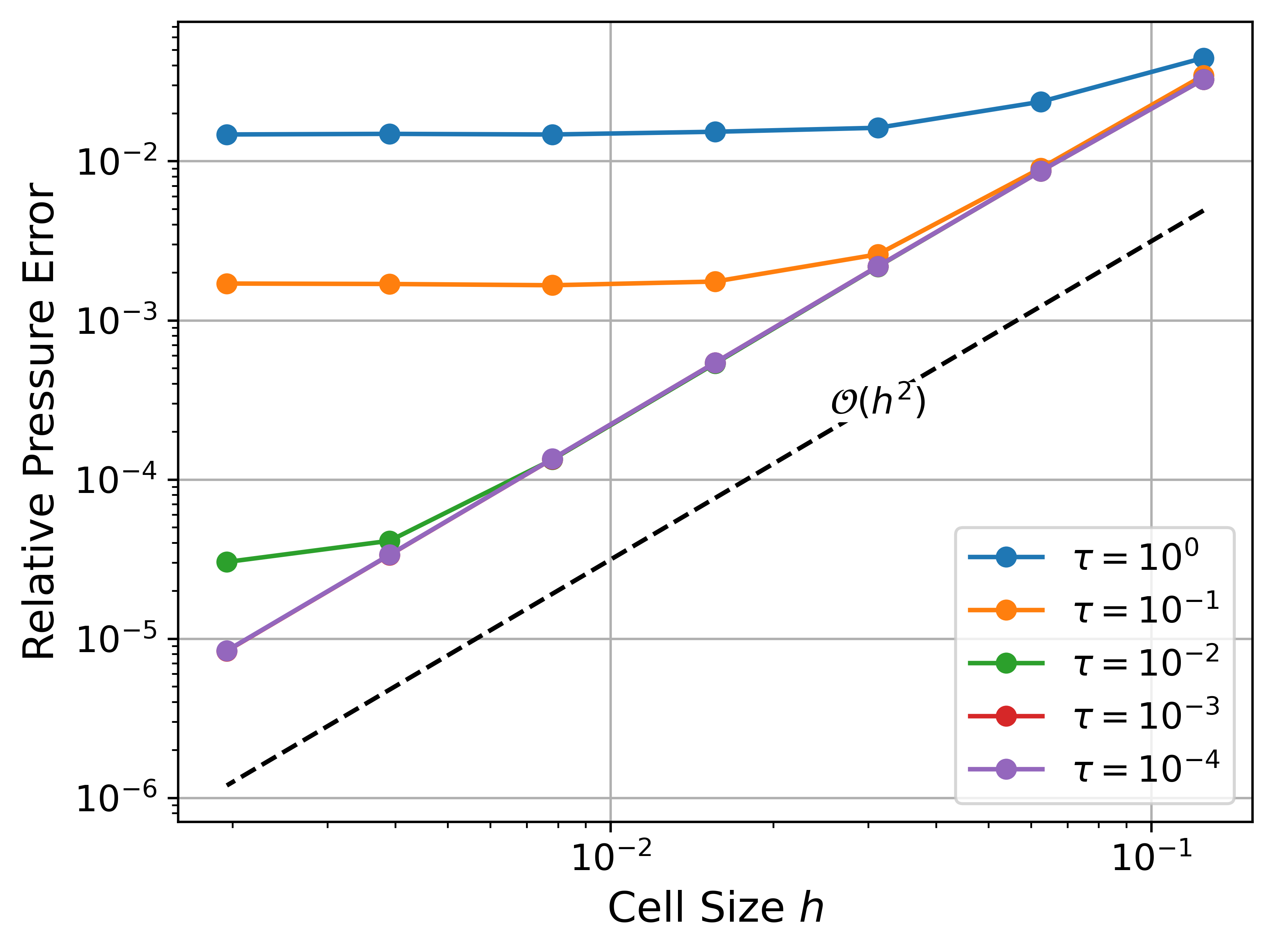}
        \vspace{0.25em}
        \small (d) Pressure convergence (GA)
    \end{minipage}

    \caption{Convergence behavior of the adaptive classification. Left column: local adaptation (LA); right column: global adaptation (GA). Top row: flux error convergence; bottom row: pressure error convergence. Each curve corresponds to a fixed tolerance $\tau$ and the mesh refinement levels $h$.}
    \label{fig:convergence_adaptive}
\end{figure}

From Figure~\ref{fig:convergence_adaptive}, two main observations emerge. The relative errors exhibit the expected joint dependence on the mesh size and the classifier tolerance; the pressure error behaves as $\mathcal{O}(\max(h^2,\tau))$, while the flux error behaves as $\mathcal{O}(\max(h,\tau))$. When the tolerance error dominates the mesh-induced discretization error, refining the mesh alone does not significantly improve the solution. Conversely, once the tolerance is reduced below the discretization floor, further tightening of $\tau$ yields little additional gain. Under the uniform structured refinement considered here, the flux convergence curves appear to display a higher-than-expected rate; this is likely due to symmetry-driven cancellations on the nearly Cartesian portions of the mesh. The general expected behavior, however, remains governed by the $\mathcal{O}(\max(h,\tau))$ bound.

The comparison between Local Adaptation and Global Adaptation further highlights the importance of incorporating residual information from neighboring cells. The GA strategy consistently reduces the error by approximately one order of magnitude for each order-of-magnitude decrease in $\tau$, in agreement with the tolerance-dependent estimate established in Theorem~\ref{thm:bound}. The LA curves, in contrast, exhibit a weaker and less uniform response to tolerance refinement, reflecting the inability of purely local residual indicators to account for cross-interface error propagation. This final observation further supports the global adaptive strategy as a reliable adaptation mechanism; consequently, it is employed exclusively in the remaining numerical experiments.

\subsection{SPE11B-based Validation Benchmark}
The next experiment considers a vertically extruded SPE11B benchmark model~\cite{Nordbotten2024} to evaluate the proposed residual-based classification procedure on a realistic reservoir-scale mesh. The distorted regions are spatially distributed across the domain and do not follow simple geometric patterns, making manual identification of cells requiring a full MFD treatment impractical. The computational domain is
\[
\Omega = [0, L_x] \times [0, L_y] \times [0, L_z] =
[0, 8400] \times [0,1200] \times [0,100].
\]
After extrusion, the resulting mesh contains $99031$ cells, $212326$ faces and $113296$ vertices. Figure~\ref{fig:spe11b_mesh} shows the resulting extruded benchmark geometry together with zoomed views of the distorted regions. 

\begin{figure}[H]
    \centering
    \includegraphics[width=1.0\textwidth]{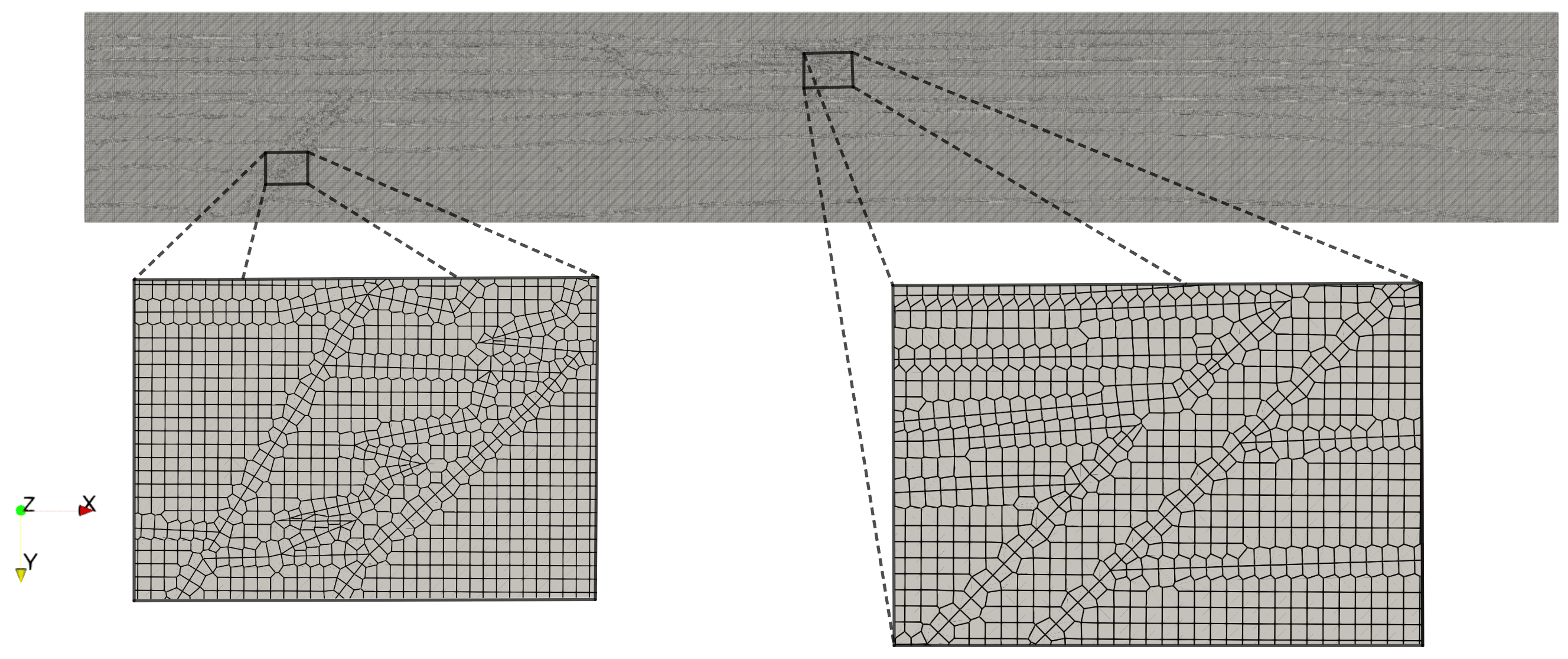}
    \caption{
    Top view of the extruded SPE11B benchmark geometry used in the validation experiment. The enlarged views highlight representative regions containing distorted and irregular polyhedral cells. Similar distorted patterns are distributed throughout the computational domain.
    }
    \label{fig:spe11b_mesh}
\end{figure}

To construct the adaptive indicators, we prescribe the linear pressure field 
\[
p_{\lin}(x,y,z) = 1 - \frac{x}{L_x},
\]
corresponding to the constant pressure gradient $\mathbf{g} = [-1/L_x, 0, 0]^T$. An isotropic permeability tensor $\diffK = \mathbf{I}$ is assumed throughout the domain. Following the projection procedure introduced in Section~3, the corresponding projected flux field is obtained using the harmonic interface permeability and serves as the reference admissible flow field for the classification procedure. Relative flux errors are evaluated with respect to the projected field in the adaptive energy norm $\Madapt$.

We first visualize the evolution of the adaptive cell classification as the tolerance parameter $\tau$ varies. Figure~\ref{fig:spe11b_cell_evolution} shows that the residual-based indicator progressively identifies and captures the geometrically distorted polyhedral regions as the tolerance decreases. For $\tau = 10^{-1}$, only a subset of the distorted regions are classified as MFD cells, while additional cells are activated as the TPFA consistency criterion becomes stricter. The overall classification pattern remains strongly correlated with the localized distorted regions visible in Figure~\ref{fig:spe11b_mesh}. Starting from approximately $\tau = 10^{-6}$, the cell classification stabilizes and the number of TPFA cells remains essentially unchanged, indicating that the adaptive procedure has fully identified the regions requiring a full MFD discretization.

\begin{table}[H]
\centering
\scriptsize
\setlength{\tabcolsep}{4pt}

\begin{tabular}{c|ccccccc}
\hline
\textbf{Metric}
& \textbf{Full TPFA}
& $\mathbf{10^{0}}$
& $\mathbf{10^{-1}}$
& $\mathbf{10^{-2}}$
& $\mathbf{10^{-4}}$
& $\mathbf{10^{-6}}$
& \textbf{Full MFD} \\
\hline

TPFA cells
& 99031
& 98945
& 74781
& 57081
& 55296
& 55271
& 0 \\

Sparsity reduction (\%)
& 88.92
& 88.85
& 61.72
& 46.23
& 44.79
& 44.77
& 0.00 \\

Relative flux error
& $4.84 \times 10^{-2}$
& $4.82 \times 10^{-2}$
& $1.07 \times 10^{-2}$
& $3.50 \times 10^{-4}$
& $9.74 \times 10^{-6}$
& $3.79 \times 10^{-11}$
& $3.75 \times 10^{-11}$ \\

\hline
\end{tabular}

\caption{
Tolerance-dependent TPFA cell count, sparsity reduction, and relative flux error for the validation benchmark. Results for tolerances below $\tau = 10^{-6}$ are omitted since all reported metrics stabilize and remain essentially unchanged.
}
\label{tab:spe11b_validation_results}
\end{table}

\begin{figure}[H]
\centering

\begin{minipage}{1.0\textwidth}
    \centering
    \includegraphics[width=\linewidth]{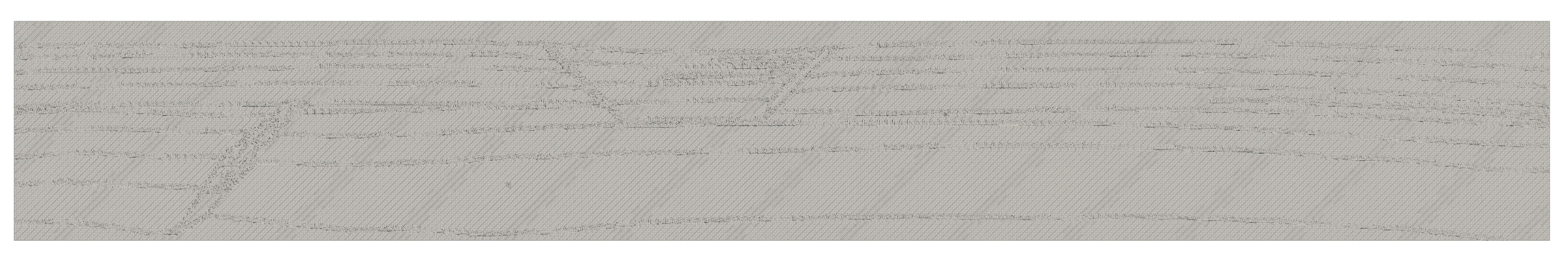}
    \\
    \small (a) Full TPFA
\end{minipage}

\vspace{0.5em}

\begin{minipage}{1.0\textwidth}
    \centering
    \includegraphics[width=\linewidth]{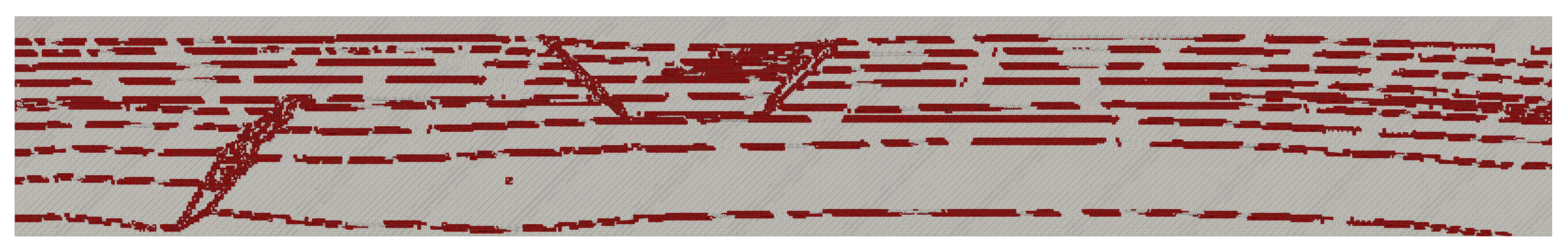}
    \\
    \small (b) $\tau = 10^{-1}$
\end{minipage}

\vspace{0.5em}

\begin{minipage}{1.0\textwidth}
    \centering
    \includegraphics[width=\linewidth]{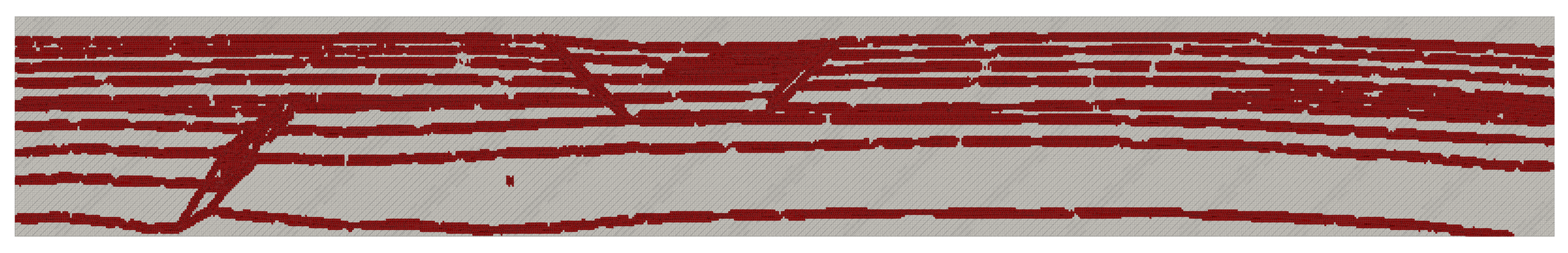}
    \\
    \small (c) $\tau = 10^{-2}$
\end{minipage}

\vspace{0.5em}

\begin{minipage}{1.0\textwidth}
    \centering
    \includegraphics[width=\linewidth]{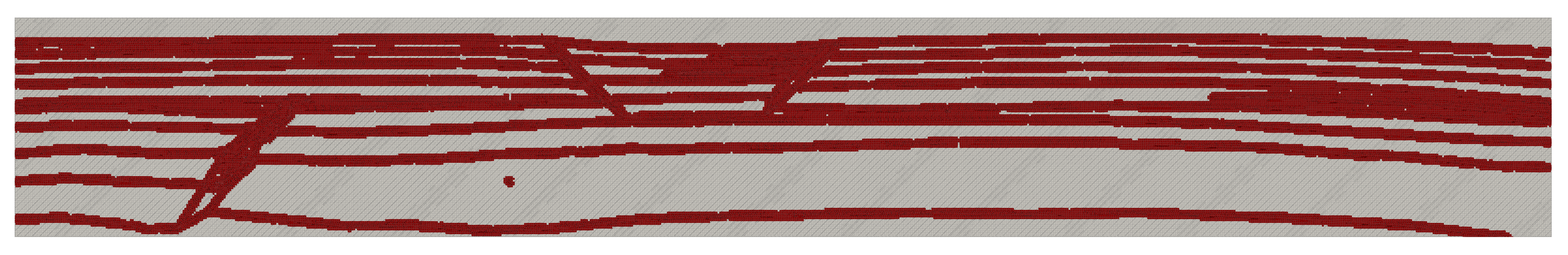}
    \\
    \small (d) $\tau = 10^{-6}$
\end{minipage}

\caption{Evolution of the adaptive cell classification for the extruded SPE11B benchmark. Red cells denote regions where the full MFD discretization is applied.
}
\label{fig:spe11b_cell_evolution}
\end{figure}

Table~\ref{tab:spe11b_validation_results} summarizes the evolution of the TPFA cell count together with the corresponding sparsity reduction and relative flux error for different values of the tolerance $\tau$. Consistent with the classification patterns observed in the cell evolution figure, tightening the tolerance progressively reduces the number of cells classified as TPFA-compatible, resulting in a denser adaptive inner-product matrix $\Madapt$ and a corresponding decrease in sparsity reduction. Even after the classification has stabilized, the adaptive MFD framework retains approximately $45\%$ sparsity reduction while achieving a relative flux error that is nearly indistinguishable from that of the full MFD discretization. These results demonstrate that a significant fraction of the expensive MFD stencil can be eliminated without sacrificing accuracy.

\begin{figure}[H]
\centering
\includegraphics[width=0.5\textwidth]{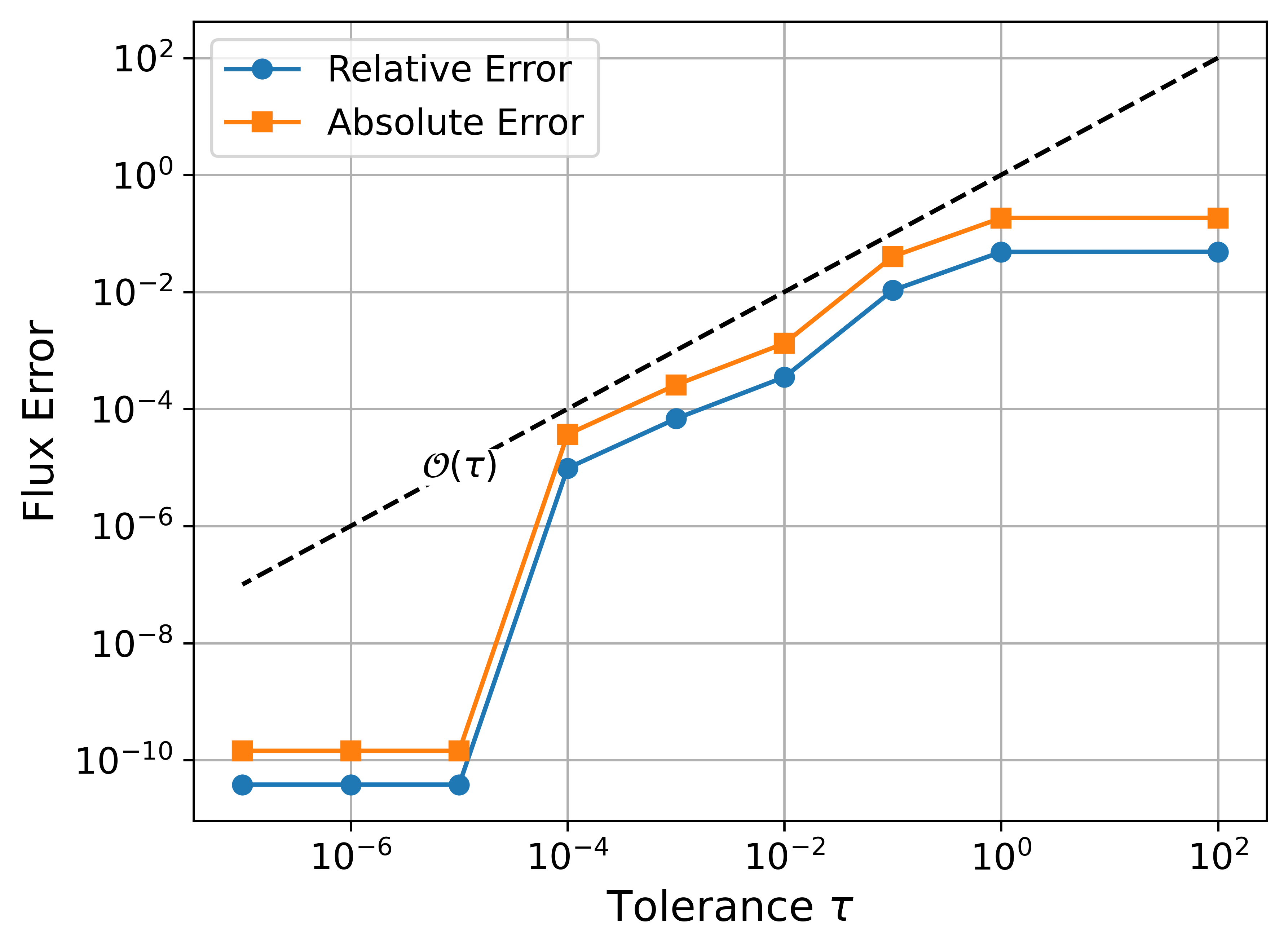}
\caption{
Mass flux error versus tolerance for the SPE11B-based model. The diagonal reference line represents the order of tolearnce $\tau$.
}
\label{fig:flux_error_case1}
\end{figure}

Figure~\ref{fig:flux_error_case1} provides a more detailed view of the tolerance-dependent behavior of both the relative and absolute flux errors. Consistent with the theoretical estimate in Theorem~\ref{thm:bound}, the relative flux error remains bounded by the prescribed tolerance and decreases in a stable manner as $\tau$ becomes smaller. This behavior demonstrates that the tolerance parameter serves as an effective mechanism for controlling the accuracy of the adaptive discretization, allowing the resulting flux error to be estimated a priori from the prescribed threshold.

\subsection{3D Faulted Reservoir Benchmark}
We next evaluate the adaptive strategy on a three-dimensional two-fault model. The computational domain is
\[
\Omega = [0,L_x] \times [0, L_y] \times [0, L_z]
= [0,1] \times [0, 1] \times [0, 0.4],
\]
and contains two fault surfaces that introduce geometric complexity into the polyhedral mesh, as shown in Figure~\ref{fig:two-fault_mesh}. For this benchmark, we consider two permeability tensor settings. In the first case, the domain is divided into three horizontal layers, where each layer is assigned a different isotropic permeability tensor. In the second case, the same layered structure is retained, but each layer additionally contains heterogeneous anisotropic permeability tensors varying from cell to cell. The detailed construction of these permeability fields is described in the following subsections. 
\begin{figure}[H]
\centering
\includegraphics[width=0.57\textwidth]{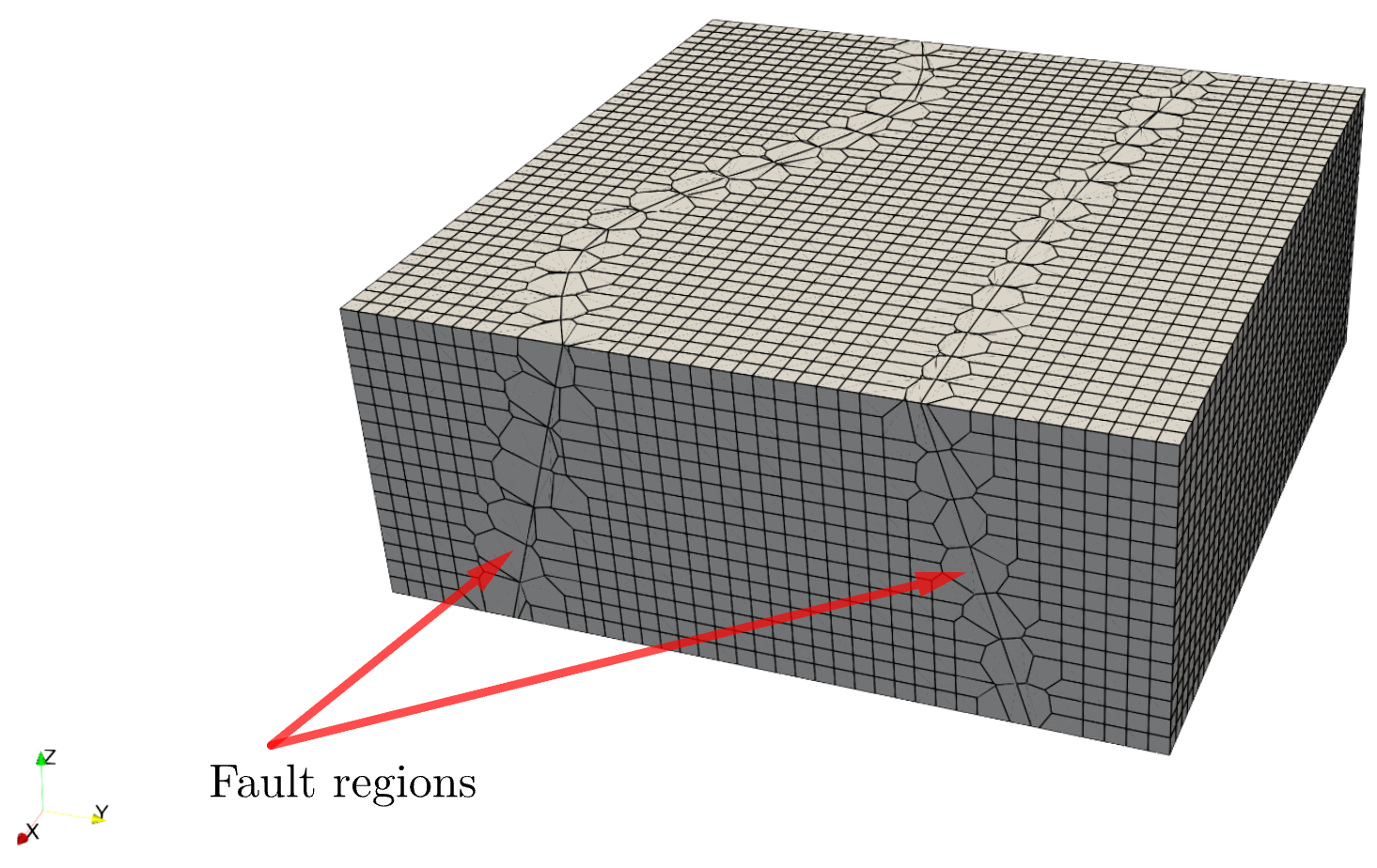}
\caption{A two-fault model generated by the MRST unstructured Voronoi gridding framework~\protect\cite{advancedMRST,Berge2018}. The mesh is predominantly Cartesian away from the faults, while irregular polyhedral cells are introduced near the fault regions to capture geometric complexity. The mesh consists of 22056 cells, 76721 faces, and 32649 vertices.}
\label{fig:two-fault_mesh}
\end{figure}

To drive flow across the entire domain, we impose a corner-to-corner flow configuration in the two-fault model. The inlet is located near the upper corner $(0,0,L_z)$ while the outlet is placed near the diagonally opposite lower corner $(L_x, L_y, 0)$. The flow problem is governed by the mixed Darcy system introduced in Section 2 with the permeability tensor $\diffK$ specified according to the permeability configuration considered in each experiment. The pressure boundary conditions are imposed on the inlet and outlet cells and homogeneous no-flow conditions are prescribed on the remaining boundaries.

For the residual-based classification procedure, we prescribe the linear pressure field
\[
p_\lin (x,y,z) = 1 - \frac{x}{L_x} - \frac{y}{L_y} - \frac{z}{L_z}
\]
which corresponds to the constant gradient field $\mathbf{g} = [-1/L_x, -1/L_y, -1/L_z]^T$. A harmonic-average-based projected flux field is also constructed from $\mathbf{g}$ using the neighboring-cell permeability information. The resulting projected pressure-flux pair is then used to evaluate the residual-based indicators and determine the adaptive TPFA/MFD partition.

After solving the flow problem, the computed numerical flux field is used to evolve the transport equation
\[
\phi \frac{\partial S}{\partial t} + 
\nabla \cdot (\mathbf{m}S) = 0 \quad \text{in } \Omega,
\]
where $S$ denotes the transported saturation field and $\phi = 0.3$ is the porosity. The transport equation is discretized using a fully implicit finite-volume upwind scheme based on the computed numerical fluxes. On the inflow boundary, the injected saturation is prescribed as $S = S_\text{inj}$ and the initial condition is given by $S(x,0) = 0$ throughout the domain. Since no closed-form analytical solution is available for this faulted configuration, the full MFD is used as the numerical reference solution for evaluating both flux and saturation error in the following experiments.

\subsubsection{Layered Isotropic Permeability Case}
The first permeability configuration consists of a layered isotropic permeability field, shown in Figure~\ref{fig:layered_perm}. The permeability tensor is assigned cell-wise according to the cell-center depth coordinate $z$, with three piecewise-constant isotropic layers given by
\[
k(z) = 
\begin{cases}
    4, & z < 0.13, \\
    1, & 0.13 \le z \le 0.26, \\
    0.1, & z \ge 0.26,
\end{cases}
\]
and the permeability tensor is taken as 
\[
\diffK(z) = k(z) \mathbf{I}
\]
where $\mathbf{I}$ denotes the $3 \times 3$ identity tensor. Thus, the upper layer is the most permeable, while the lower layer is the least permeable.

\begin{figure}[H]
\centering

\begin{minipage}{0.32\textwidth}
    \centering
    \includegraphics[width=\linewidth]{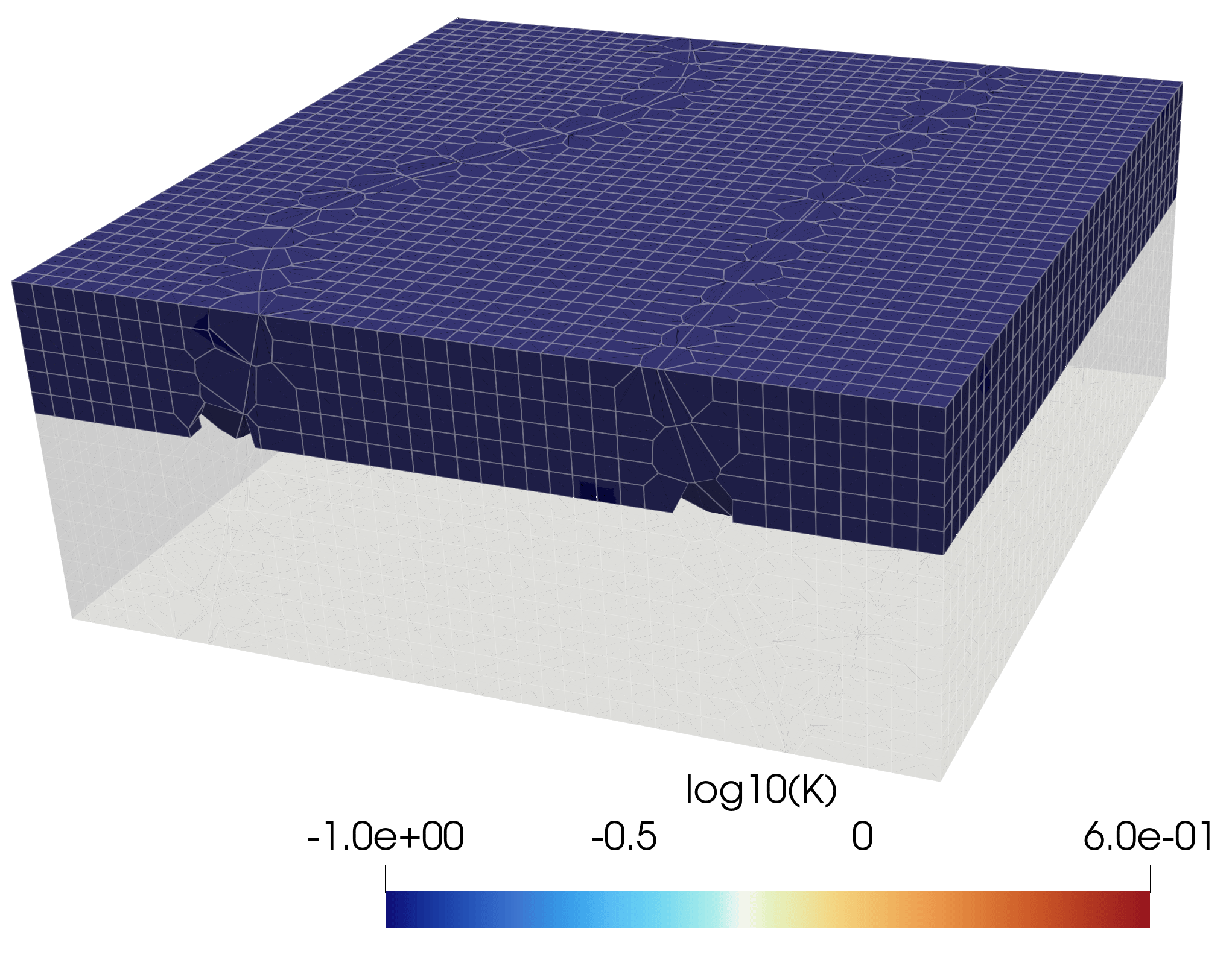}
    \\
    \small (a) Top layer: $\diffK(z) = 4 \mathbf{I}$
\end{minipage}
\hfill
\begin{minipage}{0.32\textwidth}
    \centering
    \includegraphics[width=\linewidth]{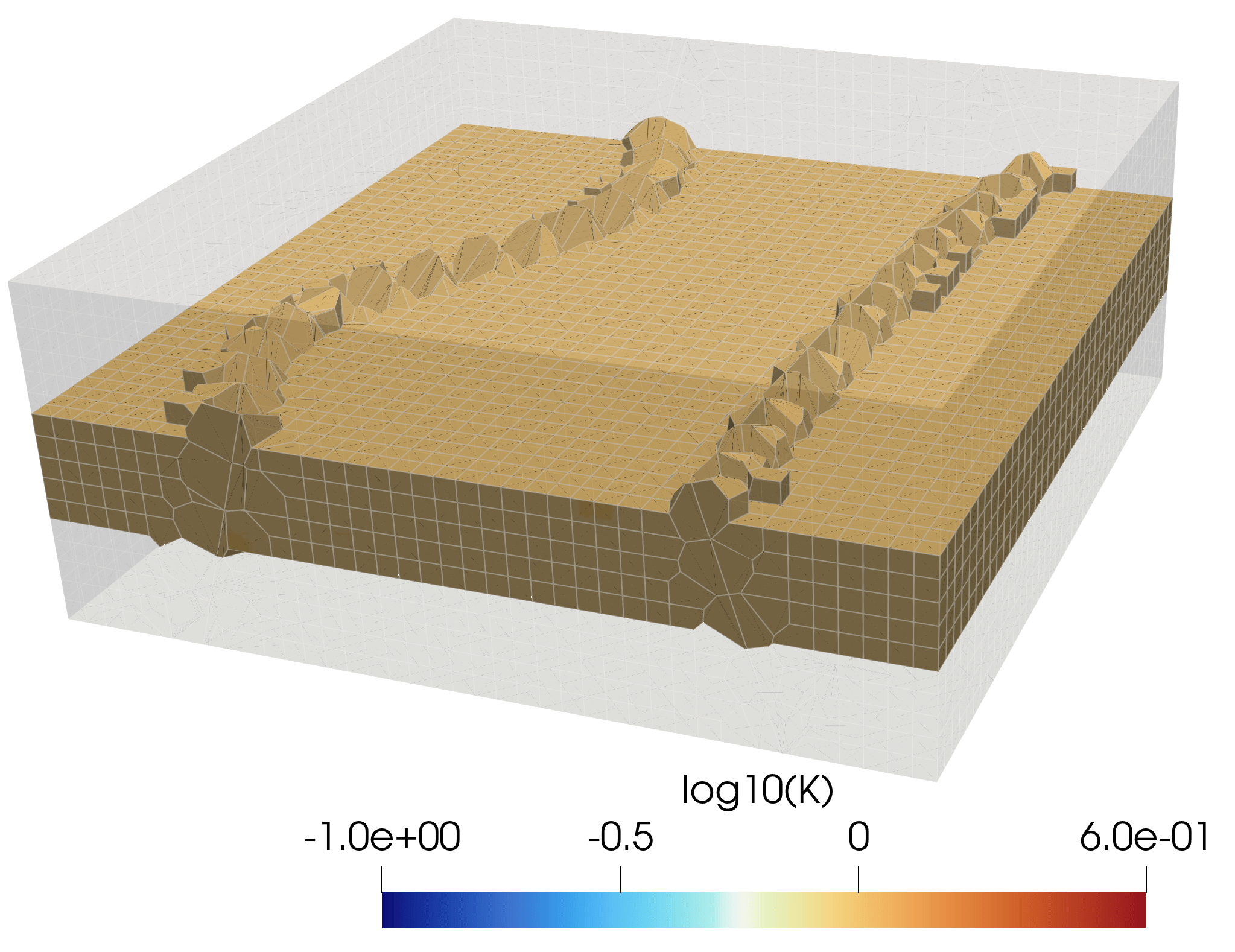}
    \\
    \small (b) Middle layer: $\diffK(z) = \mathbf{I}$
\end{minipage}
\hfill
\begin{minipage}{0.32\textwidth}
    \centering
    \includegraphics[width=\linewidth]{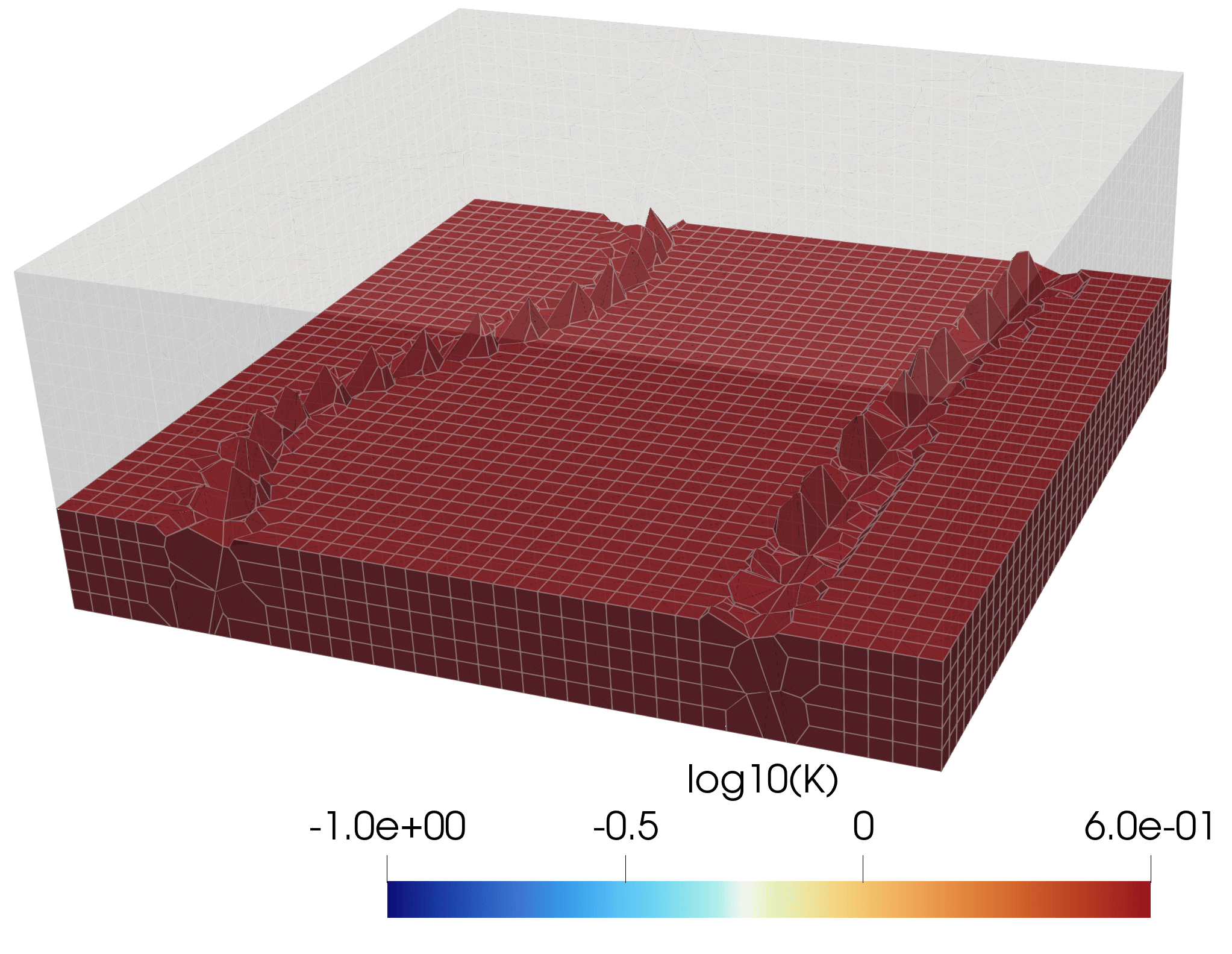}
    \\
    \small (c) Bottom layer: $\diffK(z) = 0.1 \mathbf{I}$
\end{minipage}

\caption{
Three-layer isotropic permeability field used in the 3D two-fault benchmark. Colors indicate $\log_{10}(k)$, where $k$ denotes the scalar permeability assigned to each layer.
}
\label{fig:layered_perm}
\end{figure}

\noindent Figure~\ref{fig:layered_homogeneous_fault_cells} illustrates the evolution of the TPFA/MFD classification as the tolerance parameter is reduced. As the classification criterion becomes more restrictive, an increasing number of cells near the fault surfaces are identified as MFD-compatible. Despite the permeability contrast between layers, the permeability tensor remains isotropic throughout the domain. The resulting classification pattern remains strongly localized around the fault regions, indicating that stencil selection is driven primarily by geometric distortion rather than permeability variation.

\begin{figure}[H]
\centering

\begin{minipage}{0.32\textwidth}
    \centering
    \includegraphics[width=\linewidth]{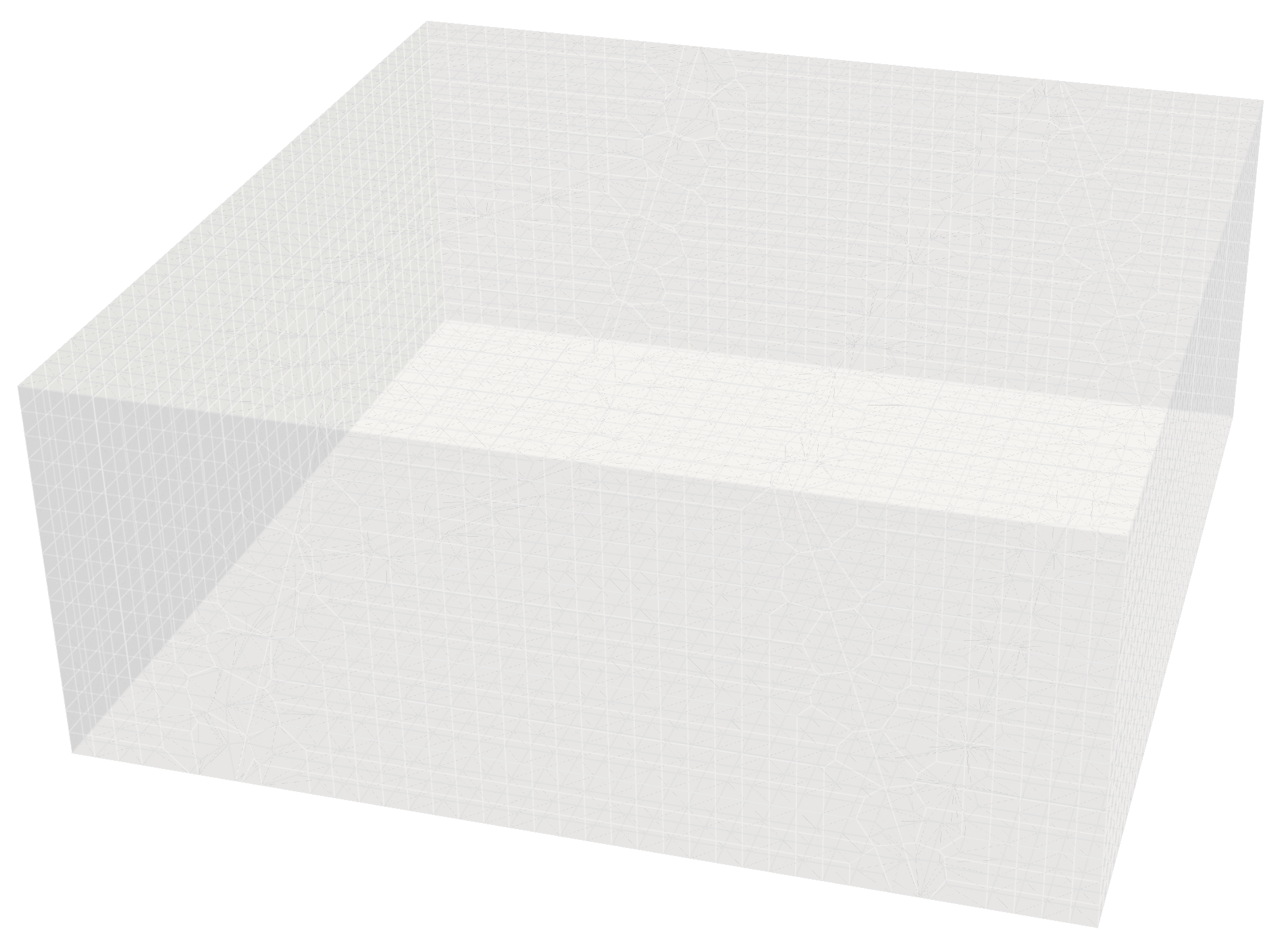}
    \\
    \small (a) Full TPFA
\end{minipage}
\hfill
\begin{minipage}{0.32\textwidth}
    \centering
    \includegraphics[width=\linewidth]{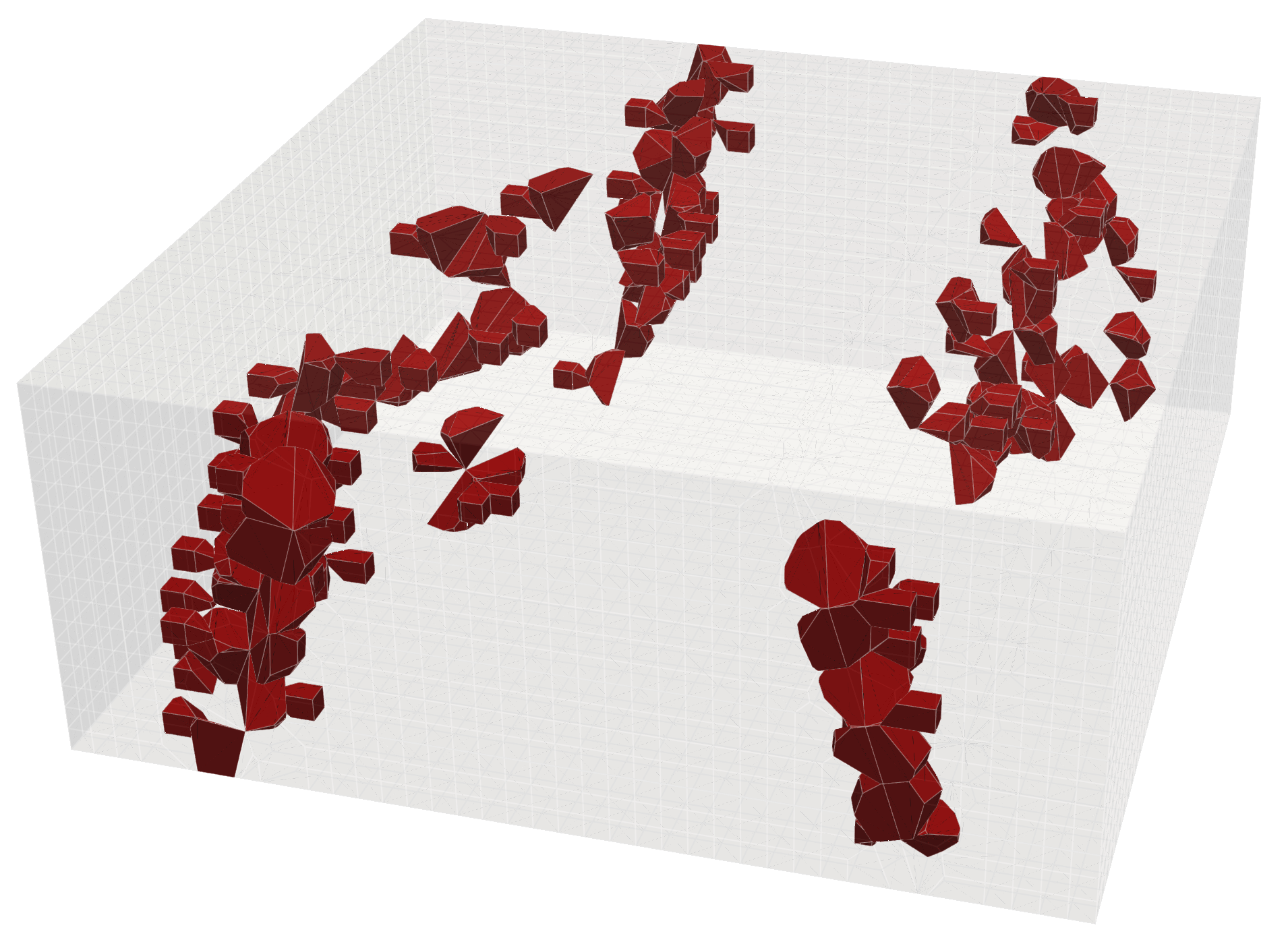}
    \\
    \small (b) $\tau = 10^{0}$
\end{minipage}
\hfill
\begin{minipage}{0.32\textwidth}
    \centering
    \includegraphics[width=\linewidth]{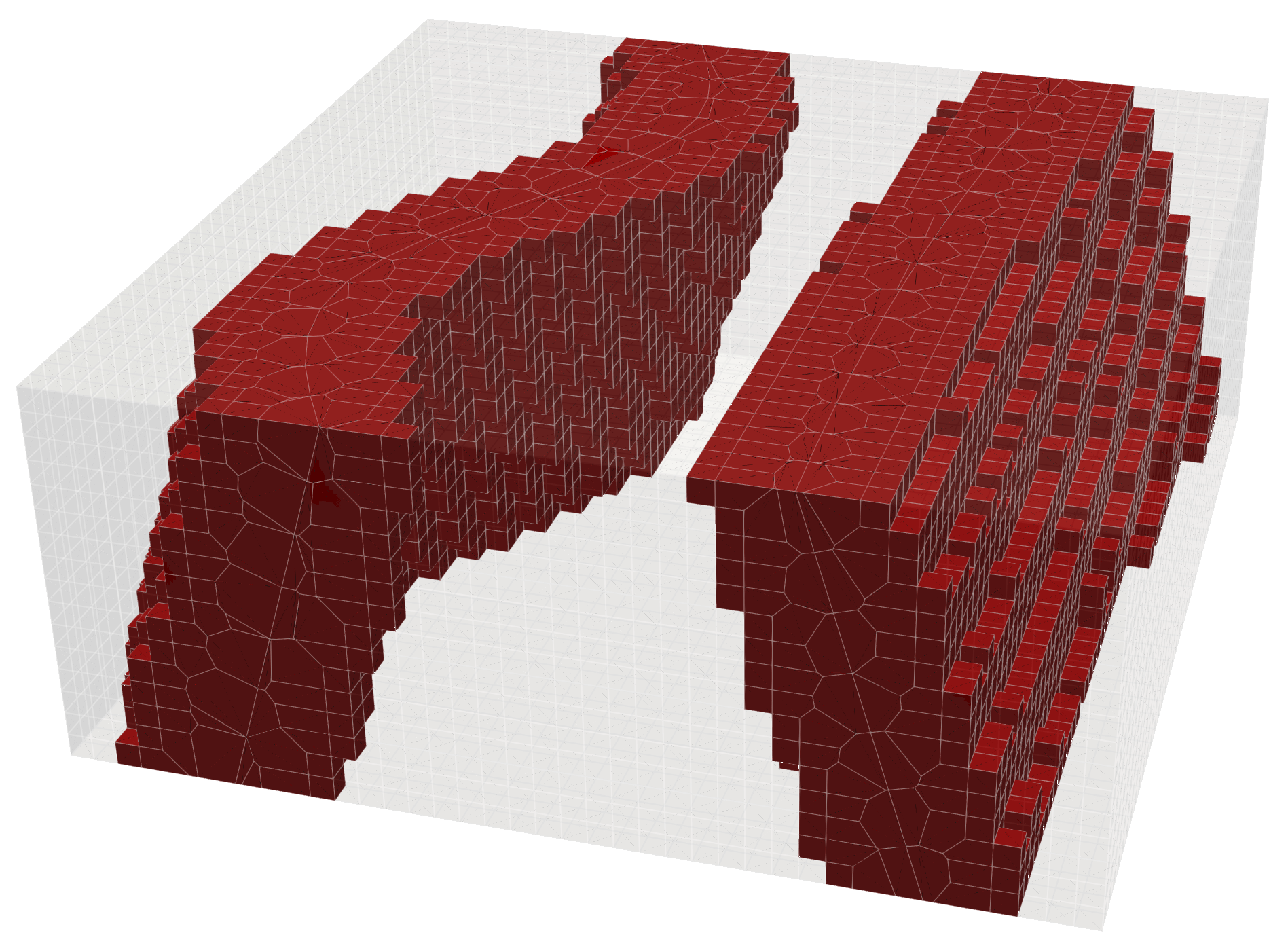}
    \\
    \small (c) $\tau = 10^{-8}$
\end{minipage}

\caption{
Evolution of cell classification for the layered isotropic permeability case in the two-fault model. The MFD-activated regions (shown in red) become increasingly concentrated around the fault surfaces as the tolerance $\tau$ decreases.
}
\label{fig:layered_homogeneous_fault_cells}
\end{figure}

\begin{table}[H]
\centering
\scriptsize
\setlength{\tabcolsep}{4pt}

\begin{tabular}{c|ccccccc}
\hline
\textbf{Metric}
& \textbf{Full TPFA}
& $\mathbf{10^{0}}$
& $\mathbf{10^{-1}}$
& $\mathbf{10^{-2}}$
& $\mathbf{10^{-4}}$
& $\mathbf{10^{-6}}$
& $\mathbf{10^{-8}}$ \\
\hline

TPFA cells
& 22056
& 21895
& 18398
& 16511
& 16068
& 16055
& 16054 \\

Sparsity reduction (\%)
& 92.31
& 89.56
& 52.80
& 47.07
& 45.82
& 45.79
& 45.78 \\

Relative flux error
& $7.31 \times 10^{-3}$
& $7.27 \times 10^{-3}$
& $1.07 \times 10^{-3}$
& $9.04 \times 10^{-14}$
& $7.44 \times 10^{-14}$
& $5.13 \times 10^{-14}$
& $4.10 \times 10^{-14}$ \\

Relative saturation error
& $1.46 \times 10^{-3}$
& $1.34 \times 10^{-3}$
& $2.25 \times 10^{-5}$
& $2.33 \times 10^{-14}$
& $1.98 \times 10^{-14}$
& $3.07 \times 10^{-14}$
& $1.09 \times 10^{-14}$ \\

\hline
\end{tabular}
\caption{
Tolerance-dependent TPFA cell count, sparsity reduction, and relative errors for the layered isotropic permeability case. The full MFD solution is used as the reference solution for the error calculations.
}
\label{tab:layered_isotropic_results}
\end{table}

Table~\ref{tab:layered_isotropic_results} quantifies the impact of the residual-based classification on both accuracy and operator complexity. A notable transition occurs near $\tau = 10^{-2}$, where the flux and saturation errors drop abruptly to the $10^{-14}$ level, indicating that the dominant TPFA-incompatible regions have already identified by the indicator. At the same tolerance, the adaptive discretization retains approximately $47\%$ sparsity reduction, demonstrating that a substantial fraction of the full MFD can be removed while preserving a solution that is virtually identical to the full MFD reference. 

\begin{figure}[H]
\centering

\begin{minipage}{0.49\textwidth}
    \centering
    \includegraphics[width=\linewidth]{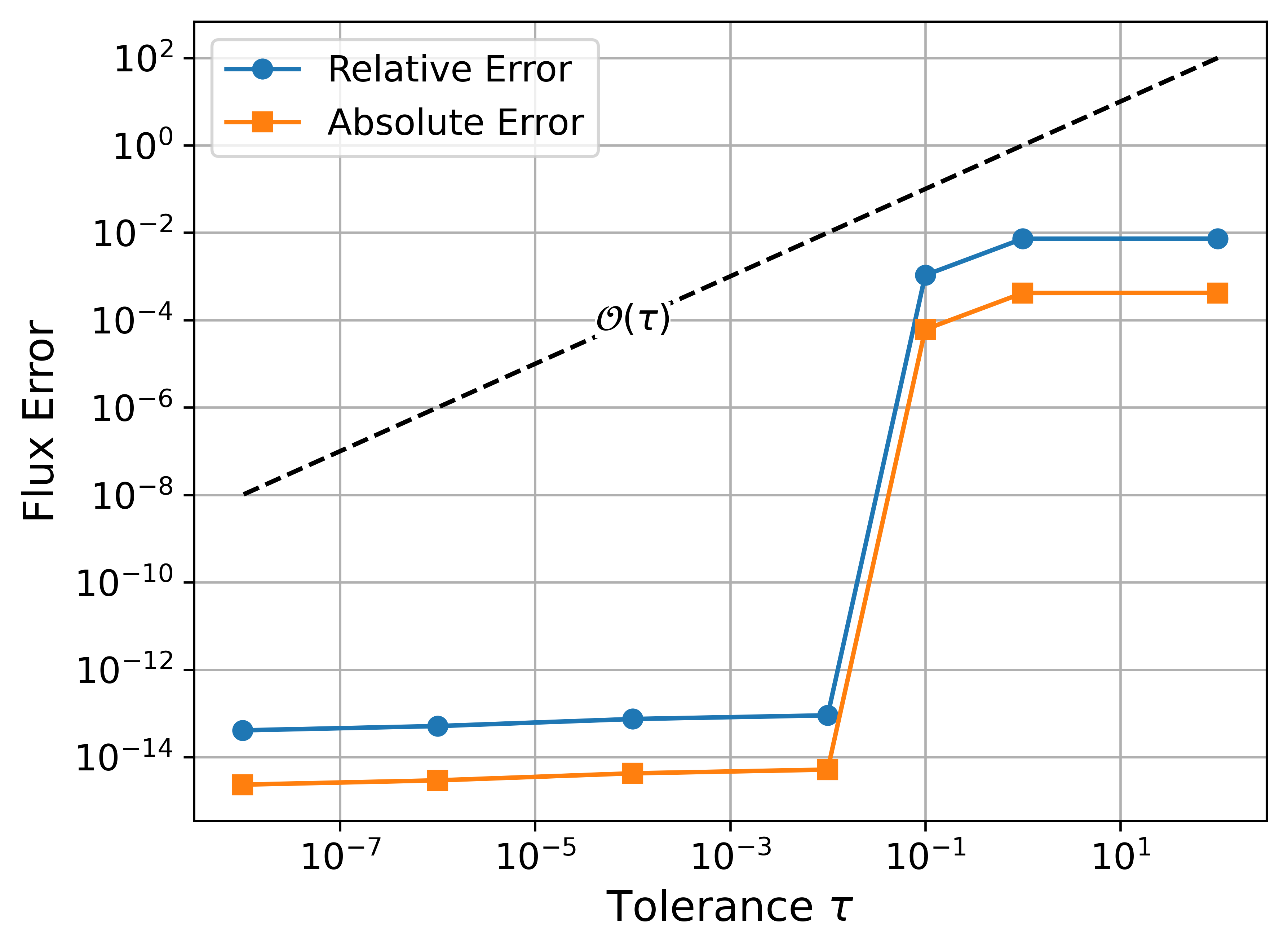}
    \\
    \small (a) Mass flux error
\end{minipage}
\hfill
\begin{minipage}{0.49\textwidth}
    \centering
    \includegraphics[width=\linewidth]{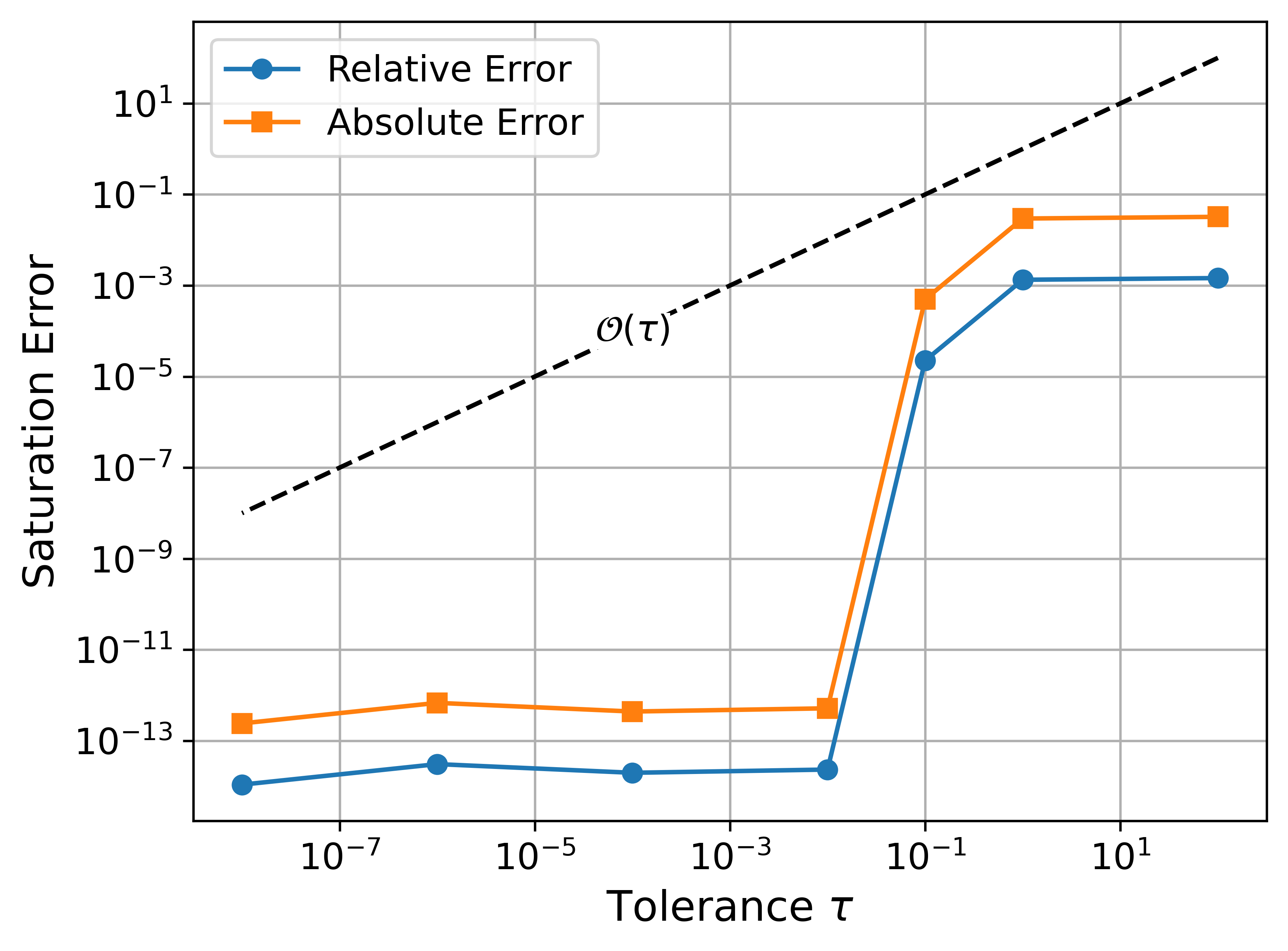}
    \\
    \small (b) Saturation error
\end{minipage}

\caption{
Relative flux and saturation errors as functions of the tolerance for the layered isotropic permeability case. The dashed diagonal line represents the prescribed tolerance $\tau$.
}
\label{fig:heterogeneous_error_plots}
\end{figure}

As expected, the flux error produced by the adaptive scheme remains bounded by the prescribed tolerance parameter, as shown in Figure~\ref{fig:heterogeneous_error_plots}. A pronounced transition occurs near $\tau = 10^{-2}$, consistent with the behavior observed in Table~\ref{tab:layered_isotropic_results}. For smaller tolerance values, only a modest number of additional cells are classified as MFD-compatible, resulting in slightly smaller errors. However, these refinement do not change the overall error magnitude and the error curves rapidly approach a numerical plateau. Beyond $\tau = 10^{-6}$, the cell classification remains essentially unchanged, suggesting that all regions requiring the full MFD stencil have already been identified. In addition, a similar trend is observed in the saturation errors. Both the transition near $\tau = 10^{-2}$ and the subsequent numerical plateau closely mirror the behavior of the flux errors. This agreement highlights the strong dependence of the transport solution on the underlying numerical flux field, indicating that improvement in flux accuracy are directly reflected in the quality of the computed saturation solution.

\subsubsection{Layered Heterogeneous Anisotropic Permeability Case} 
For the second permeability configuration, we assign a layered heterogeneous anisotropic permeability field to the same two-fault model. The domain is also divided into three horizontal layers according to the cell-center depth coordinate $z$, with base permeability values $k_{\text{base}} = 4,1$ and $0.1$ in the upper, middle and lower layers, respectively. Within each layer, a cell-wise log-normal perturbation is introduced through the principal permeability component $k_{xx}$. Specifically, for each cell,
\[
k_{xx} = k_{\text{base}}(z)\exp(0.1 \xi), 
\quad
\xi \sim \mathcal{N}(0,1),
\]
and the remaining diagonal components are defined by fixed anisotropy ratios
\[
k_{yy} = 0.5 k_{xx}, \quad
k_{zz} = 0.2 k_{xx}.
\]
Hence, the cell-wise permeability tensor is given by
\[
\diffK = 
\begin{bmatrix}
    k_{xx} & 0 & 0 \\
    0 & k_{yy} & 0 \\
    0 & 0 & k_{zz}
\end{bmatrix}.
\]
This construction preserves the same layered background structure as the isotropic case, while introducing cell-to-cell heterogeneity and directional anisotropy within each layer.

\begin{figure}[H]
\centering

\begin{minipage}{0.32\textwidth}
    \centering
    \includegraphics[width=\linewidth]{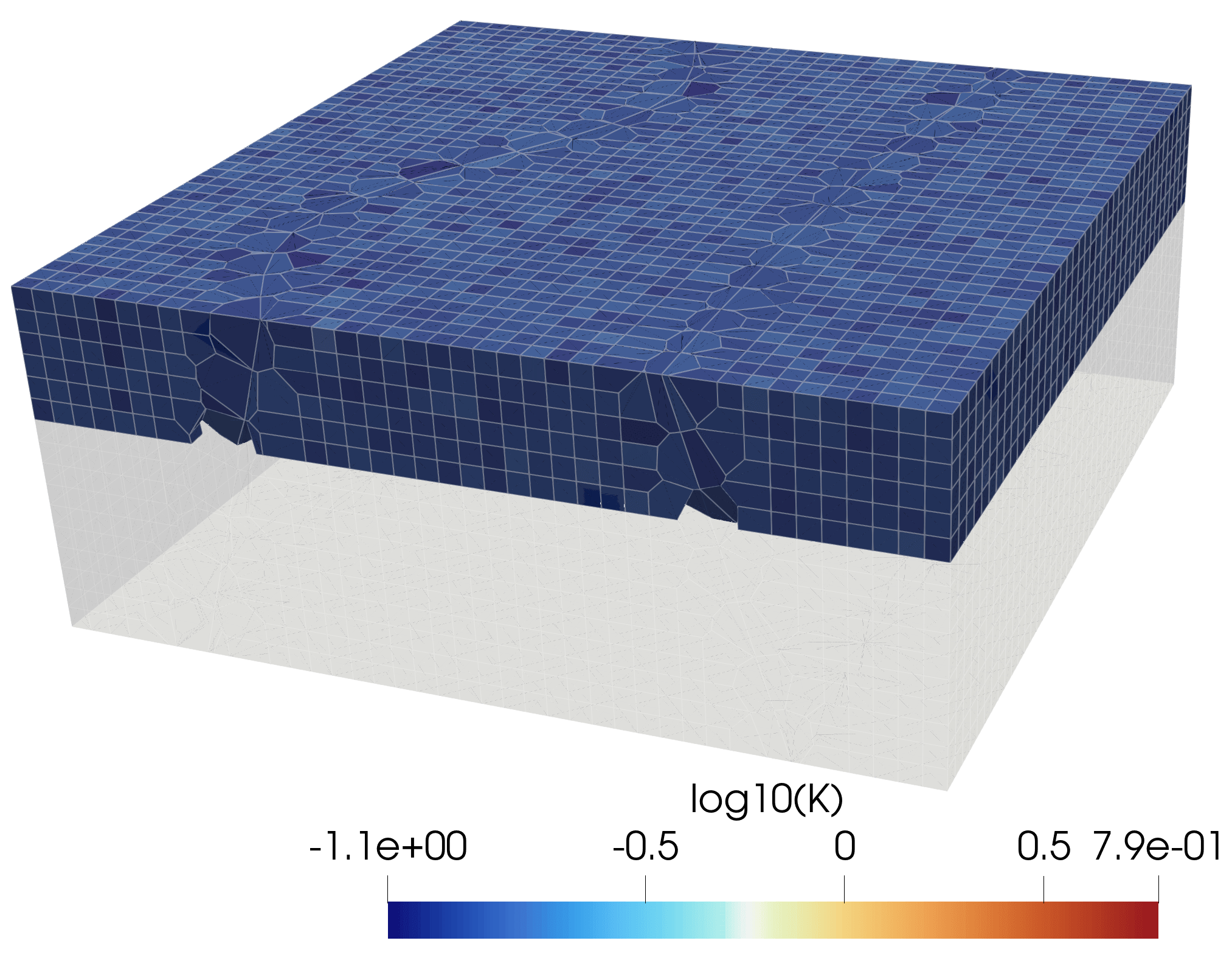}
    \\
    \small (a) Top layer: $k_\text{base} = 4$
\end{minipage}
\hfill
\begin{minipage}{0.32\textwidth}
    \centering
    \includegraphics[width=\linewidth]{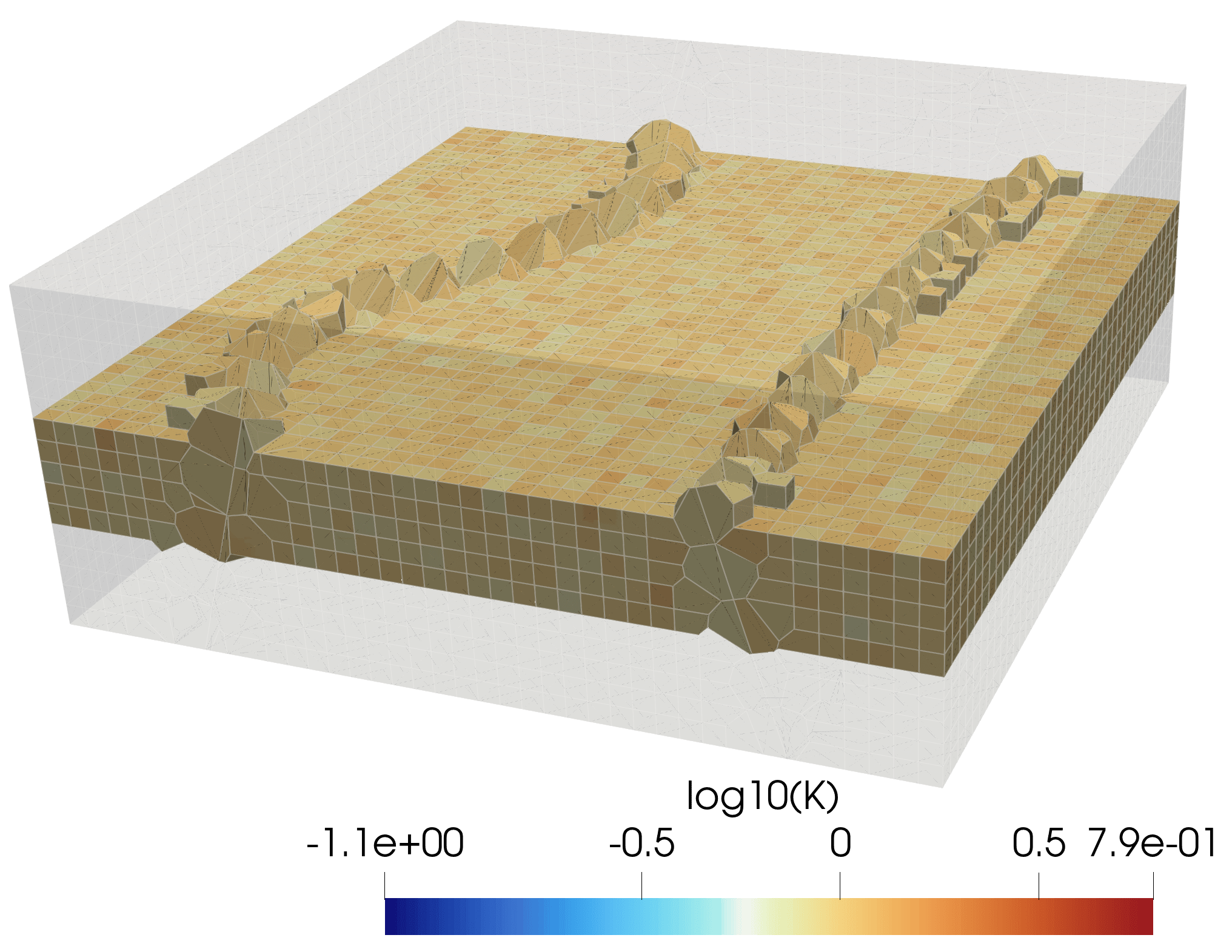}
    \\
    \small (b) Middle layer: $k_\text{base} = 1$
\end{minipage}
\hfill
\begin{minipage}{0.32\textwidth}
    \centering
    \includegraphics[width=\linewidth]{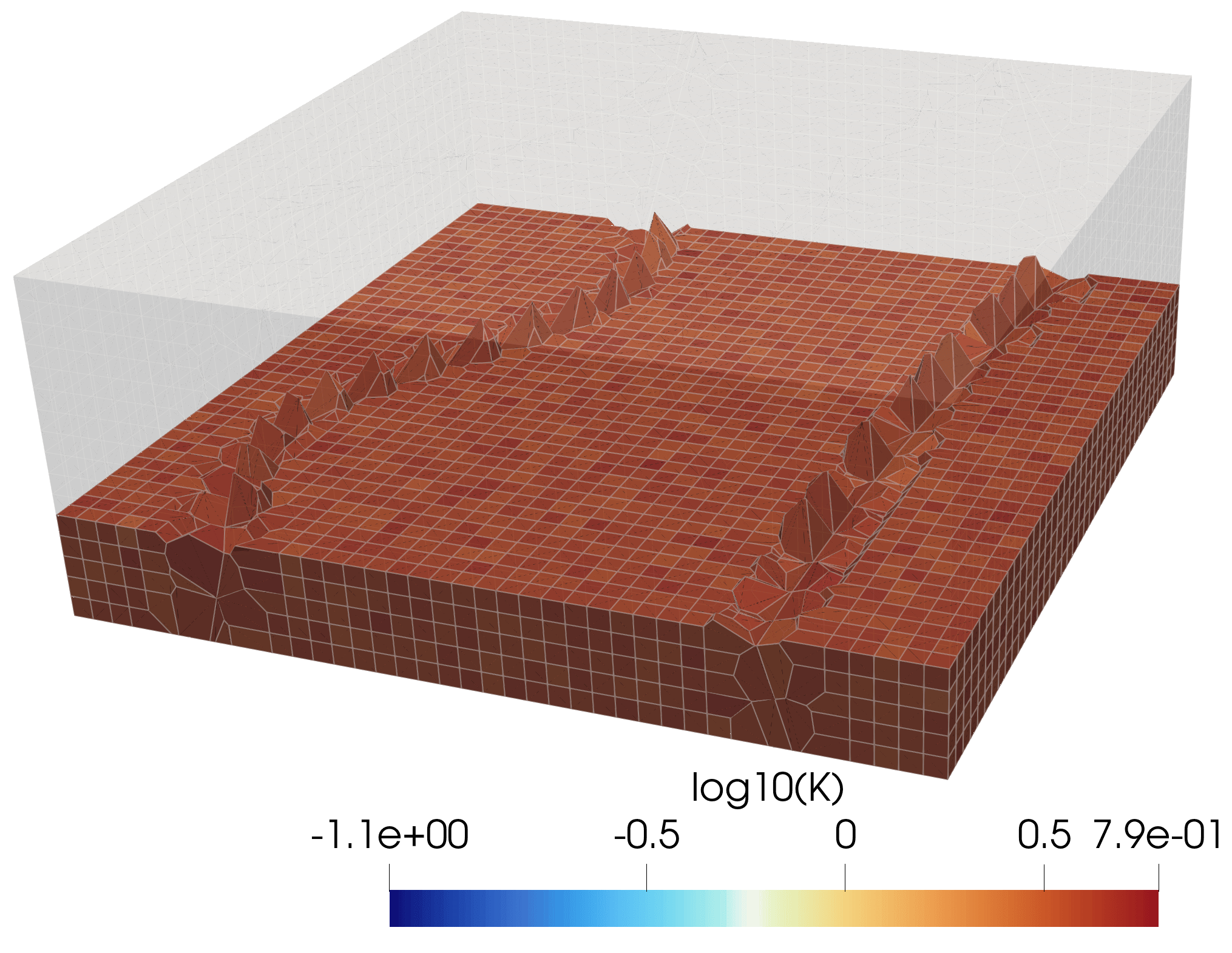}
    \\
    \small (c) Bottom layer: $k_\text{base} = 0.1$
\end{minipage}
\caption{
Three-layer heterogeneous anisotropic permeability field used in the 3D two-fault benchmark. Colors indicate $\log_{10}(k_{xx})$, where $k_{xx}$ denotes the perturbed principal permeability component.
}
\label{fig:layered_perm2}
\end{figure}


We then visualize the evolution of the adaptive cell classification for the layered heterogeneous anisotropic setting described in Figure~\ref{fig:layered_anisotropic_fault_cells}. By visual inspection, the resulting classification pattern is similar to that observed in the layered isotropic case. The MFD cells continue to concentrate around the fault regions where geometric distortions are most pronounced, while the Cartesian bulk cells away from the faults remain classified as TPFA cells.

\begin{figure}[H]
\centering
\begin{minipage}{0.32\textwidth}
    \centering
    \includegraphics[width=\linewidth]{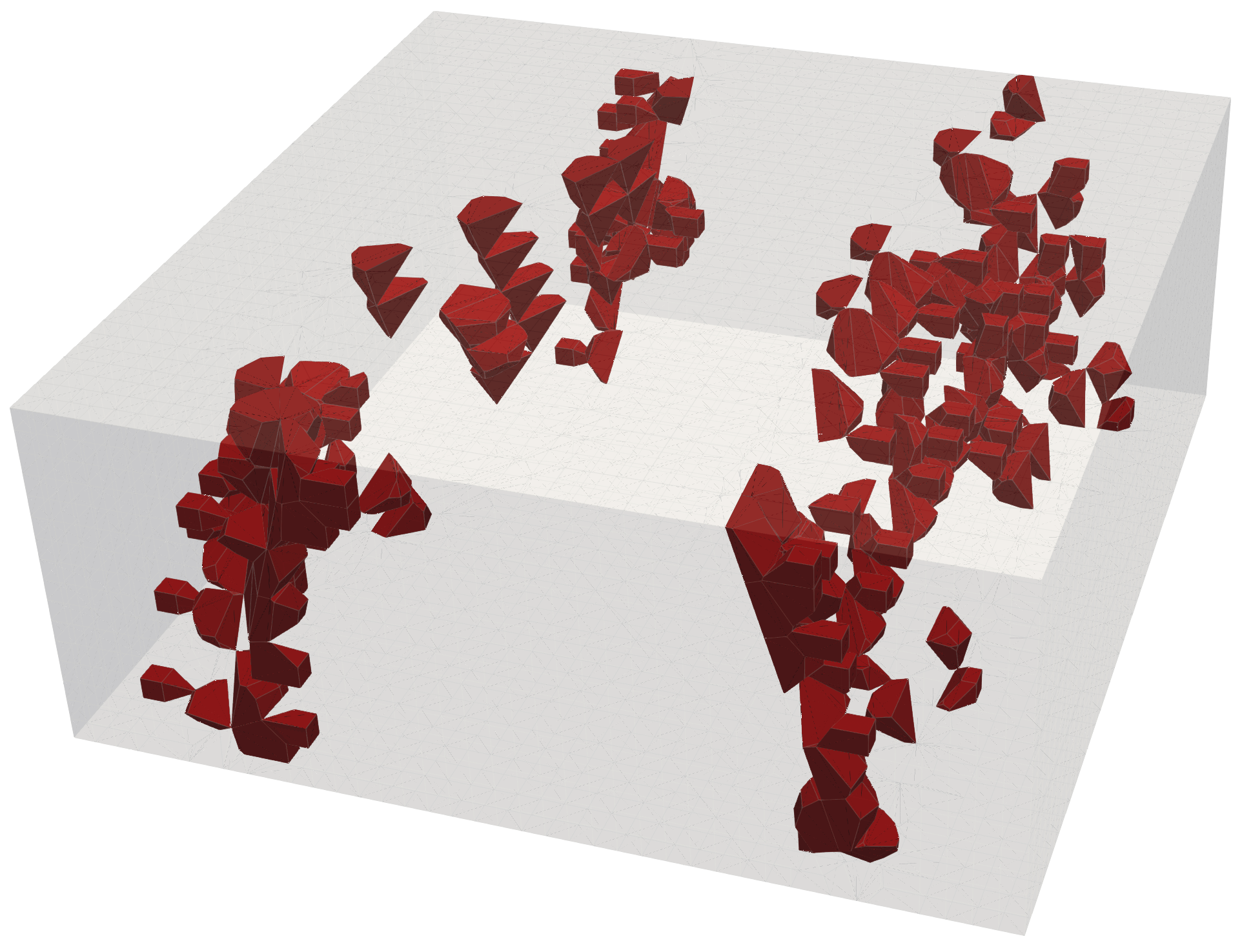}
    \\
    \small (a) $\tau = 10^{0}$
\end{minipage}
\hspace{0.09\textwidth}
\begin{minipage}{0.32\textwidth}
    \centering
    \includegraphics[width=\linewidth]{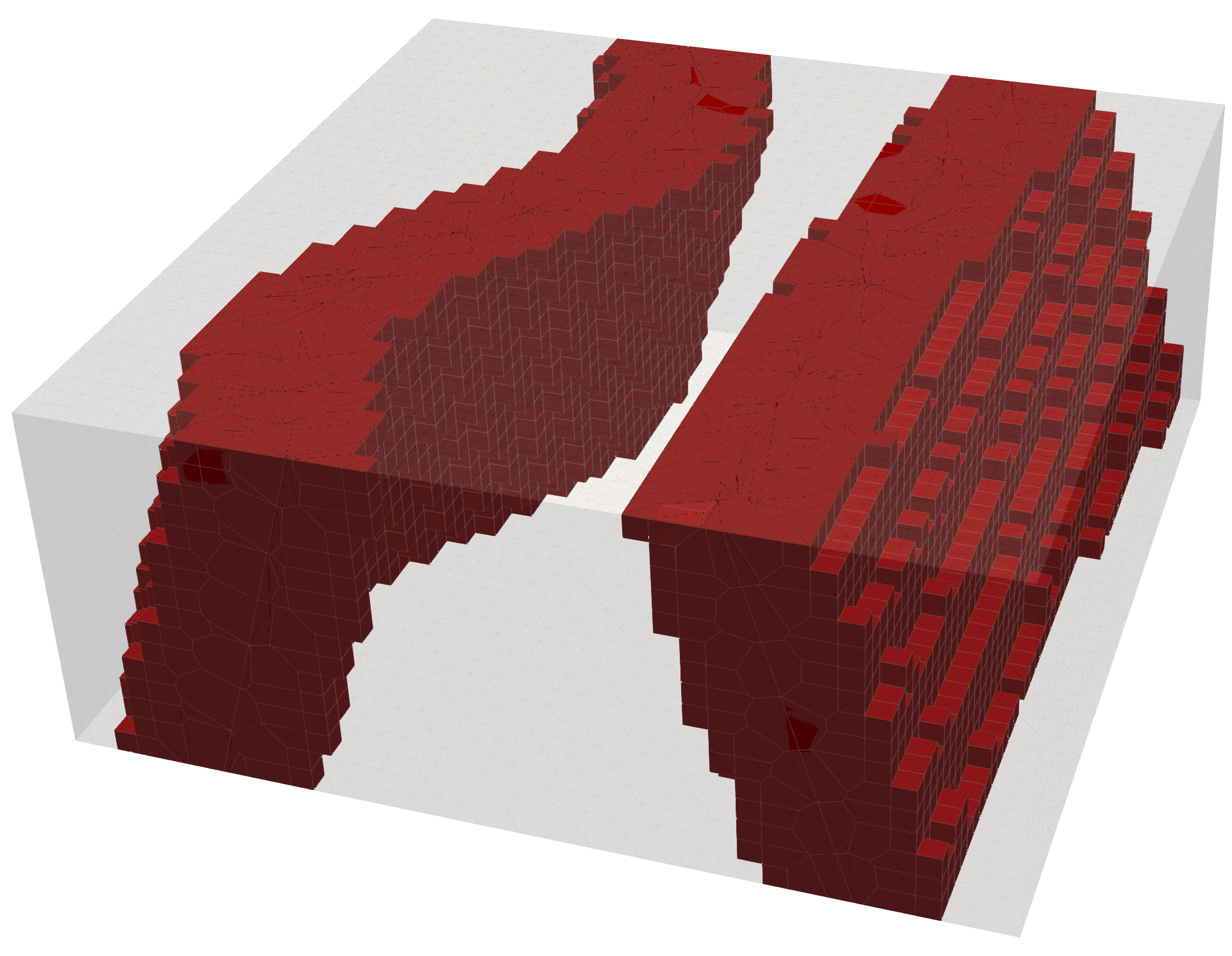}
    \\
    \small (b) $\tau = 10^{-8}$
\end{minipage}
\caption{
Evolution of the adaptive cell classification for the layered heterogeneous anisotropic permeability case. Red cells denote regions where the MFD discretization is activated.
}
\label{fig:layered_anisotropic_fault_cells}
\end{figure}

The quantitative results in Table~\ref{tab:layered_anisotropic_results} reveal differences that are not immediately apparent from the classification plots. Although the overall classification pattern remains similar to that of the isotropic case, the anisotropic configuration consistently activates a slightly larger number of MFD-compatible cells for a given tolerance value. Nevertheless, the TPFA cell count again approaches a stable value as the tolerance is reduced, indicating that the adaptive partition has effectively converged. The relative flux and saturation errors are also slightly larger than those observed in the previous case. Despite these differences, the characteristic transition near $\tau = 10^{-2}$ and the subsequent convergence behavior remain unchanged. Figure~\ref{fig:anisotropic_error_plots} further confirms that the flux errors remain bounded by the prescribed tolerance.

\begin{table}[H]
\centering
\scriptsize
\setlength{\tabcolsep}{4pt}

\begin{tabular}{c|ccccccc}
\hline
\textbf{Metric}
& \textbf{Full TPFA}
& $\mathbf{10^{0}}$
& $\mathbf{10^{-1}}$
& $\mathbf{10^{-2}}$
& $\mathbf{10^{-4}}$
& $\mathbf{10^{-6}}$
& $\mathbf{10^{-8}}$ \\
\hline

TPFA cells
& 22056
& 21859
& 18269
& 16529
& 16059
& 16050
& 16050 \\

Sparsity reduction (\%)
& 92.30
& 88.94
& 52.07
& 47.09
& 45.76
& 45.74
& 45.74 \\

Relative flux error
& $2.59 \times 10^{-2}$
& $2.16 \times 10^{-2}$
& $4.97 \times 10^{-4}$
& $4.31 \times 10^{-13}$
& $1.17 \times 10^{-12}$
& $8.93 \times 10^{-13}$
& $8.93 \times 10^{-13}$ \\

Relative saturation error
& $2.90 \times 10^{-3}$
& $2.64 \times 10^{-3}$
& $2.36 \times 10^{-6}$
& $3.88 \times 10^{-13}$
& $1.05 \times 10^{-12}$
& $8.01 \times 10^{-13}$
& $8.01 \times 10^{-13}$ \\

\hline
\end{tabular}

\caption{
Tolerance-dependent TPFA cell count, sparsity reduction, and relative errors for the layered heterogeneous anisotropic permeability case.
}
\label{tab:layered_anisotropic_results}
\end{table}

These observations provide useful insight into the behavior of the proposed residual indicator. Since the projected flux construction incorporates the local permeability tensor through the harmonic face permeability tensor $\kappa_\fc$, the classification criterion naturally accounts for both heterogeneity and anisotropy. As a result, small differences in TPFA cell counts and error magnitudes are observed relative to the isotropic case, indicating that the additional permeability information is reflected in the classification process. Beyond these differences, the overall behavior suggests that geometric complexity plays a much larger role than permeability variations in determining where MFD discretization is required.

\begin{figure}[H]
\centering

\begin{minipage}{0.49\textwidth}
    \centering
    \includegraphics[width=\linewidth]{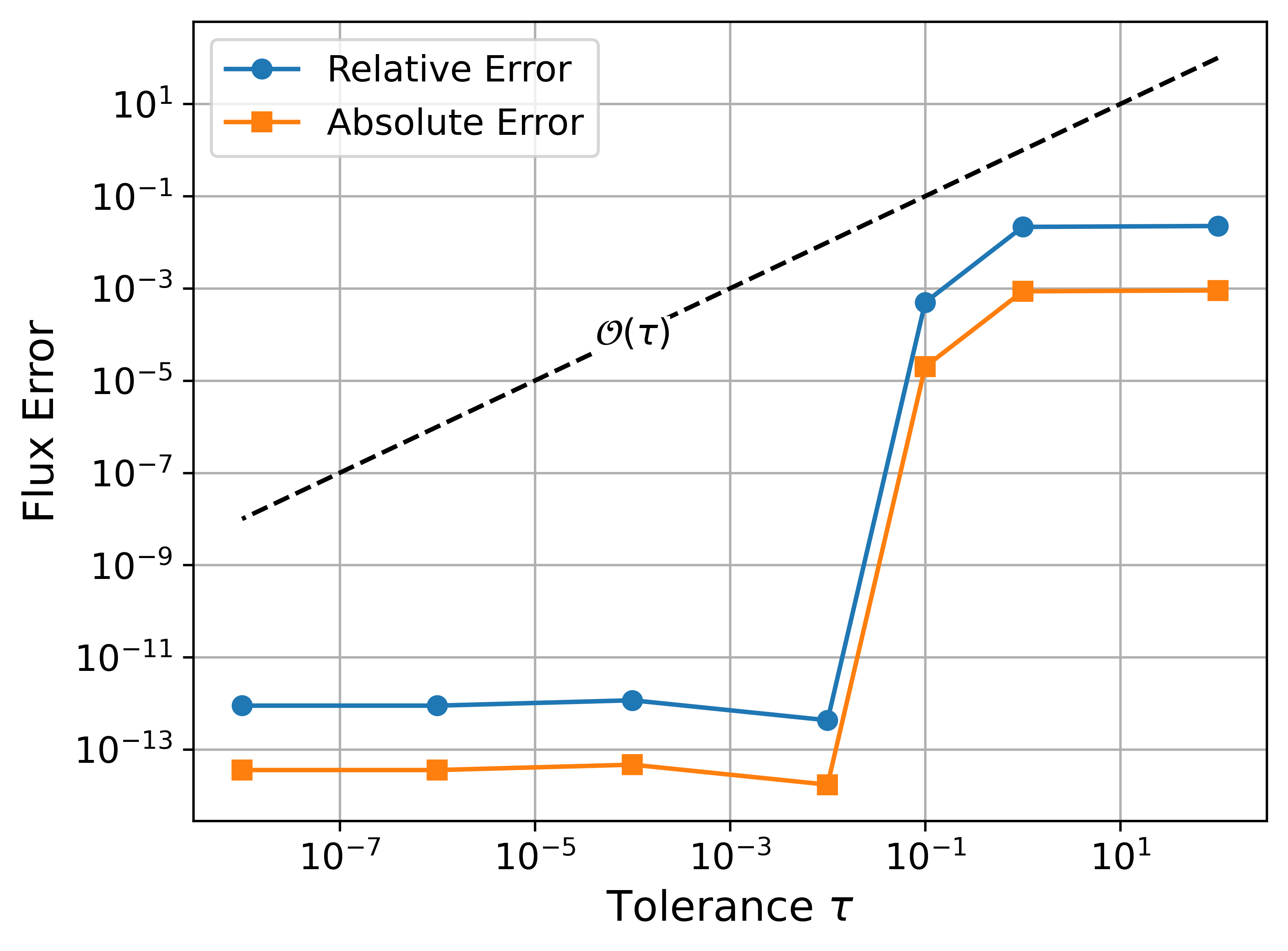}
    \\
    \small (a) Mass flux error
\end{minipage}
\hfill
\begin{minipage}{0.49\textwidth}
    \centering
    \includegraphics[width=\linewidth]{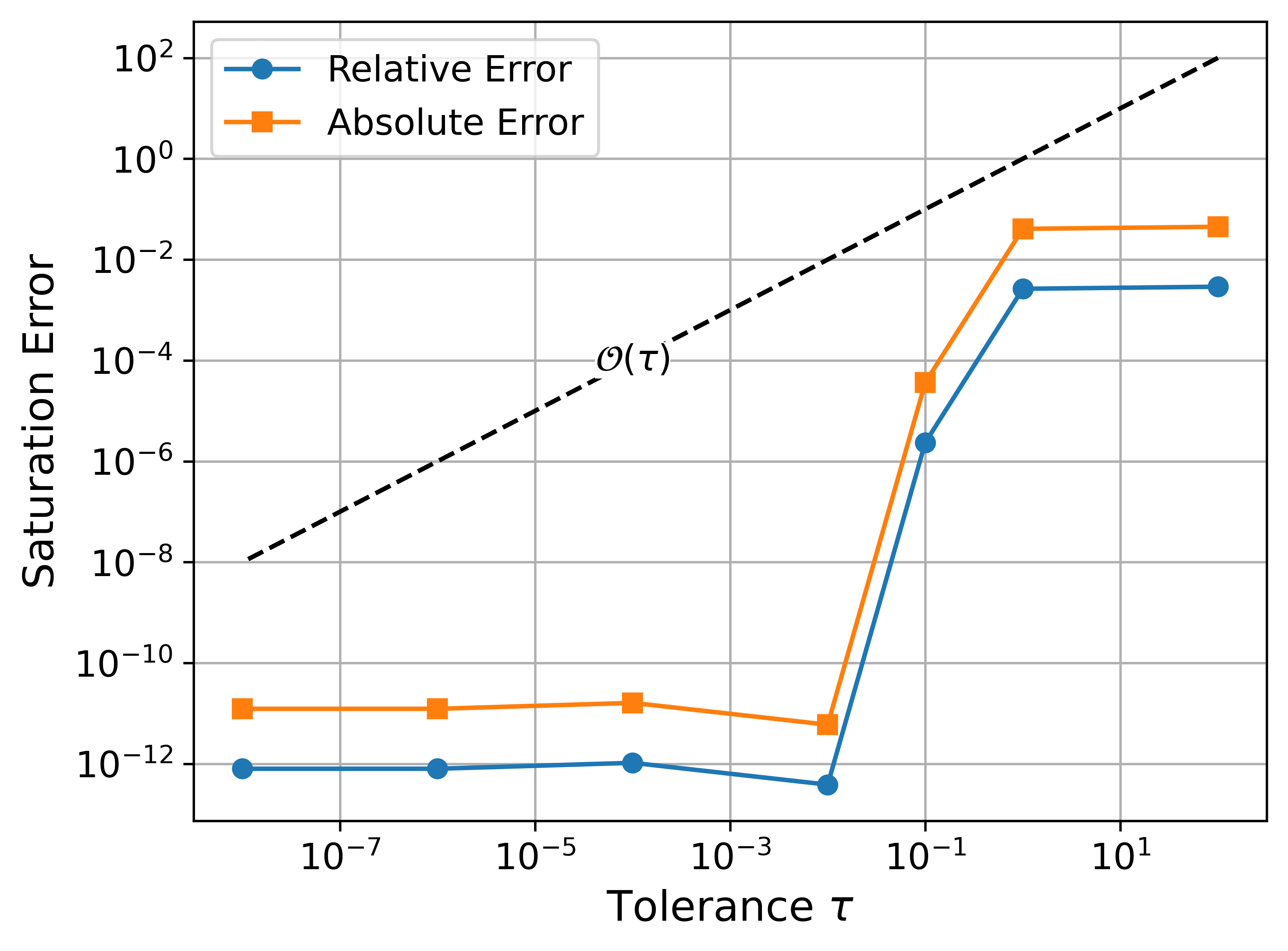}
    \\
    \small (b) Saturation error
\end{minipage}

\caption{
Relative and absolute flux and saturation errors as functions of the tolerance for the layered heterogeneous anisotropic permeability case.
}
\label{fig:anisotropic_error_plots}
\end{figure}

\subsection{Fully Polyhedral Case}
We finally consider a fully polyhedral benchmark case. The computational domain is
\[
\Omega = [0,L_x] \times [0, L_y] \times [0,L_z] =
[0,2] \times [0,1] \times [0,0.5],
\]
and consists of 14071 cells, 90100 faces and 65171 vertices, as shown in Figure~\ref{fig:fully_polyhedral_mesh}. This final case is entirely composed of irregular polyhedral cells. A homogeneous isotropic permeability tensor $\diffK = \mathbf{I}$ is used throughout the domain.

\begin{figure}[H]
\centering
\includegraphics[width=0.76\textwidth]{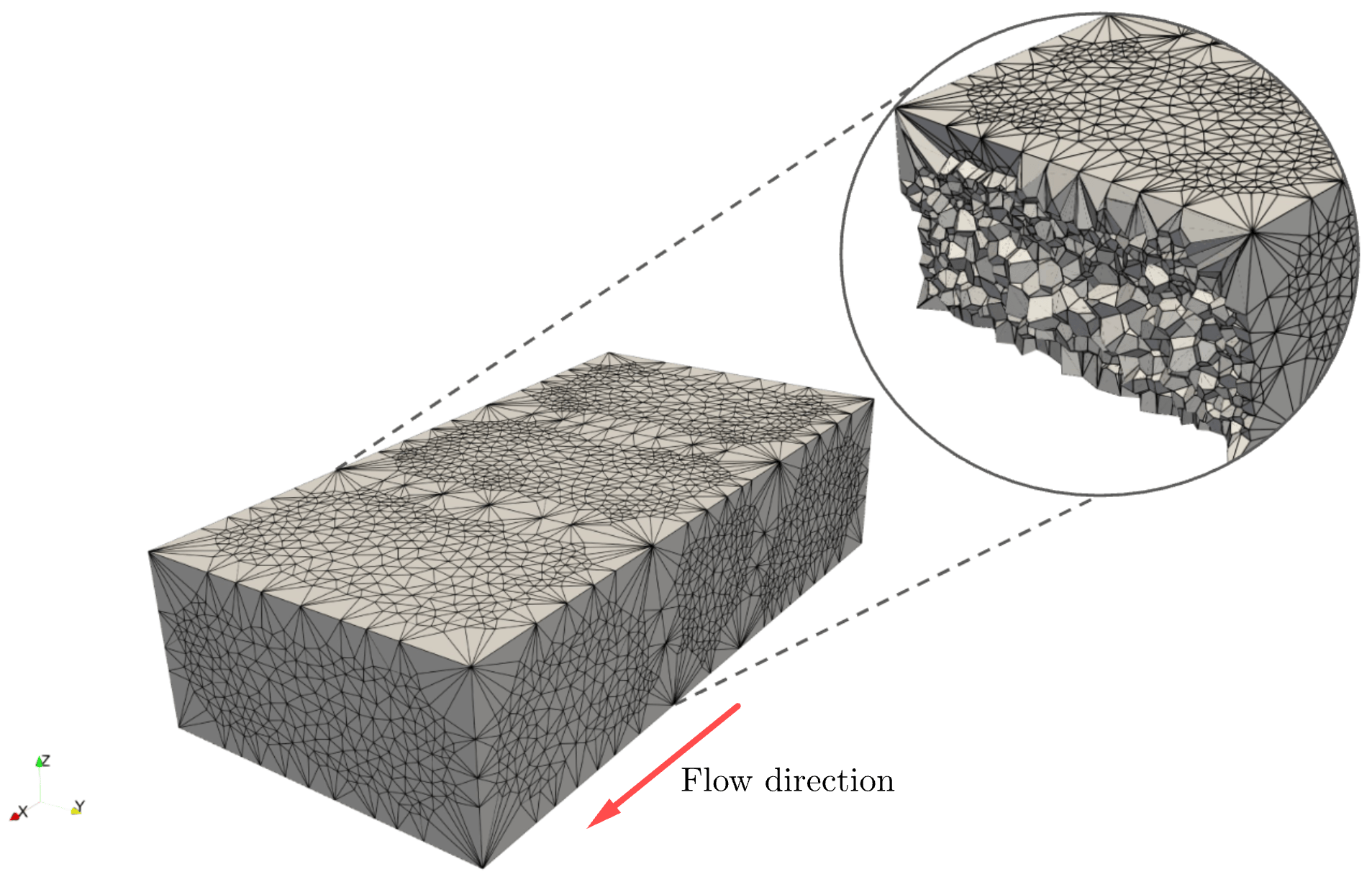}
\caption{
Fully polyhedral benchmark mesh generated by Vorocrust~\protect\cite{vorocrust}. The magnified view highlights the interior polyhedral cell structure, showing that the domain is composed of complex polyhedral elements throughout the volume.
}
\label{fig:fully_polyhedral_mesh}
\end{figure}

To generate a left-to-right flow configuration, pressure boundary conditions are prescribed on the left and right boundaries of the domain with homogeneous no-flow conditions imposed elsewhere. We use the linear pressure field
\[
p_\lin(x,y,z) = 1 - \frac{x}{L_x}
\]
for cell classification which induces a uniform flow in the $x$-direction. The corresponding projected flux is generated in the same manner as in the previous benchmarks. The computed numerical fluxes are subsequently used to evolve the saturation field using the same fully implicit finite-volume upwind discretization employed in the previous cases. This configuration produces a shock-like saturation front propagating across the fully polyhedral mesh. Relative flux errors are evaluated with respect to the projected flux field and saturation errors are computed relative to the full MFD solution.

\begin{figure}[H]
\centering

\begin{minipage}{0.32\textwidth}
    \centering
    \includegraphics[width=\linewidth]{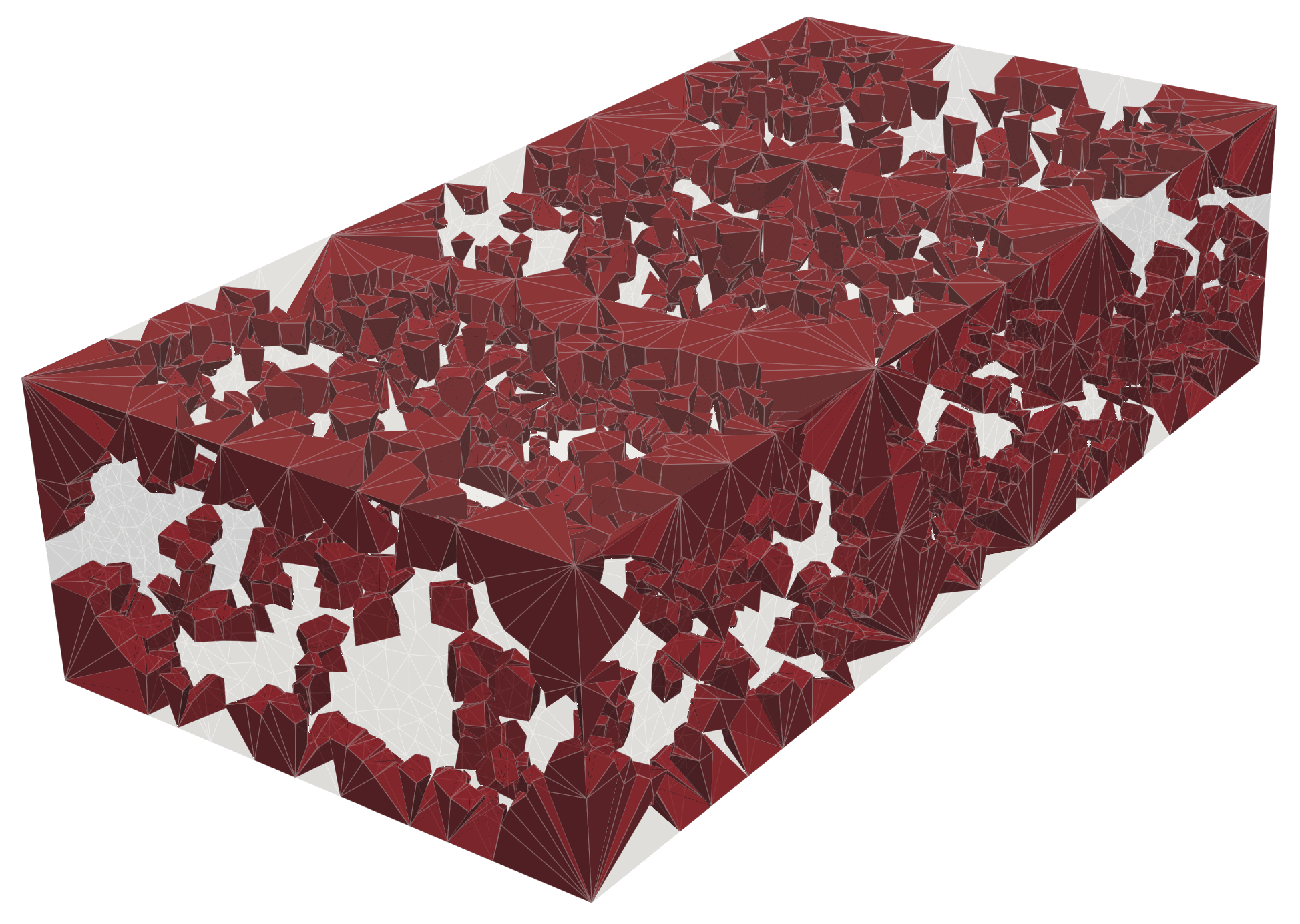}
    \\
    \small (a) $\tau = 10^{0}$
\end{minipage}
\hfill
\begin{minipage}{0.32\textwidth}
    \centering
    \includegraphics[width=\linewidth]{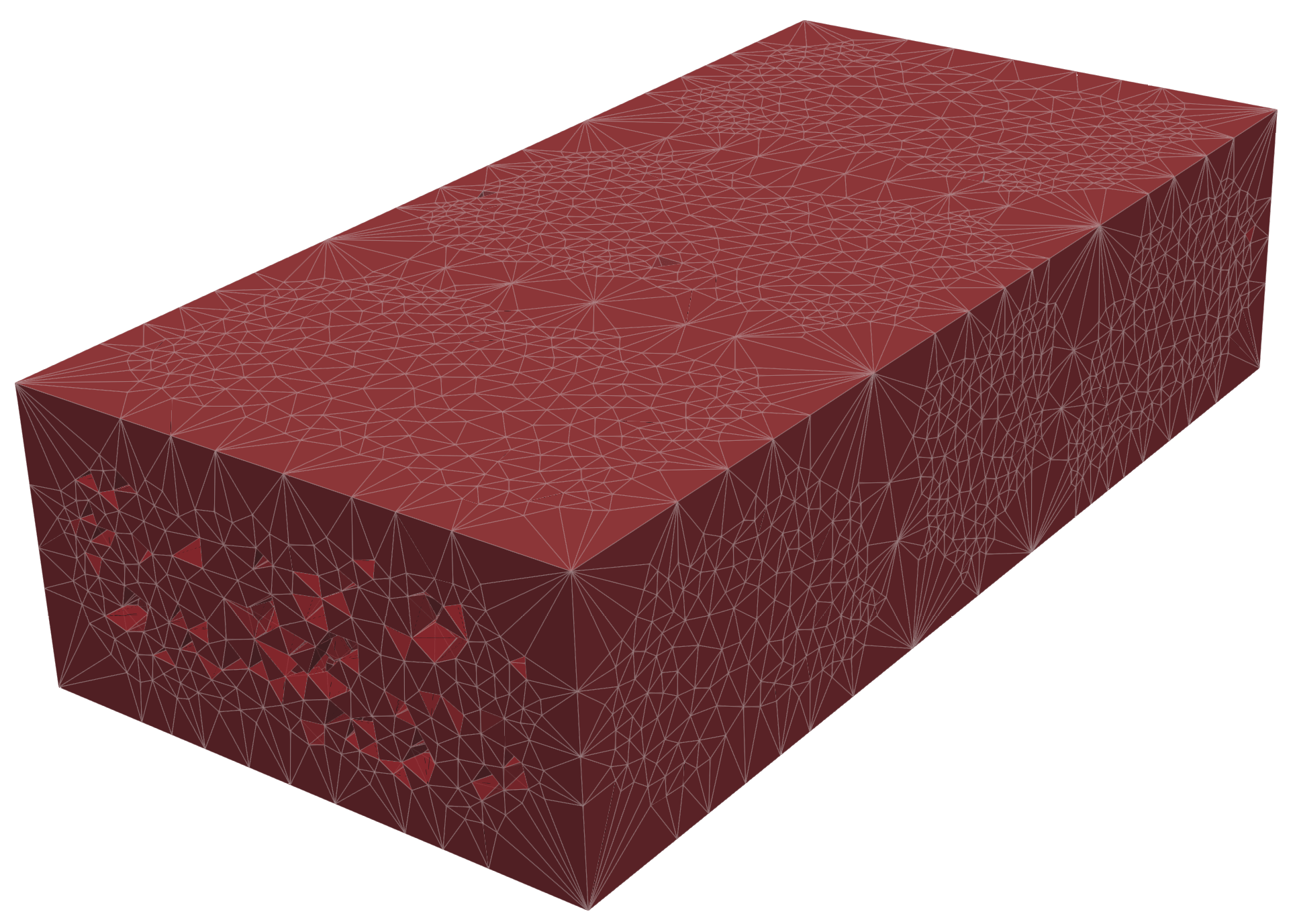}
    \\
    \small (b) $\tau = 10^{-1}$
\end{minipage}
\hfill
\begin{minipage}{0.32\textwidth}
    \centering
    \includegraphics[width=\linewidth]{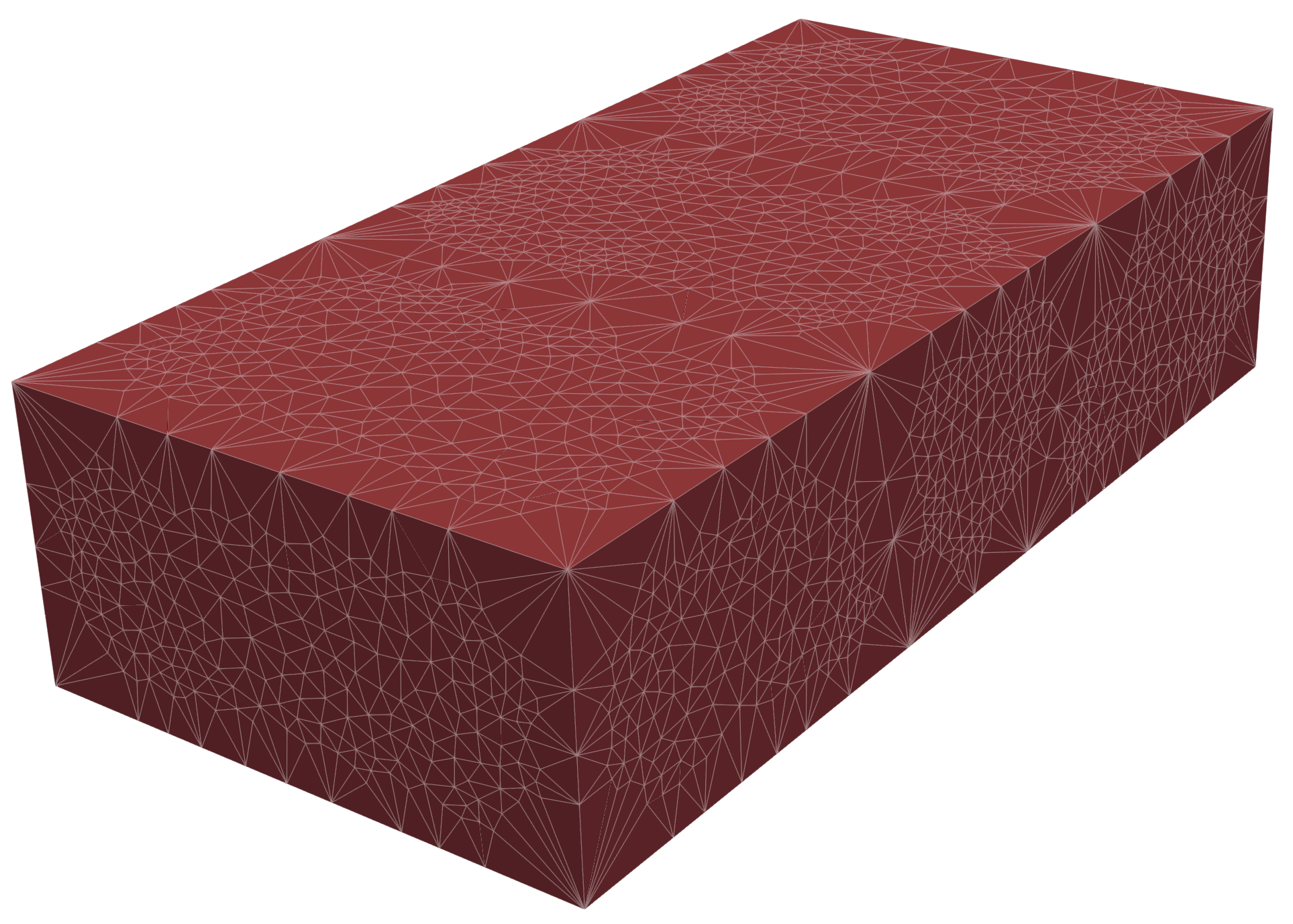}
    \\
    \small (c) $\tau = 10^{-2}$
\end{minipage}

\caption{
Evolution of the adaptive cell classification for the fully polyhedral case. 
Red cells also denote MFD cells selected by the residual-based indicator. 
}
\label{fig:polyhedral_cell_evolution}
\end{figure}

Given the highly irregular polyhedral cell geometry, significant violations of TPFA consistency are expected throughout the domain. The classification patterns shown in Figure~\ref{fig:polyhedral_cell_evolution} are fully consistent with this expectation with the residual-based indicator rapidly activating MFD discretization across a large portion of the mesh. As quantified in Table~\ref{tab:vorocrust_results}, only a small fraction of cells remain classified as TPFA-compatible at $\tau = 10^{-1}$. Further tightening of the tolerance leads to a complete transition to the full MFD discretization with all cells classified as MFD-compatible by $\tau = 10^{-2}$.

\begin{table}[H]
\centering
\scriptsize
\setlength{\tabcolsep}{4pt}

\begin{tabular}{c|cccccc}
\hline
\textbf{Metric}
& \textbf{Full TPFA}
& $\mathbf{10^{0}}$
& $\mathbf{10^{-1}}$
& $\mathbf{10^{-2}}$
& $\mathbf{10^{-3}}$
& $\mathbf{10^{-4}}$ \\
\hline

TPFA cells
& 14071
& 12246
& 1091
& 0
& 0
& 0 \\

Sparsity reduction (\%)
& 96.12
& 85.31
& 8.44
& 0.00
& 0.00
& 0.00 \\

Relative flux error
& $2.20 \times 10^{-1}$
& $1.33 \times 10^{-1}$
& $3.33 \times 10^{-2}$
& $3.23 \times 10^{-11}$
& $3.23 \times 10^{-11}$
& $3.23 \times 10^{-11}$ \\

Relative saturation error
& $1.33 \times 10^{-1}$
& $6.51 \times 10^{-2}$
& $1.63 \times 10^{-3}$
& $0.00$
& $0.00$
& $0.00$ \\

\hline
\end{tabular}

\caption{
Tolerance-dependent TPFA cell count, sparsity reduction, and relative errors for the fully polyhedral case.
}
\label{tab:vorocrust_results}
\end{table}

Beyond confirming the anticipated behavior of the adaptive scheme, this transition provides a quantitative measure of the TPFA consistency level of the fully polyhedral mesh. The fact that a tolerance of approximately $\tau = 10^{-2}$ is sufficient to trigger that TPFA consistency violations are not confined to isolated regions but are instead pervasive throughout the domain. In this sense, the proposed indicator not only identifies where MFD discretization is required but also reveals the characteristic scale of the underlying TPFA inconsistency.

The error behavior for the fully polyhedral benchmark remains broadly consistent with that observed in the previous test cases. Despite the severe TPFA consistency violations present in the mesh, the relative flux errors continue to decrease as the tolerance is tightened and remain bounded by the prescribed tolerance until the adaptive partition converges to the full MFD discretization. A slightly different behavior is observed in the saturation error plots. Although the relative saturation errors closely track the flux errors, the absolute saturation errors occasionally exceed the prescribed tolerance. This is expected, since the theoretical bound applies to the flow problem and not to the diffusive transport discretization. Accordingly, the saturation errors are not required to satisfy the same tolerance-dependent bound as the flux errors.

\begin{figure}[H]
\centering

\begin{minipage}{0.49\textwidth}
    \centering
    \includegraphics[width=\linewidth]{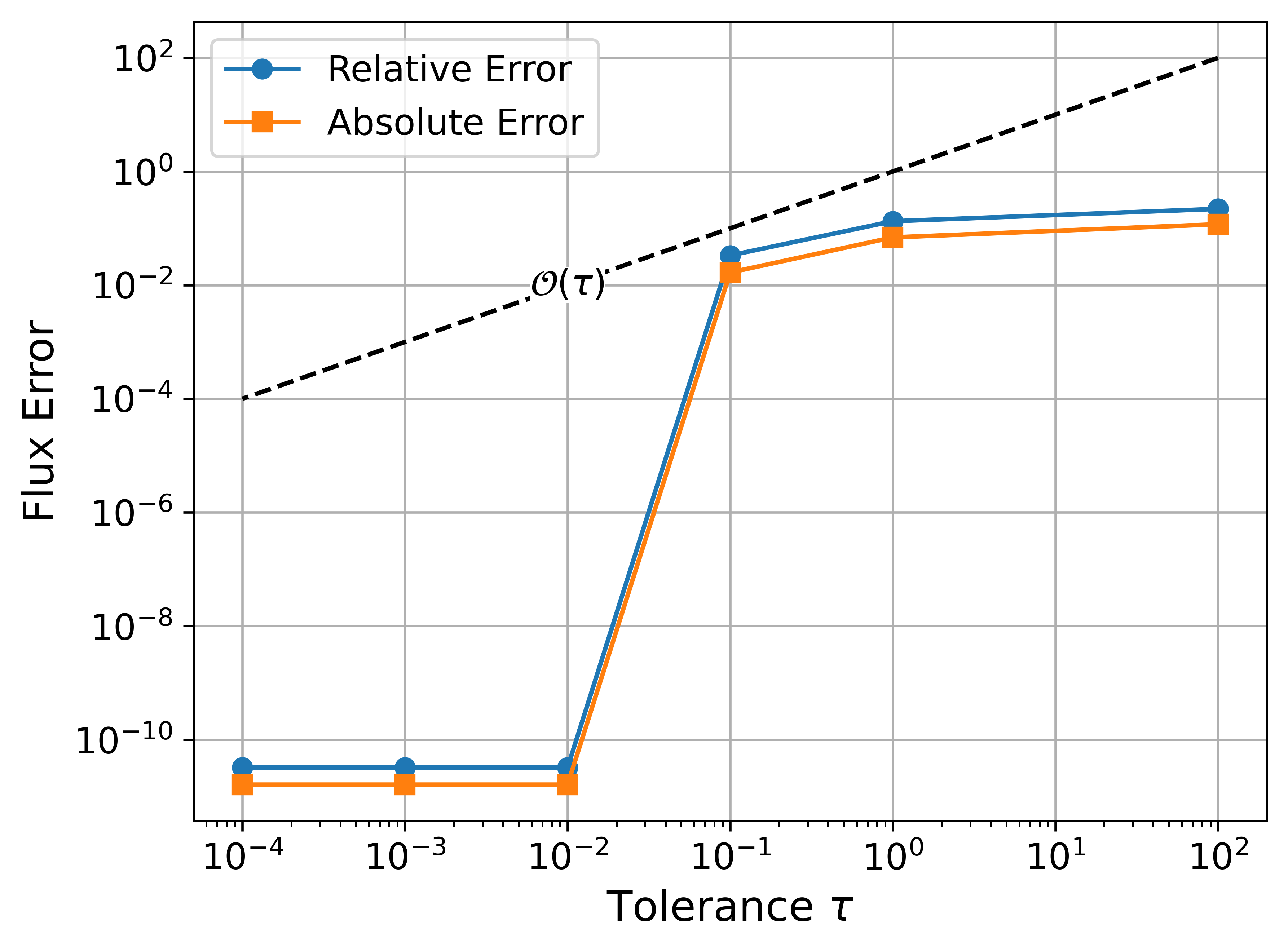}
    \\
    \small (a) Mass flux error
\end{minipage}
\hfill
\begin{minipage}{0.49\textwidth}
    \centering
    \includegraphics[width=\linewidth]{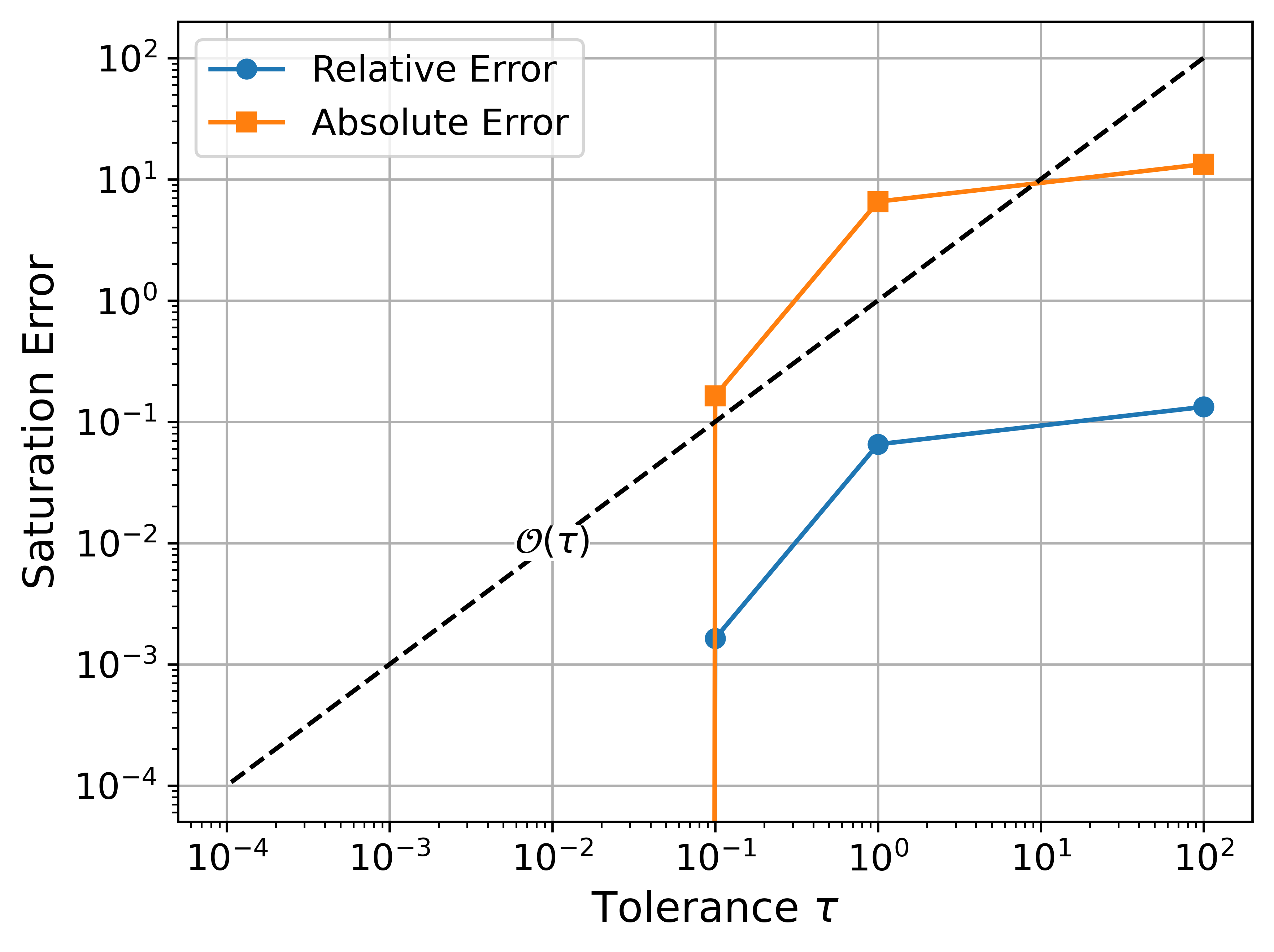}
    \\
    \small (b) Saturation error
\end{minipage}

\caption{
Relative and absolute flux and saturation errors versus tolerance for the fully polyhedral case. For $\tau \le 10^{-2}$, the adaptive scheme coincides with the full MFD discretization, resulting in zero saturation error relative to the full MFD reference solution.
}
\label{fig:polyhedral_error_plots}
\end{figure}

To further examine the practical impact of the adaptive discretization, we additionally visualize the resulting saturation fronts for different tolerance values as shown in Figure~\ref{fig:polyhedral_saturation_fronts}. These results provide a direct illustration of how improvements in flux accuracy translate into improved transport predictions and demonstrate the influence of the adaptive flow discretization on the quality of the computed saturation solution.

\begin{figure}[H]
\centering

\begin{minipage}{0.49\textwidth}
    \centering
    \includegraphics[width=\linewidth]{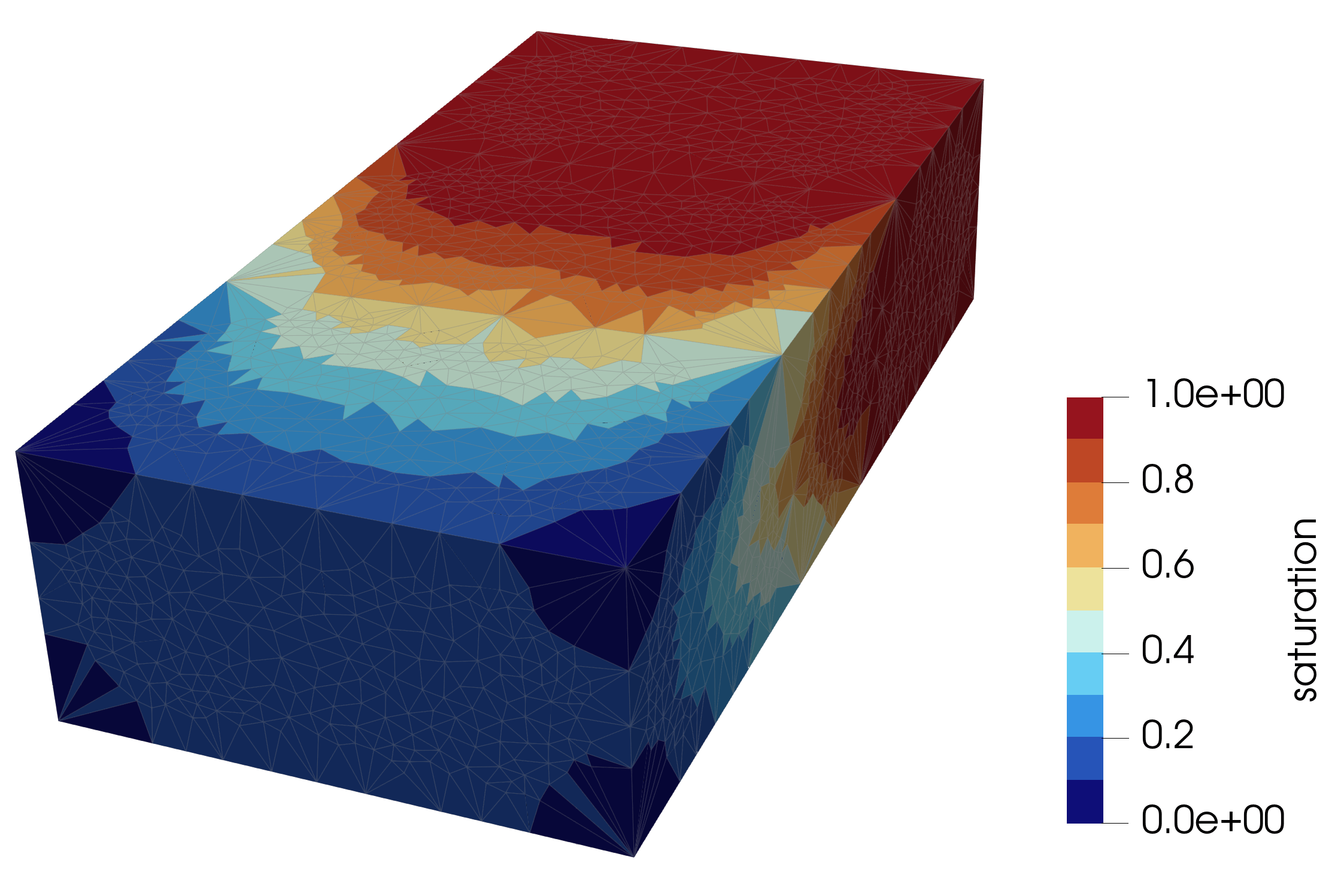}
    \\
    \small (a) Full TPFA
\end{minipage}
\hfill
\begin{minipage}{0.49\textwidth}
    \centering
    \includegraphics[width=\linewidth]{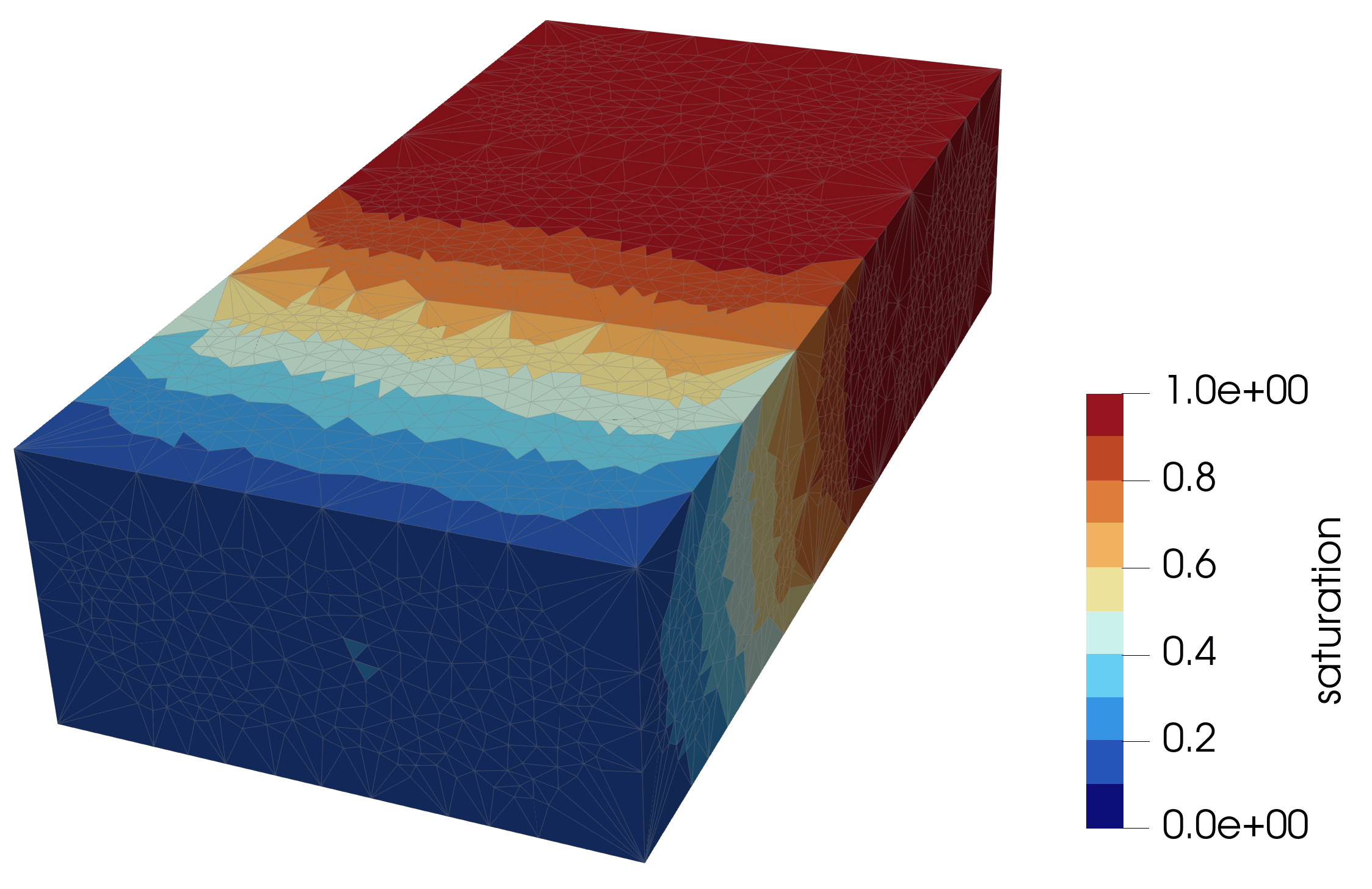}
    \\
    \small (b) $\tau = 10^{0}$
\end{minipage}

\vspace{0.5em}

\begin{minipage}{0.49\textwidth}
    \centering
    \includegraphics[width=\linewidth]{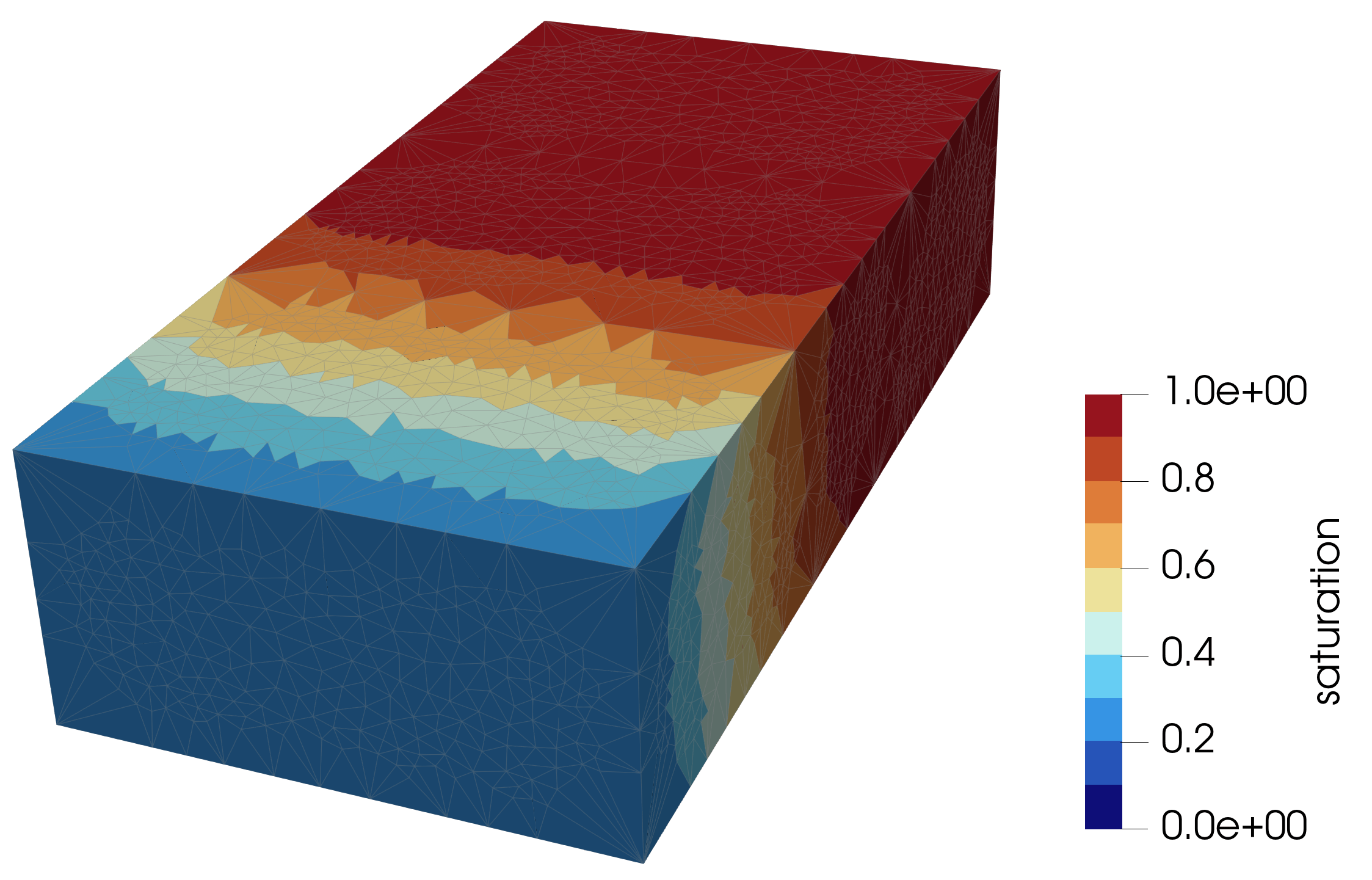}
    \\
    \small (c) $\tau = 10^{-1}$
\end{minipage}
\hfill
\begin{minipage}{0.49\textwidth}
    \centering
    \includegraphics[width=\linewidth]{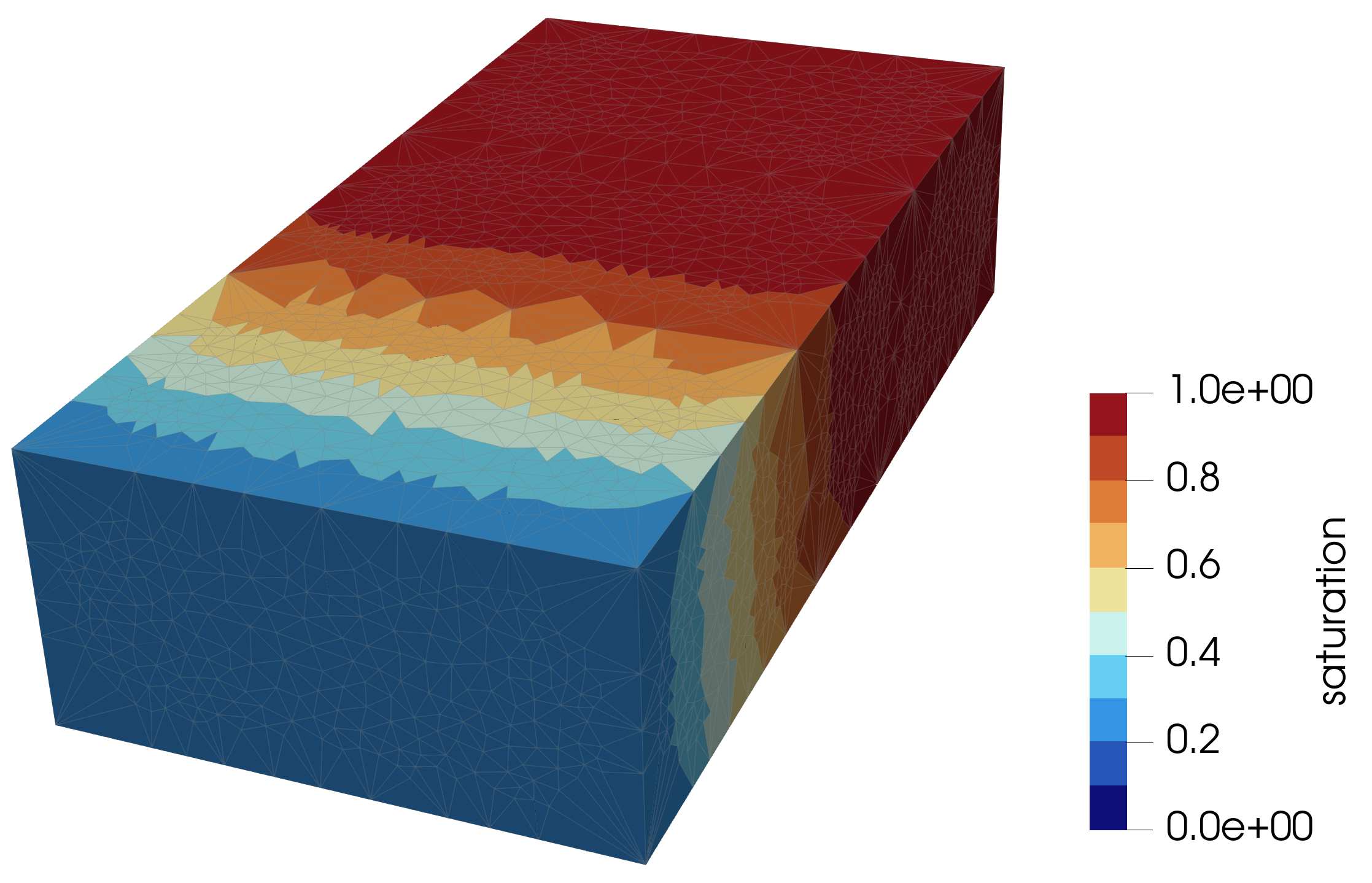}
    \\
    \small (d) $\tau = 10^{-2}$ (full MFD)
\end{minipage}

\caption{
Comparison of saturation fronts at $t = 1.00s$ for different adaptive tolerances in the fully polyhedral case. 
}
\label{fig:polyhedral_saturation_fronts}
\end{figure}

As expected, the full TPFA discretization produces the most distorted saturation front on the fully polyhedral mesh. The distortion remains clearly visible for $\tau = 10^{0}$, where the majority of cells are still classified as TPFA-compatible. By $\tau = 10^{-1}$, however, the adaptive solution is already visually very close to the full MFD reference with only minor discrepancies remaining in a few localized regions. This observation is particularly noteworthy because the adaptive partition has not yet converged to the full MFD discretization. The results therefore indicate that the cells most strongly affecting the transport solution are identified and switched to MFD mode before the remaining TPFA-compatible cells are eliminated.

This observation is particularly relevant for practical reservoir simulations, where regions with substantially different levels of TPFA consistency often coexist within the same computational domain. Under such conditions, the proposed adaptive framework is able to preserve TPFA-compatible regions while selectively activating the full MFD discretization only where it is required. The results presented across all benchmark problems indicate that this localized adaptation is sufficient to recover transport solutions that remain very close to those obtained with the full MFD discretization without incurring the cost of applying the full MFD stencil throughout the entire domain.


\subsection*{Reproducible Data and Figures}
All numerical results and figures presented in this work can be reproduced using the publicly available software repository and Docker environment provided in~\cite{docker}. The repository contains the complete source code, input files and post-processing scripts required to reproduce the convergence plots, error plots, cell classification and saturation field visualization and tabulated results reported in this paper. 

More specifically, the Docker environment reproduces the Local and Global Adaptation convergence studies presented in Section~4.1, including the associated convergence plots. For the benchmark problems reported in Section~4.2-4.4, the container generates the flux and saturation error plots together with the numerical quantities used throughout the paper, including relative flux and saturation errors, TPFA cell counts, sparsity reductions. In addition, adaptive cell classification and saturation field visualization are provided as VTK output files which can be viewed directly in ParaView to reproduce the corresponding 3D figures reported in the paper.

The Docker image provides a fully configured computational environment to facilitate reproducibility across different computing platforms. The repository and accompanying software artifacts are provided in accordance with FAIR principles for research software~\cite{Barker2022}.

\pagebreak

%% file: conclusion.tex
\section{Conclusion}
In this study, we developed an adaptive mimetic finite difference (MFD) framework anchored by a residual-based consistency indicator which serves to identify regions within a computational domain where the two-point flux approximation (TPFA) fails to maintain consistency due to geometric distortions. This work contributes significantly to the field of reservoir simulation through several key findings.

First, the proposed consistency indicator is rooted in the physics of the governing equations allowing for automated identification of distortion-sensitive regions without reliance on heuristic geometric metrics or manual tuning of discretization strategies. This methodology streamlines the process of determining appropriate discretization techniques across diverse reservoir configurations.

Second, rigorous numerical experiments validated the adaptive framework's ability to preserve solution accuracy on par with full MFD discretizations even in the presence of complex geometries. The adaptation mechanism ensures that mass conservation and consistency are satisfied, thereby delivering reliable flux computations essential for accurate reservoir modeling.

Moreover, the implementation of the adaptive framework enhances the global sparsity of the inner product matrix associated with the MFD method. By selectively activating MFD discretizations in regions that require higher accuracy while retaining TPFA in well-aligned areas, the computational costs are substantially reduced. This dynamic adaptability not only reduces the required computational resources but also results in numerical errors that can occasionally be smaller than those observed in full MFD solutions.

The introduction of a user-defined tolerance parameter, $\tau$, plays a crucial role in this adaptive approach. Both theoretical analysis and numerical results demonstrate that relative flux errors are explicitly controlled by this parameter allowing for a clear trade-off between solution fidelity and computational efficiency. As $\tau$ is adjusted, the framework provides a straightforward mechanism for adapting the discretization strategy based on user-defined accuracy requirements. In practical reservoir simulation workflows, the proposed indicator also provides a natural criterion for selecting $\tau$. Since the indicator is formulated as a normalized residual quantity, $\tau$ can be chosen at a level comparable to the residual tolerance of the underlying nonlinear Newton solver.

The implications of this work extend beyond these immediate findings. Future research will focus on developing robust preconditioning techniques tailored for the adaptive TPFA-MFD systems to enhance convergence rates in large-scale reservoir simulations. Additionally, further studies could explore the relationship between mesh characteristics, flow physics, and optimal choice of MFD inner products to refine the adaptability of the framework. Lastly, extending the proposed adaptive framework to more complex nonlinear reservoir simulation settings, including compositional flow models and coupled energy transport equations, remains as important direction for future investigation.

\subsection*{Acknowledgements}
This work was funded by TotalEnergies and Chevron through the FC-MAELSTROM Project. Portions of this work were performed under the auspices of the U.S. Department of Energy by Lawrence Livermore National Laboratory under Contract DE-AC52-07NA27344. LLNL-JRNL-2021274-DRAFT.